\documentclass[11pt]{article}

\usepackage[margin=1in]{geometry}

\usepackage{amsmath,amssymb,amsfonts}%
\usepackage{amsthm}%
\usepackage{mathrsfs}%
\usepackage{bbm}
\usepackage{mathtools}
\usepackage{MnSymbol}
\usepackage{dsfont}
\usepackage{mathrsfs}
\usepackage{xcolor}

\usepackage{natbib}
\setcitestyle{authoryear,open={(},close={)}}

\usepackage{bm}

\usepackage[hyperfootnotes=false]{hyperref}
\hypersetup{colorlinks,citecolor=blue,linkcolor=blue}

\theoremstyle{plain}
\newtheorem{theorem}{Theorem}
\newtheorem{proposition}[theorem]{Proposition}%
\newtheorem{lemma}[theorem]{Lemma}%
\newtheorem{corollary}[theorem]{Corollary}

\theoremstyle{remark}
\newtheorem{definition}{Definition}
\newtheorem{example}{Example}%
\newtheorem{remark}{Remark}%
\def\eqd{\stackrel{\mbox{\scriptsize{\textup{d}}}}{=}}
\def\simiid{\stackrel{\mbox{\scriptsize{\rm iid}}}{\sim}}

\newcommand{\ddr}{\mathrm{d}}
\newcommand{\crm}{\tilde \mu}
\newcommand{\CRM}{\textup{CRM}}
\newcommand{\E}{\mathbb{E}}
\newcommand{\W}{\mathcal{W}}
\newcommand{\R}{\mathbb{R}}

\newcommand{\N}{\mathbb{N}}
\newcommand{\M}{\mathcal{M}}
\newcommand{\law}{\mathcal{L}}

\newcommand{\s}{_\mathcal{S}}

\newcommand{\bl}{\mathrm{BL}}
\newcommand{\wbl}{\mathcal{W}_\mathrm{BL}}
\newcommand{\ad}{\mathrm{AD}}
\newcommand{\wad}{\mathcal{W}_{\mathrm{AD}}}
\newcommand{\Wt}{\mathcal{W}_{\dXt}}

\newcommand{\dXt}{d_{\mathbb{X}, \text{t}}}

\title{Merging Rate of Opinions via \\ Optimal Transport on Random Measures}
\author{Marta Catalano\footnote{Luiss University, Italy. Email: \texttt{mcatalano@luiss.it}}~~and Hugo Lavenant\footnote{Bocconi University, Italy. Email: \texttt{hugo.lavenant@unibocconi.it}}}

\date{\today}

\begin{document}

\maketitle

\begin{abstract}
Random measures provide flexible parameters for Bayesian nonparametric models. Given two different priors for a random measure, we develop a natural framework to investigate the rate at which the corresponding posteriors merge, as the sample size increases. We define a new distance between the laws of random measures that is built as a Wasserstein distance on the ground space of unbalanced measures, endowed with the bounded Lipschitz metric. We develop tight analytical bounds for its specification to completely random measures, including the special case of Poisson and gamma random measures. The bounds are interpreted in terms of an adapted extended Wasserstein distance between the L\'evy measures and are used to investigate the merging between the posteriors of normalized gamma and generalized gamma priors. After a careful study on the identifiability of the law of the random measure, interesting asymptotic and finite-sample insights are derived without putting \emph{any} assumption on the true data generating process.

\medskip

\end{abstract}

\section{Introduction}

In a Bayesian framework the distribution of the observations $(X_i)_{i\ge1}$ is usually modelled through a random parameter $\theta$. As the data arrives the prior distribution $\mathcal{L}(\theta)$ gets updated through conditional probability into $\mathcal{L}(\theta|X_1,\dots,X_n)$, known as the posterior distribution. A long-standing topic in Bayesian statistics is to understand and quantify the inferential impact of different priors on $\theta$. A natural way to proceed is to compute, for two different priors $\mathcal{L}(\theta^1)$ and $\mathcal{L}(\theta^2)$, a distance between their posteriors
\[
d( \mathcal{L}(\theta^1|X_1,\dots, X_n), \mathcal{L}(\theta^2|X_1,\dots, X_n)).
\] 
When the parameter space is finite-dimensional there are many well-studied proposals for a distance $d$. For example, assuming that $\theta$ is in a Euclidean space, \citet{Ley2017} have used the Wasserstein distance.  This clearly does not cover Bayesian nonparametric (BNP) models, which use infinite-dimensional parameter spaces to increase their flexibility and avoid restrictive assumptions on the distribution of the data. In this work we focus on BNP models whose infinite-dimensional random parameter consists of a random measure $\crm$. The data on a Polish space $\mathbb{X}$ are modelled as conditionally independent and identically distributed (i.i.d.) given $\crm$, so that for every $n \ge 1$,
\begin{equation}
\label{eq:model_intro}
X_1,\dots,X_n | \crm \simiid t(\crm),
\end{equation}
for some transformation $t$. A popular choice for $t$ is the normalization $t(\crm) = \crm/ \crm(\mathbb{X})$, which defines a random probability measure whenever $0<\crm(\mathbb{X}) <+\infty$ almost surely (\citeauthor{Regazzini2003}, \citeyear{Regazzini2003}, \citeauthor{James2006}, \citeyear{James2006}, \citeauthor{James2009}, \citeyear{James2009}). As a classical example, \citet{Ferguson1974} shows that the Dirichlet process is recovered by normalizing a gamma random measure. Other common choices include kernel mixtures for densities and hazard rates (\citeauthor{Lo1984}, \citeyear{Lo1984}, \citeauthor{DykstraLaud1981}, \citeyear{DykstraLaud1981}, \citeauthor{James2005}, \citeyear{James2005}), exponential transformations for survival functions \citep{Doksum1974} and cumulative transformations for cumulative hazards \citep{Hjort1990}, together with the extensive literature they generated.\\
Consider two different priors $\mathcal{L}(\crm^1)$ and $\mathcal{L}(\crm^2)$ for model \eqref{eq:model_intro}. Our goal is to introduce a distance between posterior random measures,
\begin{equation}
\label{eq:impact}
d(\mathcal{L}(\crm^1|X_1,\dots,X_n), \mathcal{L}(\crm^2|X_1,\dots,X_n)),
\end{equation}
and use it to investigate three compelling questions that do not have a general answer in a non-dominated setting:
\begin{itemize}
\item[a.] Given two different nonparametric priors, will the distance between the posteriors vanish as the number of observations goes to infinity, that is, will there be \emph{merging of opinions} \citep{BlackwellDubins1962}?
\item[b.] How fast does the distance vanish, that is, what is the \emph{merging rate}? 
\item[c.] Is the finite-sample distance strictly decreasing as a function of sample size, that is, does the \emph{merging start} immediately or after a certain number of observations?
\end{itemize}

To this end, we devote the first part of this work to defining a natural distance between generic laws of random measures by relying on optimal transport (OT). OT studies the mass transportation problem between probability measures, and it is conveniently used to define distances between measures, the so-called Wasserstein distances, in terms of the minimal cost of transporting mass from one measure to another; see  \cite{villani2009optimal} and \cite{Santambrogio2015} for exhaustive accounts. In principle, it defines a distance between probabilities on any Polish space. Our proposal is then to consider the Polish space of bounded measures endowed with the bounded Lipschitz (BL) distance, in a similar spirit to the Wasserstein over Wasserstein distance or Hierarchical Optimal Transport distance (\citeauthor{Nguyen2016}, \citeyear{Nguyen2016}, \citeauthor{yurochkin2019hierarchical}, \citeyear{yurochkin2019hierarchical},  \citeauthor{dukler2019wasserstein}, \citeyear{dukler2019wasserstein},  \citeauthor{CatalanoLavenant2024}, \citeyear{CatalanoLavenant2024}). We study the topology of our \emph{Wasserstein over BL distance}, showing that it metrizes the weak convergence over the weak convergence, together with the convergence of the total mass of the mean measure. The metrization of weak convergence is an extremely useful property in BNP (\citeauthor{Nguyen2013}, \citeyear{Nguyen2013}, \citeauthor{Nguyen2016}, \citeyear{Nguyen2016}, \citeauthor{Ho2017}, \citeyear{Ho2017}) since one usually deals with (random) measures with atomic components and possibly disjoint supports. 

The asymptotic analysis of the distance in \eqref{eq:impact} requires analytical upper and lower bounds. To this end, we focus our attention on a class of random measures that is usually used for models in \eqref{eq:model_intro}. All the proposals we mentioned use completely random measures (CRMs) to define a prior on $\crm$. CRMs are a rich class of random measures defined by \citet{Kingman1967}, which includes Poisson, gamma, and generalized gamma random measures as specific examples. We refer to \cite{LijoiPruenster2010} for an overview on CRMs as unifying thread of Bayesian nonparametric models. The law of a CRM without fixed atoms is characterized by a L\'evy intensity $\nu$ on $(0,+\infty) \times \mathbb{X}$, which is a deterministic measure with potentially infinite mass, whose expression is usually very explicit and tractable. One of the core results of our work is to find a general upper bound of the Wasserstein over BL distance between CRMs in terms of an \emph{adapted discrepancy} between the L\'evy measures. We build this adapted discrepancy using the notion of extended Wasserstein distance $\W_*$ of (\citeauthor{FigalliGigli2010}, \citeyear{FigalliGigli2010}, \citeauthor{Guillen2019}, \citeyear{Guillen2019},  \citeauthor{Catalano2024}, \citeyear{Catalano2024}), which remarkably allows for comparisons between measures with infinite mass, extending it to the multivariate space $(0,+\infty) \times \mathbb{X}$ in a similar spirit to the adapted (or nested) Wasserstein distance of \cite{pflug2012distance}, \cite{bartl2024wasserstein}, and \cite{backhoff2022estimating}. When restricted to random measures with the same mean measure, the adapted extended Wasserstein discrepancy provides a distance between L\'evy intensities that is both analytically and computationally tractable, and dominates the Wasserstein over BL: this distance will play the main role in the second part of the work, which focuses on the use of the distance in Bayesian nonparametrics. We also extend its definition to random measures characterized by a random L\'evy measure, such as Cox processes and Cox CRMs. 

Before moving on to the analysis of \eqref{eq:impact}, it is necessary to investigate the identifiability of the random parameter $\crm$ in \eqref{eq:model_intro}, that is, which random measures $\crm$ induce the same $t(\crm)$. 
This is crucial to avoid outputting a strictly positive distance even though the laws of the observation are the same, that is, to erroneously detect a lack of merging of opinions. Identifiability depends on the choice of transformation $t$ and the class of distributions for $\crm$; we focus on a specific setting, namely normalization of CRMs. The task 
amounts to understanding which random measures induce the same normalization, dividing the random measures into equivalence classes accordingly, and choosing a representative for each equivalence class. This will lead to the definition of scaled CRMs $\crm \s$, such that $\E(\crm \s(\mathbb{X})) = 1$, and their generalizations to Cox CRMs.

This sets the ground for a principled framework to study the merging rate of opinions in terms of the distance between (the representatives of) the posteriors, as in \eqref{eq:impact}. Relying on the almost-conjugacy of models \eqref{eq:model_intro} with respect to a CRM prior \citep{James2009}, we are able to shed new light on the use of the Dirichlet process and the normalized generalized gamma process, with unexpected results both from the asymptotic (Question a. and b.) and finite-sample (Question c.) point of view. \\
We first focus on the Dirichlet process \citep{Ferguson1973}, which has been a cornerstone in the development of BNP, allowing one to make inference on the law of the observations through closed-form expressions. We use our setup to investigate to which extent different parameters for the Dirichlet process will lead to different learning outcomes, both in a finite sample and asymptotic scenario, without putting \emph{any} assumption on the ``true" data-generating process. In this context, we prove that the merging in Wasserstein over BL will always occur at a rate faster than 1/$n$, where $n$ is the sample size, independently of the dimension of the space $\mathbb{X}$ and model's parameters. Moreover, we show that if the two Dirichlet processes have the same concentration parameter the merging in adapted extended Wasserstein distance will start immediately with the first observation, whereas if they differ the distance may first start increasing, even if the base probabilities $P_0$ coincide. From a statistical perspective this finding is quite remarkable: even if two Bayesians have the same prior guess $P_0$ of the distribution of the data, the difference between their opinions may increase if they see a finite amount of data, according to how much confidence they are willing to put on their prior guesses.
\\
Next we move to generalized gamma CRMs \citep{Brix1999}, a versatile class of CRMs that recovers many well-known CRMs as special or limiting cases. This is often invoked as a generalization of the Dirichlet process that adds a discount parameter $\sigma$ to increase the flexibility of the model. Its normalization has been introduced by \cite{Lijoi2007} and used in many statistical analyses (e.g., \citeauthor{DeBlasi2009}, \citeyear{DeBlasi2009}, \citeauthor{Barrios2013}, \citeyear{Barrios2013}, \citeauthor{GriffinLeisen2017}, \citeyear{GriffinLeisen2017}, \citeauthor{CaronFox2017}, \citeyear{CaronFox2017}, \citeauthor{LauCripps2022}, \citeyear{LauCripps2022}, \citeauthor{Lee2023}, \citeyear{Lee2023}). We use our setup to investigate its merging with the Dirichlet process, again without putting any assumption on the data-generating process. We prove that the merging rate occurs at the rate $\max(n^{-1/(1+\sigma)}, k/n)$, where $k$ is the number of distinct observations in the dataset. This result unravels the regimes where using a normalized generalized gamma CRM will actually guarantee a different learning outcome with respect to the Dirichlet process, with the following take-home message: either the number of observations $n$ is sufficiently small or the number of distinct values $k$ should be sufficiently large, and how large it must be depends on the choice of discount $\sigma$, namely $k \gtrsim n^{\frac{\sigma}{1+\sigma}}$.

\medskip

{\bfseries Relation to previous work.} This work is related to many different topics, including optimal transport on (completely) random measures, the merging of opinions, frequentist consistency and contraction rates, sensitivity and robustness analysis. \\

\noindent
\emph{Optimal transport between random measures.} 
The idea of endowing a space of measures with a distance and using it as a cost for an optimal transport distance on the laws of random measures is not new. This has been done between laws of random probability measures when the distance is an optimal transport distance, with recent applications in machine learning \citep{yurochkin2019hierarchical,dukler2019wasserstein} and Bayesian statistics (\citeauthor{Nguyen2016}, \citeyear{Nguyen2016},  \citeauthor{CatalanoLavenant2024}, \citeyear{CatalanoLavenant2024}).
In a different direction, \cite{huesmann2016optimal} extended this ``OT over OT'' problem to random measures that are equivariant, which implies that their law is invariant under translation.
This corresponds to measures with infinite mass (because the space is unbounded), and the average transport cost is used as a measure of distance. \cite{erbar2023optimal} then studied the gradient flows in the geometry generated by this optimal transport problem. There have also been proposals of distances tailored for laws of point processes; see, e.g., \citet{barbourbrown1992}, \citet{schuhmacher2008new}, \citet{decreusefond2016}, and references therein. However, we focus on general random measures, not only point processes, which forces us to consider also the distribution of the jumps of the random measures, and not only the one of their atoms.\\
\noindent
\emph{Optimal transport on CRMs.} The first proposals to define a distance between CRMs with tractable upper bounds can be found in \cite{Catalano2020, Catalano2021}, where the authors propose using 
\[
\sup_A \W(\mathcal{L}(\crm^1(A)), \mathcal{L}(\crm^2(A)))
\]
as a distance between CRMs, where $A$ is any Borel set and $\W$ is the Wasserstein distance on the Euclidean space. The key progress of \citet{Catalano2024} has been to define an OT distance directly at the level of the L\'evy intensities, as follows. For a set $A$, $\mathcal{L}(\crm(A))$ is an infinitely divisible law with L\'evy measure $\nu_A(\cdot) = \nu(\cdot, A)$. One can then use the extended Wasserstein distance $\W_*$ to define
\begin{equation}
\label{def:jasa}
\sup_A \W_*(\nu_A^1, \nu_A^2),
\end{equation}
which is a distance between CRMs \citep{Catalano2024} and leads to closed-form expressions whenever i) the CRMs do not have fixed atoms; ii) the CRMs are almost surely (a.s.) finite and homogeneous, i.e., $ \nu^i = \rho^i \otimes P^i$ is a product measure; iii) the CRMs share the same base measure, i.e., $P^1= P^2$. These conditions are usually met in the setting of \cite{Catalano2024} but are very restrictive in our framework since i) and ii) typically hold \emph{a priori} but not \emph{a posteriori}, and iii) entails that the atoms of the CRM have the same distribution, which is likely to be the first feature to change when choosing a different prior. Therefore, we need to define the adapted extended Wasserstein discrepancy, a new OT distance on L\'evy intensities with same mean measure that considers the joint distribution of the atoms and the jumps instead. Remarkably, it coincides with \eqref{def:jasa} whenever conditions i), ii), and iii) apply (see Remark~\ref{rem:wass_homogeneous} below).\\
\noindent
\emph{Merging of opinions.} The notion of merging of opinions was introduced and thoroughly studied in the acclaimed work of \citet{BlackwellDubins1962}, which holds for a very general setup beyond exchangeability. In principle, their results cover model \eqref{eq:model_intro}. However, in this setting, their assumptions can be considered restrictive for at least two reasons. First of all the merging is guaranteed for prior laws on $(X_n)_{n \ge 1}$ that are absolutely continuous with one another. Whereas this holds in most parametric settings, it is not usually the case in nonparametric ones, where non-dominated models are the norm. As a classical example, two Dirichlet processes with different total base measures are mutually singular \citep{KorwarHollander1973}. It has also been observed by \citet{DiaconisFreedman1986}, who avoid this condition by assessing the merging under a weaker topology. Moreover, both \citet{BlackwellDubins1962} and \citet{DiaconisFreedman1986} assume that the data-generating mechanism coincides with one of the models under consideration, which may well not be the case in practice. Our aim is to overcome both restrictions and build a framework that does not require \emph{any} assumption on the data. Indeed, we propose one of the few asymptotic analyses that does not put strong assumptions on the distribution of the data (e.g., to be a sampled from the model or i.i.d. from a true distribution): this is possible because we focus on the similarities between specific versions of the posteriors rather than on the adherence of the learning mechanism to the unknown \emph{truth}.\\
\noindent
\emph{Frequentist consistency and contraction rates.} In the same spirit, our framework goes beyond frequentist consistency properties in that it does not assume the data to be i.i.d. from a true distribution. However, if one were willing to make this assumption, consistency of the model with respect to each prior would entail merging of opinions but would not provide meaningful information about the convergence rate. Indeed the contraction rate of each prior to the truth could be (and usually is) considerably slower than the merging rate. For example, in this work we show that two Dirichlet processes with different parameters always merge at the rate 1/$n$, independently of the dimension of the space. This is much faster than the parametric rate of $1/\sqrt{n}$, which is the best possible one for the contraction rate, and moreover we expect the latter to deteriorate when the dimension increases. A merging rate that is faster than the contraction rate for a class of priors may suggest that it is the empirical part of the posterior - rather than the prior - that is slowing down convergence.
Thus, not only our framework has the advantage of not putting assumptions on the data-generating process, it also provides new perspectives when one is willing to put them.\\
\noindent
\emph{Sensitivity analysis and robustness.} To our knowledge, this work provides the first theoretical setup of \emph{global sensitivity} analysis in Bayesian nonparametric non-dominated models, in the sense that it studies the performance of the model of interest with respect to any choice of prior. Previous studies on sensitivity for BNP models usually consist in varying the parameters of the priors or taking functional perturbations \citep{Nieto-BarajasPruenster2009} and providing empirical comparisons between the posteriors. In a dominated setting, more classical distances between (posterior) densities can be considered, as in \citet{SahaKurtek2019}. A recent work of \citet{Giordano2022} focuses on \emph{local sensitivity}, which studies robustness with respect to small perturbation of the prior through differential approximations \citep{Gustafson1996} and can be very useful when posterior distributions are not known in closed form. Since local sensitivity focuses on small perturbations, a good impact is always a small impact. On the contrary, our framework does not limit itself to small perturbations, and thus a big impact of the prior can also be regarded as a benign sign of flexibility of the model. Although our statistical analysis focuses on the latter aspect, our results may be read or adapted to treat robustness as well. 

\medskip

{\bfseries Structure of the paper.} In Section~\ref{sec:distances} we define the Wasserstein over BL distance between the laws of random measures. We characterize its topology and provide a general lower bound, specializing it to Poisson random measures. In Section~\ref{sec:crm} we focus on CRMs and provide a universal upper bound in terms of the adapted extended Wasserstein discrepancy between their L\'evy intensities. The bound is specialized to Poisson random measures and extended to Cox CRMs with a random L\'evy intensity. In Section~\ref{sec:identifiability} we establish the class of (Cox) CRMs that are identifiable for model~\eqref{eq:model_intro} with respect to normalization, leading to the notion of scaled (Cox) CRMs. The adapted extended Wasserstein discrepancy is a distance on this class: in Section~\ref{sec:merging} we study its finite sample and asymptotic behavior to investigate the merging of opinions for BNP models. 
All proofs are deferred to the Supplementary Material \citep{CatalanoLavenant2025proofs}. Table~\ref{table:distances} summarizes the distances that we use in this work. 

\medskip

{\bfseries Notations for random quantities.} We denote deterministic scalar quantities with lowercase letters (e.g., $c,\alpha$) and real-valued random variables with capital letters (e.g., $C,U$). Random objects on spaces of measures are distinguished by a $\sim$, e.g., $\tilde{\nu}$ denotes a random measure while $\nu$ a deterministic measure.

\begin{table}
\begin{center}
\caption{Summary of the different distances we use in this work.}
\begin{tabular}{ |c|c|c|c| }
\hline 
Symbol & Name & Space & Introduced in \\
\hline
$\W_{d_\mathbb{X}}$ & {\small $1$-Wasserstein distance.} & $\mathcal{P}(\mathbb{X})$ & {\small Eq. \eqref{eq:definition_W1}}   \\
$\W_{\dXt}$ & {\small$1$-Wasserstein distance with metric $\dXt$.} & $\mathcal{P}(\mathbb{X})$ & {\small Eqs. \eqref{eq:definition_W1} and \eqref{eq:def_dxt}}  \\
$\W_*$ & {\small Extended Wasserstein distance.} & $\mathcal{M}((0, + \infty))$ & {\small Def. \ref{def:extended_wass}}   \\
$\bl$ & {\small Bounded Lipschitz distance.} & $\mathcal{M}_B(\mathbb{X})$ & {\small Eq. \eqref{def:bl}}  \\
$\wbl$ & {\small Wasserstein over bounded Lipschitz.} & $\mathcal{P}(\mathcal{M}_B(\mathbb{X}))$ & {\small Def. \ref{def:wbl}}  \\
$\ad$ & {\small Adapted extended Wasserstein discrepancy. } & $\mathcal{M}_m((0, + \infty) \times \mathbb{X})$ & {\small Def. \ref{def:distance_levy} and Prop. \ref{th:distance}}  \\
$\wad$ & {\small Wasserstein distance over $\ad$.} & $\mathcal{P}(\mathcal{M}_m((0, + \infty) \times \mathbb{X}))$ & {\small Eq. \eqref{eq:definition_W1} and Def. \ref{def:distance_levy}} \\
 \hline
\end{tabular}
\label{table:distances}
\end{center}
\end{table}
\section{Distance between measures and random measures}
\label{sec:distances}

We start by discussing distances between probability measures and non-negative measures, with classical notions such as the Wasserstein distance and the bounded Lipschitz distance. Next, we define a distance between laws of random measures by lifting the distance between random measures into a distance between probabilities over random measures via optimal transport.

\subsection{Distances between measures}

We work on a Polish space $\mathbb{X}$ endowed with a distance $d_\mathbb{X}$ and its Borel $\sigma$-algebra $\mathcal{B}(\mathbb{X})$. We write $\mathcal{P}(\mathbb{X})$ for the set of probability measures, $\M(\mathbb{X})$ for the space of locally finite measures, and $\mathcal{M}_B(\mathbb{X}) \subseteq \M(\mathbb{X})$ for the set of finite non-negative measures. The spaces $\mathcal{P}(\mathbb{X})$ and $\M_B(\mathbb{X})$ are endowed with the topology of weak convergence, making the maps $\mu \mapsto \int f \, \ddr \mu$ continuous for $f \in C_b(\mathbb{X})$ the space of continuous and bounded functions. It turns them into Polish spaces, and they are endowed with their Borel $\sigma$-algebra, which is the smallest one making the maps $\mu \mapsto \int f \, \ddr \mu$ measurable for $f \in C_b(\mathbb{X})$.  

We recall that on a metric space $(\mathbb{Y},d_\mathbb{Y})$, given $P^1,P^2$ two elements of $\mathcal{P}(\mathbb{Y})$, the optimal transport distance (of order $1$) between them is defined as 
\begin{equation}
\label{eq:definition_W1}
\W_{d_\mathbb{Y}}(P^1,P^2) = \min_{\pi \in \Pi(P^1,P^2)} \E_{(X,Y) \sim \pi} (d_\mathbb{Y}(X,Y)) = \sup_f \left\{ \int_\mathbb{Y} f \, \ddr P^2 - \int_\mathbb{Y} f \, \ddr P^1 \ : \  f \text{ is } 1 \text{-Lipschitz} \right\},
\end{equation}
where the minimum is taken over $\Pi(P^1,P^2)$ the set of joint probability distributions having marginals $P^1,P^2$, and the equality follows by Kantorovich duality. Here $1$-Lipschitz means $|f(x) - f(y)| \leq d_\mathbb{Y}(x,y)$ for all $x,y \in \mathbb{Y}$. In the sequel we will use this definition for $(\mathbb{Y},d_\mathbb{Y}) = (\mathbb{X},d_\mathbb{X})$, but also when $\mathbb{Y} = \M_B(\mathbb{X})$ or $\mathbb{Y} = \M((0,+\infty) \times \mathbb{X})$ is a space of measures. 

Restricting to the case $(\mathbb{Y},d_\mathbb{Y}) = (\mathbb{X},d_\mathbb{X})$, note that with $\W_{d_\mathbb{X}}$ we cannot compare measures of different mass, that is, $\W_{d_\mathbb{X}}(P^1,P^2)$ becomes infinite if $P^1(\mathbb{X}) \neq P^2(\mathbb{X})$. This can be seen by taking constant functions $f(x) =c$ and letting $c \to \pm \infty$. 
A natural generalization is the \emph{Bounded Lipschitz} distance $\bl$, which replaces the supremum over $1$-Lipschitz functions in \eqref{eq:definition_W1} by a supremum over $1$-bounded and $1$-Lipschitz functions. Specifically, we define 
\begin{equation}
\label{def:bl}
\bl(\mu^1,\mu^2) = \sup_f \left\{ \int_\mathbb{X} f \, \ddr \mu^2 - \int_\mathbb{X} f \, \ddr \mu^1 \ : \  f \text{ is } 1\text{-bounded and } 1 \text{-Lipschitz} \right\},
\end{equation}
where $1$-bounded means $|f(x)| \leq 1$ for all $x \in \mathbb{X}$. The $\bl$ distance metrizes the weak convergence on $\mathcal{M}_B(\mathbb{X})$ 
(see, e.g., Theorem 8.3.2 in \cite{bogachev2007measure}). We note that some works call $\bl$ the Kantorovich-Rubinstein or flat metric and rather define the function class of the Bounded Lipschitz distance to be $\{f: \|f\|_L + \|f\|_{\infty} \le 1\}$, where $\|\cdot\|_L$ is the Lipschitz norm. The induced distance, first defined in \cite{fortet1953convergence} to study the convergence of the empirical measure, is easily shown to be strongly equivalent to our definition (see, e.g., Section 8.3 of \cite{bogachev2007measure}).

Regarding the comparison between $\W_{d_\mathbb{X}}$ and $\bl$ on $\mathcal{P}(\mathbb{X})$, note that we always have $\bl \leq \W_{d_\mathbb{X}}$. However, the discrepancy can be large: $\bl$ is always bounded by $2$, while $\W_{d_\mathbb{X}}$ can be arbitrary very large if the first moments of the measures are large. 
To perform a finer comparison, we introduce the ``truncated'' distance
\begin{equation}
\label{eq:def_dxt}
\dXt  = \min(d_\mathbb{X},2). 
\end{equation}
This distance metrizes the same topology as $d_\mathbb{X}$ but is bounded. Moreover, every $1$-Lipschitz and $1$-bounded function is $1$-Lipschitz for $\dXt$, while every $1$-Lipschitz function for $\dXt$ is $1$-Lipschitz for $d_\mathbb{X}$ and $2$-bounded, so that $f/2$ is $1$-Lipschitz and $1$-bounded. Consequently, on $\mathcal{P}(\mathbb{X})$, 
\begin{equation}
\label{eq:comparison_BL_W}
\frac{1}{2} \W_{\dXt} \leq \bl \leq \W_{\dXt} \leq \W_{d_\mathbb{X}}.
\end{equation} 
Thus, $\W_{\dXt}$  and  $\bl$ are strongly equivalent metrics when restricted to $\mathcal{P}(\mathbb{X})$, but $\W_{\dXt}$ easier to interpret as it corresponds to a classical Wasserstein distance, though with respect to a 
truncated distance.

\subsection{Distances between random measures}

By a random measure $\tilde{\mu}$, we mean a measurable map from a probability space into $\mathcal{M}_B(\mathbb{X})$. As an example of random measure on $\mathbb{X}$, the reader could think of point processes such as Poisson random measures, or at completely random measures, which will be studied in the next section.
The law of a random measure is an element of $\mathcal{P}(\mathcal{M}_B(\mathbb{X}))$. The latter space is endowed with the topology of weak convergence when $\mathcal{M}_B(\mathbb{X})$ is endowed with the topology of weak convergence, and we say that a sequence of random measures $(\tilde{\mu}_n)_n$ converges in distribution to $\tilde{\mu}$ if the law of $\tilde{\mu}_n$ converges weakly to the one of $\tilde{\mu}$ as $n \to + \infty$. It is the case if and only if, for any $f \in C_b(\mathbb{X})$, the sequence of random variables $\int f \, \ddr \tilde{\mu}_n$ converges weakly to $\int f \, \ddr \tilde{\mu}$ as $n \to + \infty$ \cite[Theorem 4.19]{Kallenberg2017}. 

The definition in \eqref{eq:definition_W1} makes sense for any ground metric space $(\mathbb{Y},d_\mathbb{Y})$, and we will use it with $(\mathbb{Y},d_\mathbb{Y}) = (\M_B(\mathbb{X}),\bl)$. Indeed, interpreting laws of random measures as probability distributions over $\mathcal{M}_B(\mathbb{X})$, it is possible to lift the distance $\bl$ into a distance over the space $\mathcal{P}(\mathcal{M}_B(\mathbb{X}))$. We call the resulting distance the ``Wasserstein over Bounded Lipschitz'' distance, denoted by $\wbl$. 
If we were to look at laws of random probability measures, that is, elements of $\mathcal{P}(\mathcal{P}(\mathbb{X}))$, then an analog of the Wasserstein over Bounded Lipschitz distance could be the ``Wasserstein over Wasserstein'' distance, which has been proposed and studied in different contexts (\citeauthor{Nguyen2016}, \citeyear{Nguyen2016}, \citeauthor{yurochkin2019hierarchical}, \citeyear{yurochkin2019hierarchical},  \citeauthor{dukler2019wasserstein}, \citeyear{dukler2019wasserstein},  \citeauthor{CatalanoLavenant2024}, \citeyear{CatalanoLavenant2024}). We note that our Wasserstein over BL can be used on laws of random probabilities as well and requires weaker assumptions on first moments than the Wasserstein over Wasserstein distance.

\begin{definition}
\label{def:wbl}
If $\mathbb{Q}^1$, $\mathbb{Q}^2 \in \mathcal{P}(\mathcal{M}_B(\mathbb{X}))$ are two laws of random measures, we define the distance $\wbl$ between them as 
\begin{equation*}
\wbl(\mathbb{Q}^1,\mathbb{Q}^2) = \min_{\pi \in \Pi(\mathbb{Q}^1,\mathbb{Q}^2)}  \E_{(\tilde{\mu}^1,\tilde{\mu}^2) \sim \pi} (\bl(\tilde{\mu}^1,\tilde{\mu}^2)). 
\end{equation*}
\end{definition}

A random measure $\tilde{\mu}$ has finite mean if $\E(\tilde{\mu}(\mathbb{X})) < + \infty$. In this case its mean measure $\E(\tilde{\mu})$, defined by $\E(\tilde{\mu})(A) = \E(\tilde{\mu}(A))$, is an element of $\mathcal{M}_B(\mathbb{X})$.  

\begin{proposition}
\label{prop:topology_wbl}
The quantity $\wbl$ defines a complete distance over the laws of random measures with finite mean. If $(\tilde{\mu}_n)_n$, and $\tilde{\mu}$ are random measures with finite mean, then $\wbl(\law(\tilde{\mu}_n),\law(\tilde{\mu}))$ converges to $0$ as $n \to + \infty$ if and only if $(\tilde{\mu}_n)_n$ converges in distribution to $\tilde{\mu}$ and $\E(\tilde{\mu}_n(\mathbb{X}))$ converges to $\E(\tilde{\mu}(\mathbb{X}))$.     
\end{proposition}

We also provide a universal lower bound of this distance. Meaningful upper bounds are delayed until the next section, when we assume more structure on the random measures. 

\begin{proposition}
\label{prop:lower_bound_WBL}
For any random measures $\tilde{\mu}^1, \tilde{\mu}^2$ with finite mean, there holds
\begin{equation*}
\bl(\E(\tilde{\mu}^1), \E(\tilde{\mu}^2)) \leq \wbl(\mathcal{L}(\tilde{\mu}^1), \mathcal{L}(\tilde{\mu}^2)).   
\end{equation*}
\end{proposition}

\begin{example}
\label{ex:poisson_lower}
Let $\tilde{\mathcal{N}}^1$ and $\tilde{\mathcal{N}}^2$ be two Poisson random measures of intensities $\alpha_i P^i_0$, for $i=1,2$, where $\alpha_i>0$ and $P^i_0$ is a probability on $\mathbb{X}$. The lower bound of Proposition~\ref{prop:lower_bound_WBL} yields
\begin{equation*}
\bl(\alpha_1 P^1_0, \alpha_2 P^2_0) \leq \wbl(\mathcal{L}(\tilde{\mathcal{N}}^1), \mathcal{L}(\tilde{\mathcal{N}}^2)).   
\end{equation*}
In particular, taking $f=1$, we have $|\alpha_1 - \alpha_2| \le \wbl(\mathcal{L}(\tilde{\mathcal{N}}^1), \mathcal{L}(\tilde{\mathcal{N}}^2))$. On the other hand, if $\alpha_1 = \alpha_2 = \alpha$, then by~\eqref{eq:comparison_BL_W} we have $\frac{\alpha}{2} \W_{\dXt}(P^1_0,P^2_0)  \le \wbl(\mathcal{L}(\tilde{\mathcal{N}}^1), \mathcal{L}(\tilde{\mathcal{N}}^2))$. 
\end{example}

\section{Completely random measures and bounds via optimal transport on Lévy intensities}
\label{sec:crm}

In our applications to Bayesian statistics, the random measures are not arbitrary but rather possess some additional structure that will lead to tractable computations. In this section we focus on completely random measures (CRMs), a key ingredient to building many Bayesian nonparametric priors, and refer to \cite{LijoiPruenster2010} for an overview.  We derive a tractable upper bound on the $\wbl$ distance between CRMs and Cox CRMs. These upper bounds will be characterized in terms of an optimal transport distance on Lévy intensities, which is interesting per se and will be further investigated in the next sections. 

In the original definition of \citet{Kingman1967}, completely random measures are random measures such that, for every disjoint $A_1,\dots, A_n \in \mathcal{B}(\mathbb{X})$, the random variables $\crm(A_1),\dots,\crm(A_n)$ are mutually independent. This definition entails that any CRM can be decomposed in the sum of a deterministic measure, a discrete random measure with fixed atoms and a discrete random measure without fixed atoms, whose law is characterized by a Lévy intensity. Since it is the third component that plays the main role in Bayesian nonparametrics, we will give a slightly more restrictive definition of completely random measure as follows. 

Recall that $\M(\mathbb{Y})$ is the set of locally finite non-negative measures, possibly with infinite mass. Mostly we consider $\mathbb{Y} = (0, + \infty)$ or $\mathbb{Y} =(0,+\infty) \times \mathbb{X}$. Here and after, $\eqd$ denotes equality in distribution.

\begin{definition}
A random measure $\crm$ is a completely random measure (CRM) if there exists $\nu \in \M((0,+\infty) \times \mathbb{X})$ with $\iint \min(1,s) \, \ddr \nu(s,x) < + \infty$ such that
\begin{equation}
\label{eq:def_crm_from_levy}
\crm(A) \eqd \iint_{(0, + \infty) \times A} s \, \ddr \tilde{\mathcal{N}}(s,x),
\end{equation}
where $\tilde{\mathcal{N}}$ is a Poisson random measure on $(0,+\infty) \times \mathbb{X}$ of intensity $\nu$.
We write $\crm \sim \textup{CRM}(\nu)$ and refer to $\nu$ as the L\'evy intensity of $\crm$.
\end{definition}

Thanks to Campbell's theorem, a CRM $\crm$ has a finite mean $\E(\crm(\mathbb{X}))<+\infty$ if and only if $M_1(\nu) = \iint s \, \ddr \nu(s,x)  < + \infty$. Note that we do a slight abuse of notation, as $\nu$ is a measure on the product space $(0,+\infty) \times \mathbb{X}$, and by $M_1(\nu)$ we denote the first moment only with respect to the first coordinate.

\begin{proposition}
\label{prop:canonical_nu}
Let $\crm \sim \textup{CRM}(\nu)$ with finite mean. Then $\nu$ can be uniquely decomposed as $\ddr \nu(s,x) =  \ddr \rho_x(s) \ddr P_0(x)$, where $ P_0$ is a probability measure and $\rho$ is a transition kernel on $\mathbb{X} \times \mathcal{B}(\R_+)$ that satisfy, for $P_0$-almost every (a.e.) $x \in \mathbb{X}$,
\begin{equation*}
\int_0^{+ \infty} s \, \ddr \rho_x(s) = \E(\crm(\mathbb{X})); \qquad P_0(\cdot) = \frac{\E(\crm(\cdot))}{\E(\crm(\mathbb{X}))}.
\end{equation*}
We refer to $\ddr \nu(s,x) =  \ddr \rho_x(s) \ddr P_0(x)$ as the canonical decomposition of $\nu$.
\end{proposition}

\begin{remark} 
\label{rem:intuition}
Intuitively, the measure $P_0$ describes the distribution of atoms, while $\rho_x$ determines the conditional distribution of the jump size at location $x$. This can be understood by approximating \eqref{eq:def_crm_from_levy} with a compound distribution, so that for every $\epsilon>0$ and set $A$,
\begin{equation}
\label{eq:compound_poisson}
\iint_{(\epsilon, + \infty) \times A} s \, \ddr \tilde{\mathcal{N}}(s,x) \eqd \sum_{i=0}^N S_i \, \delta_{X_i},
\end{equation}
where by letting $r = \nu((\epsilon,+\infty) \times A)$, $N$ is a Poisson random variable with mean $r$, and $(S_i,X_i) \sim P$ independently, with $P =  \nu|_{(\epsilon,+\infty) \times A}/r$ being the normalized restriction of $\nu$ to $(\epsilon,+\infty) \times A$.
\end{remark}

We move on to the definition of a distance at the level of Lévy intensities that is
related to the distance $\wbl$ defined in the previous section. The underlying idea is simple: if we can find a coupling between the Lévy intensities $\nu^1, \nu^2$ and if $\crm^i \sim \textup{CRM}(\nu^i)$ for $i=1,2$, then we have a coupling between $\crm^1$ and $\crm^2$, which means we should be able to find an upper bound of $\wbl(\mathcal{L}(\crm^1),\mathcal{L}(\crm^2))$. The challenge is to understand how $\E(\bl(\crm^1,\crm^2))$ relates to a quantity that depends only on the Lévy intensities.

This leads to the definition of a discrepancy between Lévy intensities $\ddr \nu(s,x) =  \ddr \rho_x(s) \ddr P_0(x)$, which is ultimately justified by Theorem~\ref{theo:upper_bound_WBL} below. We observe that the two components $\rho_x$ and $P_0$ play very different roles, as highlighted in Remark~\ref{rem:intuition}. Indeed,  $P_0$ is a probability measure on $\mathbb{X}$ describing the law of the atoms, whereas $\rho_x$ is $P_0$-a.s. a measure with infinite mass on $(0,+\infty)$ that determines the conditional law of their corresponding jumps. We will thus preserve their different nature, in the spirit of the \emph{nested Wasserstein distance} for probability measures, which was originally introduced by \citet{pflug2012distance} and is often termed \emph{adapted Wasserstein distance} in financial mathematics (\citeauthor{backhoff2022estimating}, \citeyear{backhoff2022estimating}, \citeauthor{bartl2024wasserstein}, \citeyear{bartl2024wasserstein}).
Specifically, we impose that if an atom $x$ is coupled to an atom $y$ on the space $\mathbb{X}$, then the whole measure $\rho_x$ is also coupled to $\rho_y$.

A key tool is the definition of an ``extended'' Wasserstein distance to compare measures of infinite mass $\rho_x$ and $\rho_y$. A (quadratic) Wasserstein distance between measures of different and potentially infinite masses has been proposed by \citet{FigalliGigli2010} through the notion of extended coupling. We refer to \citet{Catalano2024} and \citet{Guillen2019} for a concise overview. In the present work, we will only consider measures on the half-real line $\mathbb{Y} = (0,+\infty)$, and look at Wasserstein distances of order $1$. We refer to \citet[Proposition~4]{Catalano2024} for a link between the definition below and the previous works of \cite{FigalliGigli2010} and \cite{Guillen2019}. 

\begin{definition}
\label{def:extended_wass}
The extended Wasserstein distance (of order $1$) between two positive Borel measures $\rho^1, \rho^2$ on $(0,+\infty)$ with finite  first moments about $0$ is defined as
\begin{equation}
\W_{*}(\rho^1, \rho^2) = \int_0^{+\infty} |U_1(s) - U_2(s)| \ddr s,
\end{equation}
where, for $i=1,2$ and $s > 0$, $U_i(s) = \rho^i (s, +\infty)$ is the tail integral of $\rho^i$.
\end{definition}

Following the aforementioned literature, and using the extended Wasserstein distance, we now present what we call the ``adapted extended Wasserstein discrepancy''.

\begin{definition}
\label{def:distance_levy}
Let $\nu^1, \nu^2$ be Lévy intensities with $M_1(\nu^i ) = \iint s \, \ddr \nu^i(s,x) <+\infty$ and canonical decomposition $\ddr \nu^i(s,x) =  \ddr \rho^i_x(s) \ddr P^i_0(x)$ for $i=1,2$. We define the adapted extended Wasserstein discrepancy between $\nu^1$ and $\nu^2$ as
\[
\ad(\nu^1,\nu^2) =  \inf_{\pi \in \Pi(P_0^1, P_0^2)} \; \E_{(X,Y) \sim \pi} \left( \frac{M_1(\nu^1) + M_1(\nu^2)}{2} \dXt (X,Y) +  \W_{*}(\rho_X^1,\rho_Y^2) \right).
\]
\end{definition}

The presence of the first moments $M_1(\nu^1)$ and $M_1(\nu^2)$ may seem counterintuitive, and prevents $\ad$ from being a distance. It is however necessary in order for the next theorem to hold. Note also that we use the truncated distance $\dXt$ instead of the original distance $d_\mathbb{X}$. This is again because of the next theorem as it gives a sharper upper bound; see also~\eqref{eq:comparison_BL_W} and the discussion related to it to understand why $\dXt$ is a natural distance to consider. Thanks to the use of this truncated distance, note that  
we do not need to require the finiteness of any moment of $P_0^i$ in the definition of $\ad$, for $i=1,2$. Moreover,  $\int_0^{+ \infty} U_{\rho}(s) \, \ddr s = \int_0^{+ \infty} s \rho(s) \ddr s$ by Fubini's theorem. Thus, $\W_*(\rho^1_X, \rho^2_X) \leq M_1(\nu^1) + M_2(\nu^2)$, which together with the boundedness of $\dXt$ yields
\begin{equation}
\label{eq:bound_ad_universal}
\ad(\nu^1,\nu^2) \leq 2(M_1(\nu^1) + M_1(\nu^2)).
\end{equation}
We move to a key theorem justifying the introduction of this adapted extended Wasserstein discrepancy.

\begin{theorem}
\label{theo:upper_bound_WBL}
If $\crm^i \sim \textup{CRM}(\nu^i)$ for $i=1,2$ have finite mean, then
\begin{equation*}
\wbl(\mathcal{L}(\crm^1),\mathcal{L}(\crm^2)) \leq \ad(\nu^1,\nu^2).
\end{equation*}
\end{theorem}

\begin{remark}
\label{rem:wass_homogeneous}
The adapted extended Wasserstein discrepancy $\ad$ is particularly convenient when at least one of the two L\'evy intensities is homogeneous, i.e., $\nu^1 = \rho^1 \otimes P_0^1$ or, equivalently, $\rho_X^1$ does not depend on $X$ a.s., with $X\sim P_0$. Many of the most commonly used CRMs fall into this framework, such as Gamma CRMs, with $\rho(s) = \alpha s^{-1} \exp(-bs)$, or generalized gamma CRMs, with  $\rho(s) = \alpha s^{-(1+\sigma)} \Gamma(1-\sigma)^{-1} \exp(-bs)$, where $\alpha,b>0$ and $0 \le \sigma < 1$ are parameters.
In such case, $\W_{*}(\rho_X^1,\rho_Y^2)^p = \W_{*}(\rho^1,\rho_Y^2)^p$ does not depend on the coupling and thus
\[
\ad(\nu^1, \nu^2) = \frac{M_1(\nu^1) + M_1(\nu^2)}{2} \W_{\dXt}(P_0^1, P_0^2) + \E_{Y \sim P_0^2}(\W_{*}(\rho^1,\rho_Y^2)).
\]
When both L\'evy intensities are homogeneous with the same base probability it further simplifies to 
\[
\ad(\nu^1, \nu^2) = \W_{*}(\rho^1,\rho^2).
\]
\end{remark}

\begin{example}
\label{ex:poisson_upper}
Let $\tilde{\mathcal{N}}^1$ and $\tilde{\mathcal{N}}^2$ be two Poisson random measures of intensities $\alpha_i P^i_0$, for $i=1,2$, where $\alpha_i>0$ and $P^i_0$ is a probability on $\mathbb{X}$. Then $\tilde{\mathcal{N}}^i$ is a CRM with Lévy intensity $\ddr \nu^i = \alpha_i \delta_{1}(\ddr s) \ddr P^i_0(x)$. In particular they are homogeneous CRMs with $\W_*(\rho^1,\rho^2) = |\alpha_1 - \alpha_2|$. Thus, putting together Theorem~\ref{theo:upper_bound_WBL} and Example~\ref{ex:poisson_lower}, 
\begin{equation*}
\bl(\alpha_1 P^1_0,\alpha_2 P^2_0) \le \wbl(\mathcal{L}(\tilde{\mathcal{N}}^1),\mathcal{L}(\tilde{\mathcal{N}}^2)) \leq \frac{\alpha_1 + \alpha_2}{2} \W_{\dXt}(P^1_0,P^2_0) + |\alpha_1 - \alpha_2|.
\end{equation*}
If $P_0^1 = P_0^2$ then from Example~\ref{ex:poisson_lower} we see that the lower and upper bound coincide, so that 
\begin{equation*}
\wbl(\mathcal{L}(\tilde{\mathcal{N}}^1),\mathcal{L}(\tilde{\mathcal{N}}^2)) = |\alpha_1 - \alpha_2|.
\end{equation*}
On the other hand, if $\alpha_1 = \alpha_2 = \alpha$, then again from Example~\ref{ex:poisson_lower} we deduce 
\begin{equation*}
\frac{\alpha}{2} \W_{\dXt}(P^1_0,P^2_0) \le \wbl(\mathcal{L}(\tilde{\mathcal{N}}^1),\mathcal{L}(\tilde{\mathcal{N}}^2)) \leq \alpha \W_{\dXt}(P^1_0,P^2_0),
\end{equation*}
which characterizes the value of $\wbl(\mathcal{L}(\tilde{\mathcal{N}}^1),\mathcal{L}(\tilde{\mathcal{N}}^2))$ up to a constant multiplicative factor.
\end{example}

\cite{schuhmacher2008new} have defined a distance on the laws of point processes with a finite number of points by defining a suitable transport problem tailored for point processes. Interestingly, in \citet[Remark 4.2]{schuhmacher2008new} the authors also provide an upper bound of their distance evaluated on Poisson random measures in terms of the Wasserstein distance between the normalized intensities. Direct comparisons are not meaningful, as their distance and our distance $\W_\bl$ differ, but it illustrates the need of relating a distance between Poisson random measures to a distance between their intensities, as the latter is often more tractable and interpretable.

Although $\ad$ is not a distance, because of the presence of $M_1(\nu^1) + M_1(\nu^2)$ in the expression, it becomes a distance if we freeze the value of $M_1(\nu)$. Although it may seem like an algebraic trick, in the next section we will see why it is natural to consider the space of Lévy intensities $\nu$ with a given value for $M_1(\nu)$.

\begin{proposition}
\label{th:distance}
For any $m>0$, $\ad$ defines a distance on the space 
\[
\M_{m}((0,+\infty) \times \mathbb{X}) = \left\{ \nu \in \M((0,+\infty) \times \mathbb{X})  \ : \ M_1(\nu) = m  \right\}.
\]
\end{proposition}

Eventually, for posterior distributions in our Bayesian models, we are led to consider mixtures of CRMs, that is, CRMs with a random Lévy intensity. We call them ``Cox CRM'', by analogy with Cox Poisson point processes \citep{cox1955some}. 

\begin{definition}
\label{def:cox_CRM}
A random measure $\crm$ on $\mathbb{X}$ is a Cox CRM or a doubly stochastic CRM if there exists a random element $\tilde \nu$ valued in $\M((0,+\infty) \times \mathbb{X})$ such that $\crm| \tilde \nu \sim \textup{CRM} (\tilde \nu)$.
\end{definition}

\noindent Note that any $\crm \sim \mathrm{CRM}(\nu)$ is technically a Cox CRM, by taking $\tilde{\nu}$ a.s. constant and equal to $\nu$. Moreover, we observe that $\tilde \nu$ is not a random measure according to the definition of Section~\ref{sec:crm} because it can have infinite mass almost surely. Thanks to the convexity of the Wasserstein distance, it is straightforward to extend the upper bound in Theorem~\ref{theo:upper_bound_WBL}. 

\begin{theorem}
\label{th:wass_over_ad}
Let $\crm^i$ for $i=1,2$ be Cox CRM with random Lévy intensities $\tilde{\nu}^1$ and $\tilde{\nu}^2$ respectively. If $(\tilde{\nu}^1, \tilde{\nu}^2)$ is any coupling between these random Lévy intensities then 
\[
\wbl(\mathcal{L}(\crm^1),\mathcal{L}(\crm^2)) \leq \E(\ad (\tilde{\nu}^1, \tilde{\nu}^2)).
\]
In particular, if we restrict our attention to $\tilde{\nu}^1, \tilde{\nu}^2 \in \M_m((0,+\infty) \times \mathbb{X})$ for some fixed m>0, with $\wad$ the Wasserstein distance over the $\ad$ metric in Definition~\ref{def:distance_levy}, we have
\begin{equation}
\label{def:wad}
\wbl(\mathcal{L}(\crm^1),\mathcal{L}(\crm^2)) \leq \wad(\mathcal{L}(\tilde{\nu}^1), \mathcal{L}(\tilde{\nu}^2)).
\end{equation}
\end{theorem}

\section{Normalization, identifiability, and scaling}
\label{sec:identifiability}

One of the most popular uses of completely random measures is through their normalization $t(\crm) = \crm/\crm(\mathbb{X})$, first defined in \cite{Regazzini2003}. The distribution of a sequence  $(X_n)_{n \ge 1}$ of exchangeable observations on $\mathbb{X}$ can be flexibly modelled as 
\begin{equation}
\label{eq:model}
X_1,\dots,X_n | \crm \simiid \frac{\crm}{\crm(\mathbb{X})}; \qquad \crm \sim \CRM(\nu),
\end{equation}
for any $\crm$ that takes values on the space of finite measures $\mathcal{M}_B(\mathbb{X})$ and such that $\crm(\mathbb{X}) > 0$ a.s.. In this case, \cite{Regazzini2003} show that $\crm/\crm(\mathbb{X})$ is a well-defined random probability measure, while \cite{James2006, James2009} find closed-form expressions for the predictive and posterior distribution.\\
Our goal for the remainder of this work is to use a distance between the laws of random measures $\crm$ to perform comparisons between the induced Bayesian models \eqref{eq:model}, both a priori and a posteriori. To this end, in this section we investigate which CRMs induce the same normalization in law, a.k.a., the identifiability of the law of a CRM. Theorem~\ref{th:identifiability} finds all \emph{equivalent} CRMs, in the sense that they induce the same normalization. The choice of a representative for each equivalence class leads us to the notion of scaled CRM. The analysis is then extended to Cox CRMs, which naturally emerge in the study of the posterior, as will be clear in Section~\ref{sec:merging}.

\begin{remark}For the normalization $\crm/\crm(\mathbb{X})$ to be well-defined one needs $\crm(\mathbb{X}) > 0$ a.s., which at the level of the Lévy intensities corresponds to $\nu((0,+\infty) \times \mathbb{X}) = + \infty$ \citep{Regazzini2003}. We will require a slightly stronger condition, namely that the CRM is \emph{infinitely active}. In terms of the canonical decomposition of Proposition~\ref{prop:canonical_nu}, this assumption reads 
\begin{equation*}
\rho_X(0,+\infty) = + \infty.
\end{equation*}
almost surely, where $X \sim P_0$.
This means that there will always be an infinite number of jumps in a set $A$ with $P_0(A) > 0$, and thus the compound Poisson approximation \eqref{eq:compound_poisson} cannot be extended to $\epsilon = 0$. In particular, it guarantees that for any set $A$ exactly one of the two alternatives holds: either $P_0(A) > 0$, and in this case $\crm(A) > 0$ a.s., or $P_0(A) = 0$, which implies $\crm(A) = 0$ a.s. This will be useful for our identifiability results.
\end{remark}

\subsection{Completely random measures}

Quite clearly, different CRMs $\crm$ can lead to the same normalized random measure $\crm/\crm(\mathbb{X})$: for example, $\crm$ and $\alpha \crm$ lead to the same normalization for any $\alpha > 0$. In Theorem~\ref{th:identifiability} below we build on a result on subordinators by David Aldous and Stevan Evans \cite[Lemma 7.5]{PitmanYor1992} to prove that this can be the only source for the lack of identifiability. 

\begin{theorem}
\label{th:identifiability}
Let $\crm^i \sim \textup{CRM}(\nu^i)$ be infinitely active with finite mean, for $i=1,2$. Assume that $\crm^i$ is non-degenerate, in the sense $P^i_0$ is not a Dirac mass, for $i=1,2$. Then,
\[
\frac{\crm^1}{\crm^1(\mathbb{X})} \eqd \frac{\crm^2}{\crm^2(\mathbb{X})} \qquad \text{   if and only if   } \qquad \crm^1 \eqd \alpha \crm^2 \text{ for some scalar } \alpha > 0.
\]
\end{theorem}

The assumption of non-degeneracy is only here to exclude the case $\crm^i = J_i \delta_x$ for $i=1,2$ with $x$ deterministic and $J_1,J_2$ random. In fact, in this case $\crm^1 /\crm^1(\mathbb{X}) = \crm^2 /\crm^2(\mathbb{X}) = \delta_x$ a.s., even though it does not hold $\crm^1 \eqd \alpha \crm^2$ for some $\alpha > 0$.
Theorem~\ref{th:identifiability} sheds new light on the equivalence relation $\sim$ that identifies CRMs leading to the same normalization, showing that $\crm^1 \sim \crm^2$ if and only if $\crm^1 \eqd \alpha \crm^2$ for some $\alpha >0$. We use it to choose one representative for each equivalence class, namely the only CRM in the class with total expected value equal to 1, as in the next definition.

\begin{definition}
\label{def:scaled}
A CRM $\crm$ is a scaled CRM if $\E(\crm(\mathbb{X})) = 1$. If $\crm$ is a CRM with finite mean, we define its scaled CRM as
\[
\crm \s = \frac{\crm}{\mathbb{E}(\crm(\mathbb{X}))}.
\]
\end{definition}

\begin{corollary}
\label{th:identifiability_cor}
Let $\crm^i \sim \textup{CRM}(\nu^i)$ be infinitely active with finite mean, for $i=1,2$. Then, 
\[
\frac{\crm^1}{\crm^1(\mathbb{X})} \eqd \frac{\crm^2}{\crm^2(\mathbb{X})} \qquad \text{if and only if} \qquad  \crm \s ^1 \eqd \crm \s^2.
\]
\end{corollary}

By Proposition~\ref{prop:canonical_nu} the mean measure $Q(\cdot) = \E(\crm(\cdot))$ of any scaled CRM with L\'evy intensity $\ddr \nu(s,x) =  \ddr \rho_x(s) \ddr P_0(x)$ is the probability measure $P_0$ and
\begin{equation}
\label{eq:normalization_rho}
\int_0^{+ \infty} s \, \ddr \rho_X(s) = 1.
\end{equation}
almost surely, where $X \sim P_0$. We highlight that, unlike normalized CRMs, scaled CRMs are CRMs. The following result provides a general way to find their L\'evy intensities. We state it here for diffuse jumps components $\ddr \rho_x(s) = \rho_x(s) \, \ddr s$ for simplicity, though it can be extended to atomic jump components; see Lemma SM3 in the Supplementary Material \citep{CatalanoLavenant2025proofs}.

\begin{lemma}
\label{th:levy_scaled}
Let $\crm\sim\textup{CRM}(\nu)$ such that $\ddr \nu(s,x) =  \rho_x(s) \, \ddr s \, \ddr P_0(x)$. Then $\crm \s$ has L\'evy intensity 
\[
\ddr \nu \s (s,x) = \mathbb{E}(\crm(\mathbb{X})) \rho_x ( \mathbb{E}(\crm(\mathbb{X})) s) \,  \ddr s \, \ddr P_0(x).
\]
\end{lemma}

\begin{example}[Gamma CRM]
\label{ex:gammaCRM}
A gamma CRM with base probability $P_0 \in \mathcal{P}(\mathbb{X})$, total base measure $\alpha>0$ and scale parameter $b>0$ has L\'evy intensity $\ddr \nu(s, x) = \alpha e^{-bs}s^{-1} \ddr s \,\ddr P_0(x)$.
Thanks to Campbell's theorem one easily shows that $\mathbb{E}(\crm(\mathbb{X})) = \alpha/b$, so that by Lemma~\ref{th:levy_scaled} the corresponding scaled gamma CRM has L\'evy intensity 
\begin{equation}
\label{eq:scaledgamma}
\ddr \nu \s (s, x) = \alpha \frac{e^{- \alpha s}}{s} \ddr s \,\ddr P_0(x).
\end{equation}
Thus, $\crm \s$ is a gamma CRM with base probability $P_0$, total base measure $\alpha$ and scale parameter $\alpha$. This highlights the redundancy of the parameter $b>0$, which will be assumed to be equal to 1 in the remainder of the paper, unless otherwise specified. 
\end{example}

\begin{example}[Generalized gamma CRM]
\label{ex:gengamma}
Another popular prior in BNP models is the normalized generalized gamma \citep{Lijoi2007, Barrios2013}. The generalized gamma CRM with total base measure $\alpha>0$, discount parameter $\sigma \in (0,1)$, and base probability $P_0 \in \mathcal{P}(\mathbb{X})$ has L\'evy intensity 
\[
\ddr \nu(s,x) = \frac{\alpha}{\Gamma(1-\sigma)} \frac{e^{-s}}{s^{1+\sigma}} \ddr s \, \ddr P_0(x),
\]
where $\Gamma(z) = \int_0^{+ \infty} x^{z-1} e^{-x} \, \ddr x$ is the gamma function. In particular, the gamma CRM is recovered when $\sigma = 0$ and the inverse Gaussian CRM \citep{Lijoi2005} when $\sigma =1/2$. It is easy to prove  that $\mathbb{E}(\crm(\mathbb{X})) = \alpha$, so that by Lemma~\ref{th:levy_scaled} the corresponding scaled CRM has L\'evy intensity
\begin{equation}
\label{eq:gengamma}
\ddr \nu \s(s,x) = \frac{ \alpha^{1-\sigma}}{\Gamma(1-\sigma)} \frac{e^{-\alpha s}}{s^{1+\sigma}} \ddr s \, \ddr P_0(x).
\end{equation}
Thus, we can regard $\crm \s$ as a generalized gamma CRM that also accounts for a scale parameter.
\end{example}

When comparing normalized CRMs we should thus restrict our attention to their scaled version. Importantly, Proposition~\ref{th:distance} and Theorem~\ref{theo:upper_bound_WBL} guarantee that 
\begin{equation}
\label{eq:adapted_distance}
\ad(\nu^1 \s,\nu^2 \s) =  \inf_{\pi \in \Pi(P_0^1, P_0^2)} \; \E_{(X,Y) \sim \pi} \left( \dXt (X,Y) +  \W_{*}(\rho_X^1,\rho_Y^2) \right)
\end{equation} 
is a distance on Lévy intensities of scaled CRMs that satisfies $\wbl(\mathcal{L}(\crm^1 \s),\mathcal{L}(\crm^2 \s)) \leq \ad(\nu^1 \s,\nu^2 \s)$. Since Lévy intensities uniquely characterize the law of 
CRMs, the adapted extended Wasserstein distance 
$\ad$ can be considered as a distance 
on CRMs: arguably a less canonical but more tractable one.

\subsection{Cox completely random measures}

We now extend the notion of scaling and identifiability to Cox completely random measures.

\begin{definition} 
\label{def:cox_scaled}
Let $\crm$ be a Cox CRM such that $\crm| \tilde \nu \sim \textup{CRM} (\tilde \nu)$ has a.s. finite mean and infinite activity. Then we define its Cox scaled CRM $\crm \s$ as
\[
\crm \s | \tilde \nu = \frac{\crm}{\E(\crm(\mathbb{X})| \tilde \nu)}.
\]
\end{definition}

\begin{theorem}
\label{th:identifiability_cox}
Let $\crm^i$ be a Cox CRM such that $\crm^i| \tilde \nu^i \sim \textup{CRM} (\tilde \nu^i)$ is a.s. an infinitely active CRM with finite mean, for $i=1,2$. Denoting $\ddr \tilde{\nu}^i(s,x) = \ddr \tilde{\rho}^i(s) \, \ddr \tilde{P}^i_0(x)$, we 
assume that i) $\E(\tilde{P}^i_0)$ is not purely atomic, (ii) for any Borel set $A$, either $\tilde{P}^i_0(A) > 0$ a.s. or $\tilde{P}^i_0(A) = 0$ a.s.. Then,
\[
\frac{\crm^1}{\crm^1(\mathbb{X})} \eqd \frac{\crm^2}{\crm^2(\mathbb{X})} \qquad \text{if and only if} \qquad  \crm \s ^1 \eqd \crm \s^2.
\]
\end{theorem}

Following Proposition~\ref{prop:canonical_nu}, we see that if $\crm | \tilde{\nu} \sim \mathrm{CRM}(\tilde{\nu})$, then it is a Cox scaled CRM if $\tilde{\nu} \in \M_{1}((0,+\infty) \times \mathbb{X})$ a.s. We denote by $\tilde \nu \s$ the random L\'evy measure of the Cox scaled CRM, also referred to as random scaled L\'evy measure.\\

\begin{remark}
Theorem~\ref{th:identifiability_cox} is the analog for Cox CRMs of Corollary~\ref{th:identifiability_cor} for CRMs. However, it is worth noting that the latter does not follow directly from Theorem~\ref{th:identifiability_cox}. Indeed, when dealing with CRMs we do not need the extra assumption that $P_0$ is not purely atomic.
\end{remark}

By reasoning as for CRMs, thanks to Theorem~\ref{th:identifiability_cox} when comparing normalized Cox CRMs we should restrict our attention to their scaled version. Theorem~\ref{th:wass_over_ad} ensures that $\wad(\mathcal{L}(\tilde \nu^1 \s), \mathcal{L}(\tilde \nu^2 \s))$ is a distance on random scaled L\'evy measures satisfying $\wbl(\mathcal{L}(\crm^1 \s),\mathcal{L}(\crm^2 \s)) \leq \wad(\mathcal{L}(\tilde \nu^1 \s), \mathcal{L}(\tilde \nu^2 \s))$.

\section{Merging rate of Bayesian nonparametric models}
\label{sec:merging}

In this section we use our distance to measure the impact of the prior for Bayesian nonparametric models of the form \eqref{eq:model}. Specifically, we consider the distance between the posteriors corresponding to two different priors and (i) we determine if the distance vanishes as the number of observations increases, (ii) if it vanishes, we analyze its rate of convergence to 0. This is what we call the merging rate of opinions. In Section~\ref{sec:dp} we use this general framework to study the impact of the parameters of the Dirichlet processes, whereas in Section~\ref{sec:gengamma} we investigate the impact of the discount parameter $\sigma$ of the generalized gamma CRM, showing the exceptional differences between these two scenarios. \\

The remarkable Theorem 1 in \citet{James2009} provides a regular version of the posterior $\mathcal{L}(\crm | X_{1:n})$ arising from model \eqref{eq:model}, that is, a Markov kernel $\phi:\mathbb{X}^n \times \mathcal{B}(\mathcal{M}) \to [0,1]$, where $\mathcal{B}(\mathcal{M})$ denotes the Borel $\sigma$-algebra on the space of locally finite measures $\M(\mathbb{X})$. For observed data $x_{1:n}$, we denote by $\mathcal{L}(\crm | X_{1:n} = x_{1:n})$ the probability measure on $\M(\mathbb{X})$ defined by $A \mapsto \phi(x_{1:n},A)$. A crucial aspect of our setting is that $x_{1:n}$ is a generic deterministic sequence. It could be a realization from the model, a realization from another model, or indeed any deterministic sequence of data: we are able to analyze the asymptotic behavior of the merging without putting any distributional assumptions on $x_{1:n}$.
We will explain how our results depend only on the number of observations $n$ and the number $k = k_n$ of distinct values among the $n$ first observations.

We denote by $\crm^* \sim \mathcal{L}(\crm | X_{1:n} = x_{1:n})$ a random measure that we often address as the posterior random measure, omitting the dependence on $n$ for notational convenience.
The aforementioned Theorem 1 in \citet{James2009}, which we state below for completeness, shows that $\crm^*$ is a CRM conditionally on a latent variable. Using the terminology of Section~\ref{sec:crm}, this implies that $\crm^*$ is a Cox CRM. To preserve the identifiability discussed in Section~\ref{sec:identifiability}, we do not define the distance at the level of the posterior Cox CRM $\crm^*$ but rather on its scaled version $\crm^{*} \s$ (cf. Definition~\ref{def:cox_scaled}). We focus on the distance $\wad$ between their random L\'evy measures, where $\wad$ is the Wasserstein distance over the $\ad$ metric in Definition~\ref{def:distance_levy}, which is tailored for Cox scaled CRMs and whose explicit expressions will allow us to derive exact asymptotic behaviors. As described in \eqref{def:wad}, these can be used to derive upper bounds of $\wbl$, the Wasserstein distance over the $\bl$ metric in \eqref{def:bl}, which is a natural distance for arbitrary random measures without additional structure. \\

In the rest of the section the base probability $P_0$ is assumed to be non-atomic, $(x_1^*,\dots, x_k^*)$ are the $k \le n$ distinct observations in $(x_1,\dots,x_n)$ and $n_i = \#\{j: x_j = x_i^*\}$ is the number of observations equal to $x_i^*$. We define a random variable $U$ with the following p.d.f.
\begin{equation}
\label{eq:pdf_u}
f_U(u) \propto u^{n-1} e^{- \psi(u)} \prod_{i=1}^k \tau_{n_{i}|x_i^*}(u),
\end{equation}
where $\psi$ is the Laplace exponent of $\crm$ and $\tau_{m|x}$ are its cumulants, respectively equal to 
\[
\psi(u) = \iint_{(0,+\infty) \times \mathbb{X}} (1-e^{-u s}) \, \ddr \nu(s,x), \qquad \tau_{m|x}(u) = \int_{(0,+\infty)} e^{-u s} s^{m} \ddr \rho_x(s).
\]

\begin{theorem}[\citet{James2009}]
\label{th:lijoi_pruenster} 
Let $\crm$ be a CRM with L\'evy intensity $\ddr \nu(s,x) = \ddr \rho_x(s) \ddr P_0(x)$. Consider $U=U_n$ a latent variable  with p.d.f. \eqref{eq:pdf_u} and define $\crm^*$ such that
$$
\tilde{\mu}^* |  U \eqd \tilde{\mu}_{U}+\sum_{i=1}^k J_i^{U} \delta_{x_i^*},
$$
where (i) $\tilde{\mu}_{U}$ is a CRM with Lévy intensity $\ddr \nu_{U}(s,x)=\mathrm{e}^{-U s} \ddr \rho_x(s) \ddr P_0(x)$, (ii) the i-th jump $J_i^{U}$ has p.d.f. $f_i(s) \propto$ $s^{n_i} \mathrm{e}^{-U s} \ddr \rho_{x_i^*}(s)$, (iii) $\{J_i^{U}\}_i$ are mutually independent and independent from $\tilde{\mu}_{U}$. \\
Then $\mathcal{L}(\crm^*) = \mathcal{L}(\crm |X_{1:n} = x_{1:n})$ is a version of the posterior distribution according to model \eqref{eq:model}.
\end{theorem}

\subsection{Dirichlet process}
\label{sec:dp}

In this section we study the impact of the parameters of the Dirichlet process, which is a normalized gamma CRM, by studying the OT distance between scaled posteriors. We first show that the posterior is a Cox CRM and find its corresponding scaled version (Lemma~\ref{th:gamma_cox}). Then, we prove several finite-sample and asymptotic properties for the OT distance (Proposition~\ref{th:wass_dp_posteriori}). This allows us to state new conclusions about the merging rate of opinions (Corollary~\ref{th:dp_merging}) and study its finite-sample behavior.

First of all we show that the posterior is a Cox CRM. This result builds on Theorem~\ref{th:lijoi_pruenster}, by further showing that the jumps $J_i^{U}$ are infinitely divisible, and can thus be included in the random L\'evy intensity, as shown in Lemma SM4 in the Supplementary Material \citep{CatalanoLavenant2025proofs}.

\begin{lemma}
\label{th:gamma_cox}
Let $\crm$ be a gamma CRM of parameters $(\alpha,P_0)$. Then $\crm^*$ is a Cox CRM and $\crm^*\s$ is a scaled gamma CRM with canonical L\'evy intensity $ \ddr \nu^* \s (s,x)  = \rho^* (s) \ddr s \, \ddr P_0^*(x)$, where
\begin{align*}
&\rho^*(s) =  (\alpha+n) \frac{e^{-(\alpha+n)s}}{s}; \qquad P_0^* = \frac{\alpha}{\alpha+n}  P_0 + \frac{1}{\alpha+n}\sum_{i=1}^n \delta_{x_i}.
\end{align*}
\end{lemma}

Lemma~\ref{th:gamma_cox} shows that the posterior of a gamma CRM is a Cox CRM with a peculiar feature: all the randomness of its L\'evy intensity is contained in a scale parameter, so that the corresponding Cox scaled CRM reduces to a scaled CRM. In the sequel we write $\Gamma(z,t) = \int_t^{+ \infty} x^{z-1} e^{-x} \, \ddr x$ for the incomplete Gamma function.

\begin{proposition}
\label{th:dp_posteriori}
Let $\crm^i$ be a gamma CRM of parameters $(\alpha_i,P_0^i)$
for $i=1,2$. Then it holds that $\wad(\mathcal{L}(\tilde \nu^1 \s),\mathcal{L}(\tilde \nu^2 \s)) = \ad (\nu^{1*} \s, \nu^{2*} \s) = \mathcal{J} + \mathcal{A}$,  with $\mathcal{J} = \mathcal{J}(\alpha_1, \alpha_2, n)$ and $\mathcal{A} = \mathcal{A}(\alpha_1, P_0^1, \alpha_2, P_0^2, n) $ defined as 
\begin{align*}
\mathcal{J} &= \int_0^{+\infty} | (\alpha_1 +n) \Gamma(0, (\alpha_1 +n) s) - (\alpha_2 +n) \Gamma(0,(\alpha_2 + n) s) | \ddr s, \\
\mathcal{A} &= \Wt \bigg(\frac{\alpha_1}{\alpha_1 + n} P_{0}^1 +\frac{1}{\alpha_1 + n} \sum_{i=1}^n \delta_{x_i} , \frac{\alpha_2}{\alpha_2 + n} P_{0}^2 +\frac{1}{\alpha_2 + n} \sum_{i=1}^n \delta_{x_i}\bigg).
\end{align*}
\end{proposition}

We observe that $\mathcal{J}$ measures the discrepancy between the jumps of the posterior random measures, which does not depend on the base probability measures and depends on the observations only through their cardinality $n$. As for $\mathcal{A}$, it measures the discrepancy between the atoms a posteriori and it coincides with the truncated Wasserstein distance between the predictive distributions  $\mathcal{L}(X_{n+1} | X_{1:n})$ of the corresponding models \eqref{eq:model}.

\begin{proposition}
\label{th:wass_dp_posteriori}
Consider the framework of Proposition~\ref{th:dp_posteriori} where, without loss of generality, $\alpha_1 \le \alpha_2$, and let $c = \int_0^{+ \infty} |\Gamma(0,t) - e^{-t}| \ddr t$. Then, there holds
\begin{enumerate}
\item $n \mapsto \mathcal{J}(\alpha_1, \alpha_2, n)$ is decreasing in $n$. 
\item For every $n \ge 0$,
\begin{equation}
\label{eq:dp_posterior_bound}
\mathcal{J}(\alpha_1, \alpha_2, n) \leq c \log \bigg( \frac{\alpha_2 + n}{\alpha_1+n} \bigg).
\end{equation}
\item As $n$ goes to $+\infty$,
\[ \mathcal{J}(\alpha_1, \alpha_2, n)  = \frac{c (\alpha_2 - \alpha_1)}{n} + o \bigg( \frac{1}{n} \bigg). \]
\item For every $n \ge 1$, the following inequality holds, with equality if $P_0^1 = P_0^2$ or $\alpha_1 = \alpha_2$,
\[
\mathcal{A} \le \frac{\alpha_1}{\alpha_1 + n} \Wt (P_0^1, P_0^2) + \frac{(\alpha_2 - \alpha_1) n}{(\alpha_1 + n)(\alpha_2 + n)}  \Wt \bigg( \frac{1}{n} \sum_{i=1}^n \delta_{x_i}, P_0^2 \bigg).
\]
\end{enumerate}
\end{proposition}

Proposition~\ref{th:wass_dp_posteriori} gives us the means to unravel both finite-sample and asymptotic properties of the distance between posterior random measures. From an asymptotic point of view, the next result is an immediate but remarkable consequence. For sequences $(a_n)_{n \in \N}$, $(b_n)_{n \in \N}$ we write $a_n \asymp b_n$ when the quotient $(a_n/b_n)_{n \in \N}$ stays bounded from above and from below by a strictly positive constant.

\begin{theorem} 
\label{th:dp_merging}
Let $\crm^1, \crm^2$ be two gamma CRMs of parameters $(\alpha_i,P_0^i)$, for $i=1,2$, with $\alpha_1 \neq \alpha_2$ or $P_0^1 \neq P_0^2$. 
Then for any sequence $(x_n)_{n \ge 1}$ of observations on $\mathbb{X}$,
\[
\ad(\nu^{1*} \s, \nu^{2*} \s)  \asymp \frac{1}{n}.
\]
In particular, $\wbl(\mathcal{L}(\crm^{1*} \s), \mathcal{L}(\crm^{2*} \s))  \lesssim 1/n$.
\end{theorem}

Theorem~\ref{th:dp_merging} states that there will always be merging of opinions, and provides us with the merging rate of $1/n$, which holds independently of the choice of parameters $(\alpha_i,P_0^i)$ and independently of the dimension of the space $\mathbb{X}$. Tracking down the multiplicative constants in the proof of Theorem~\ref{th:dp_merging}, we actually prove the upper bound
\begin{equation*}
\ad(\nu^{1*} \s, \nu^{2*} \s) \leq \frac{1}{n} \left\{  \min(\alpha_1,\alpha_2) \Wt (P_0^1, P_0^2) + (c+2) |\alpha_2 - \alpha_1| \right\}.
\end{equation*}
Thus, even if the difference in the priors (measured by $\Wt (P_0^1, P_0^2)$ and $|\alpha_2 - \alpha_1|$) does not affect the rate, it influences the premultiplicative factor. Moreover, thanks to Proposition~\ref{th:wass_dp_posteriori} we can also study the \emph{starting time} of the merging, which highlights a different behavior for the jump and the atom component, as discussed in the remainder of this section.

\begin{figure}
\begin{center}
\includegraphics[width = 0.49\textwidth, trim=0cm 0cm 0cm 0cm, clip]{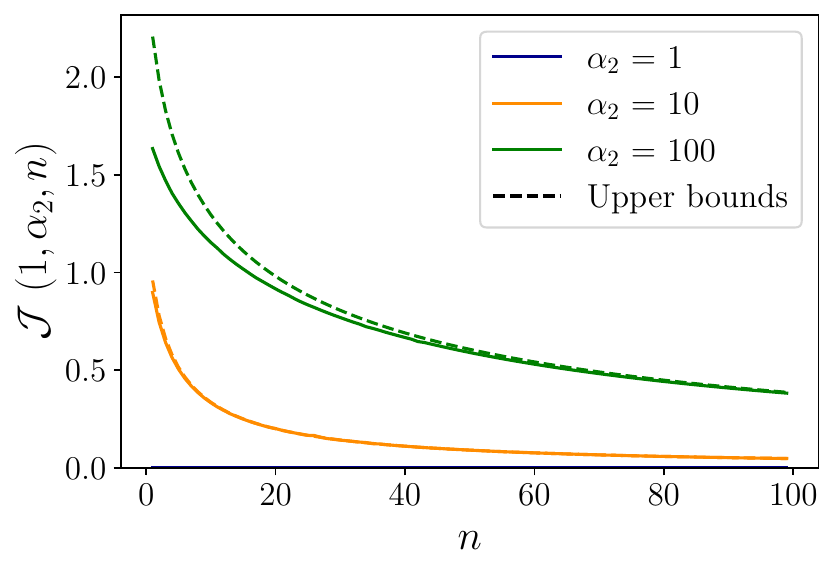}
\caption{Jump component $\mathcal{J}$ of $\wad(\mathcal{L}(\tilde \nu^{1*} \s), \mathcal{L}(\tilde  \nu^{2*} \s))$ when $\crm^1$ is a gamma$(\alpha_1=1,P_0^1)$ CRM and $\crm^2$ is a gamma$(\alpha_2,P_0^2)$ CRM, as the number of observations $n$ increases. The upper bounds are the right hand side of \eqref{eq:dp_posterior_bound}.}
\label{fig:post_gammaGamma_jumps}
\end{center}
\end{figure}

Point 1. of Proposition~\ref{th:wass_dp_posteriori} states that the jump component $\mathcal{J}$ decreases steadily in $n$, that is, the merging starts with the first observation (time 1). We illustrate this behavior in Figure~\ref{fig:post_gammaGamma_jumps}, where without loss of generality we fix $\alpha=1$ and plot the value of $\mathcal{J}$ as a function of $n$, for different values of $\alpha_2$, together with the upper bound \eqref{eq:dp_posterior_bound}, which becomes asymptotically exact as $n \to + \infty$, as stated in point 3. of Proposition~\ref{th:wass_dp_posteriori}. As for the atom component $\mathcal{A}$, this can show different behaviors depending on the parameters $(\alpha_i,P_0^i)$, as showcased by the study of the upper bounds in point 4. of Proposition~\ref{th:wass_dp_posteriori}. In particular, if $\alpha_1 = \alpha_2 = \alpha$, 
\[
\mathcal{A} = \frac{\alpha}{\alpha + n} \Wt (P_0^1, P_0^2).
\]
Thus, in this scenario $\mathcal{A}$ steadily decreases as $n$ increases, that is, the merging of the atoms starts at time 1. The left panel of Figure~\ref{fig:gammaGammaPosterioriAtoms} validates it numerically by taking $\alpha_1 = \alpha_2 = \alpha$ for different values of $\alpha$, $P_0^1$ and $P_0^2$ Gaussian distributions with different mean and same variance, and  i.i.d. observations from a Poisson distribution. 
\begin{figure}
\begin{center}
\includegraphics[width = 0.49\textwidth, trim=0cm 0cm 0cm 0cm, clip]{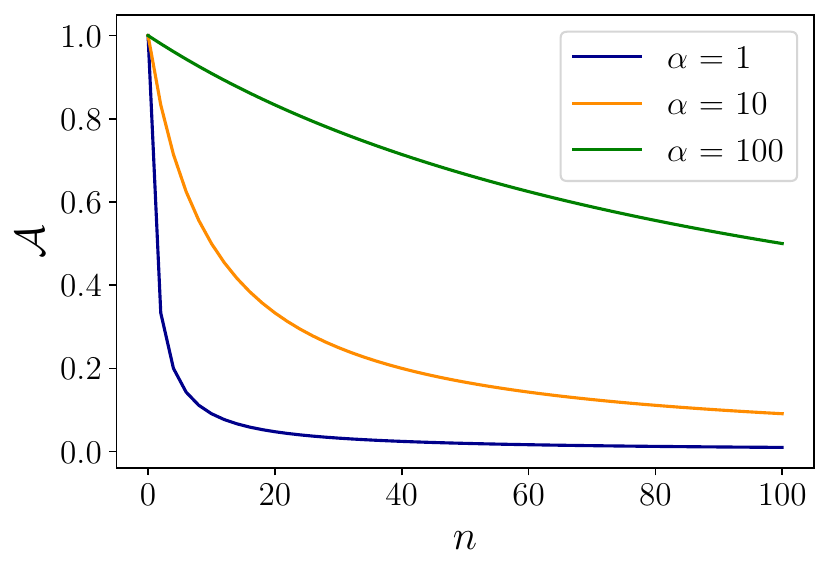}\includegraphics[width = 0.49\textwidth, trim=0cm 0cm 0cm 0cm, clip]{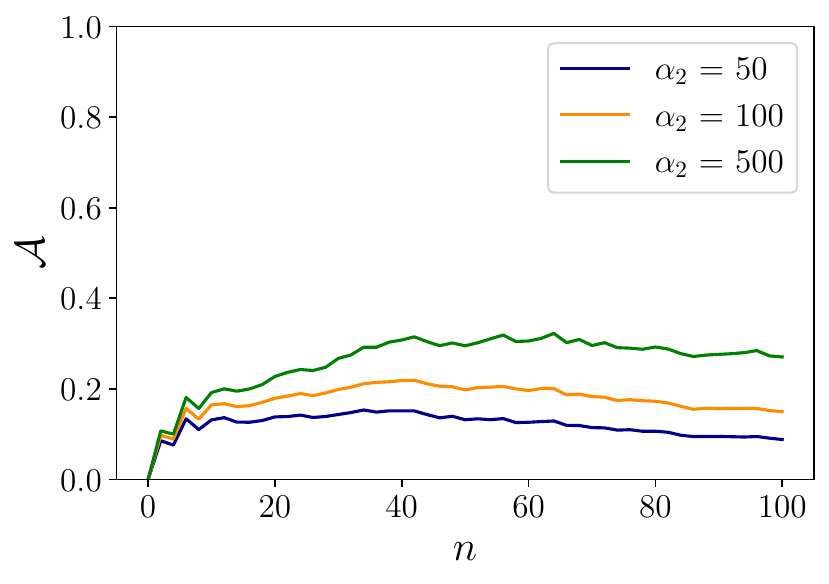}
\caption{Atom component $\mathcal{A}$ of $\wad(\mathcal{L}(\tilde \nu^{1*} \s), \mathcal{L}(\tilde  \nu^{2*} \s))$, as the number of observations $n$ increases, in two different scenarios. Left: $\crm^1$ is a gamma$(\alpha_1=\alpha,P_0^1 = \mathcal{N}(1,1))$ CRM and $\crm^2$ is a gamma$(\alpha_2 = \alpha,P_0^2 = \mathcal{N}(2,1))$ CRM, where $\mathcal{N}(m,v)$ is a Gaussian distribution of mean $m$ and variance $v$. Right: $\crm^1$ is a gamma$(\alpha_1=10,P_0^1 = \mathcal{N}(1,1))$ CRM and $\crm^2$ is a gamma$(\alpha_2,P_0^2 = \mathcal{N}(1,1))$ CRM. In both scenarios the observations are an i.i.d. sample from a Poisson distribution of mean 1.}
\label{fig:gammaGammaPosterioriAtoms}
\end{center}
\end{figure}
If on the other hand $P^1_0 = P^2_0$, 
\[
\mathcal{A} =  \frac{(\alpha_2 - \alpha_1) n}{(\alpha_1 + n)(\alpha_2 + n)}  \Wt \bigg(P_0, \frac{1}{n} \sum_{i=1}^n \delta_{x_i}\bigg).
\]  
For many standard data-generating mechanisms for the data, i.e., if we consider $(x_n)_{n \ge 1}$ as a realization of a random sequence $(X_n)_{n \ge 1}$ possibly differing from model \eqref{eq:model}, $\Wt (P_0, \frac{1}{n} \sum_{i=1}^n \delta_{x_i})$ has a non-zero limit almost surely. For example, this happens if  the data are i.i.d. samples from $P \neq P_0$ thanks to the law of large numbers or if they are exchangeable thanks to de Finetti Theorem. In such case, for $n$ large enough, the behavior of $\mathcal{A}$ with respect to $n$ reduces to the behavior of the prefactor 
\[ \omega(n) = \frac{(\alpha_2 - \alpha_1) n}{(\alpha_1 + n)(\alpha_2 + n)}.\]
This quantity can be studied analytically: it first increases and reaches a maximum 
at $n \approx \sqrt{\alpha_1 \alpha_2}$ before decreasing to $0$ at rate $1/n$, hinting that this may well be the case for $\mathcal{A}$ as well, at least when $\sqrt{\alpha_1 \alpha_2}$ is large enough. We tested this behavior numerically. In Figure~\ref{fig:gammaGammaPosterioriAtoms} on the right we fix $P_0^1 = P_0^2 = P_0$ to be a Gaussian distribution, $\alpha_1 = 10$ and $\alpha_2$ to vary, while the observations are i.i.d. from $P$ a Poisson distribution: in this setting we know that the expected value of $\Wt (P_0, \frac{1}{n} \sum_{i=1}^n \delta_{x_i})$ converges to $\Wt(P_0,P)$ with rate $1/\sqrt{n}$ \citep{Fournier2015rate}. We see that the distance is not decreasing in $n$, the maximum being achieved for $\bar n \approx 34$ when $\alpha_2 = 50$, $\bar n \approx 42$ when $\alpha_2 = 100$ and $\bar n \approx 64$ when $\alpha_2 = 500$. As expected, the maximum is better approximated by $\sqrt{\alpha_1 \alpha_2}$ (equal to approximately 21, 31, and  72, respectively)  as $\sqrt{\alpha_1 \alpha_2}$ increases.

The previous paragraph shows that if $P^1_0 = P^2_0$ but $\alpha_1 \neq \alpha_2$ then $\mathcal{A}$ may be increasing. The quantity $\mathcal{A}$ corresponds to the Wasserstein distance between the predictive distributions $\mathcal{L}(X_{n+1} | X_{1:n})$ if $(X_n)_{n \ge 1}$ follow model~\eqref{eq:model_intro}, thus is directly relevant to the practitioner. Moreover, we can also guarantee that for some values of the parameters and some data the distance $\ad(\nu^{1*} \s, \nu^{2*} \s)$ must first be increasing. Indeed,
as $\mathcal{A} = 0$ when $n=0$ and $P^1_0 = P^2_0$,
then $\ad(\nu^{1*} \s, \nu^{2*} \s) > \ad(\nu^{1} \s, \nu^{2} \s)$ as soon as
\begin{equation*}
\Wt \bigg(P_0, \frac{1}{n} \sum_{i=1}^n \delta_{x_i}\bigg) \geq \frac{(\alpha_1+n)(\alpha_2+n)}{(\alpha_2 - \alpha_1) n} (\mathcal{J}(\alpha1, \alpha2, 0) - \mathcal{J}(\alpha1, \alpha2,n)), 
\end{equation*}
meaning that the Lévy intensities a posteriori are further apart than a priori. 
For example, if $\alpha_1 = 10$, $\alpha_2 = 50$, and $n =20$, it is the case if $\Wt (P_0, \frac{1}{n} \sum_{i=1}^n \delta_{x_i}) \geq 1.012$. 
However, with our analytical tools we cannot conclude if $\wbl(\mathcal{L}(\crm^{1*} \s), \mathcal{L}(\crm^{2*} \s))$ can be increasing, being $\ad(\nu^{1*} \s, \nu^{2*} \s)$ only an upper bound to $\wbl(\mathcal{L}(\crm^{1*} \s), \mathcal{L}(\crm^{2*} \s))$. Proposition~\ref{prop:lower_bound_WBL} yields the lower bound $\wbl(\mathcal{L}(\crm^{1*} \s), \mathcal{L}(\crm^{2*} \s)) \geq \mathcal{A}/2$, so that both the upper and lower bounds are increasing for small values of $n$. However, this does not necessarily imply that $\wbl(\mathcal{L}(\crm^{1*} \s), \mathcal{L}(\crm^{2*} \s))$ is also increasing. A formal proof is left for future work.

In conclusion, depending on the total base measure and the base probabilities, we can have different behaviors for the adapted extended Wasserstein discrepancy: either a steady decrease, or an increase phase followed by a decrease, but ultimately the asymptotic rate of convergence is $1/n$. From a statistical perspective this finding is quite remarkable: even if two Bayesians have the same prior on a phenomenon, the difference between their opinions may increase if they see a finite amount of data, according to how much confidence they are willing to put on their prior guesses.

\subsection{Normalized generalized gamma process}
\label{sec:gengamma}

In this section we study the impact of the discount parameter $\sigma$ of a generalized gamma CRM, that is, the effect of using a normalized generalized gamma CRM instead of a Dirichlet process in model~\eqref{eq:model}. As in the case of Gamma CRMs, we first prove that the posterior is a Cox CRM and find the expression of its Cox scaled CRM (Proposition~\ref{th:gengamma_posteriori}), which is used to express the OT distance between the posteriors in Corollary~\ref{th:wass_gengamma_post}. We then develop an in-depth asymptotic study of the OT distance as the number of observations diverge. Since the L\'evy intensities depend on a latent variable $U=U_n$, we first need to unravel its asymptotic behavior, which crucially depends on the number of distinct observations $k=k_n$ (Theorem~\ref{th:latent}). This gives us the means to study its implications on the behavior on the merging of opinions (Theorem~\ref{th:wass_rates} and Proposition~\ref{th:wass_rates_continuous}).

Recall that $(x_1^*,\dots, x_k^*)$ are the $k \le n$ distinct observations in $(x_1,\dots,x_n)$.
We show that the posterior is a Cox CRM with the same technique used in Section~\ref{sec:dp} for the gamma CRM. We define
\begin{equation}
\label{eq:def_C}
C = C(\alpha,\sigma,n, k, U) = \alpha (1+U)^\sigma + n- k\sigma,   
\end{equation}
where we recall that $U = U_n$ has p.d.f. \eqref{eq:pdf_u}.

\begin{proposition} 
\label{th:gengamma_posteriori}
Let $\crm$ be a generalized gamma CRM of parameters $(\alpha, \sigma, P_0)$. Then $\crm^*$ is a Cox \textup{CRM} and $\crm^*\s$ is a Cox scaled CRM whose random L\'evy intensity has canonical form $ \ddr \tilde \nu \s ^* (s,x)  =  \tilde \rho^*_x (s) \ddr s \, \ddr \tilde P_0^*(x)$, where
\begin{align*}
 & \tilde \rho^*_x(s) = \frac{C^{1-\sigma}}{\Gamma(1-\sigma)} \frac{e^{-Cs}}{s^{1+\sigma}} \mathbbm{1}_{\mathbb{X}\setminus\{x_1^*,\dots,x_k^*\}}(x) +C \frac{e^{-Cs}}{s} \mathbbm{1}_{\{x_1^*,\dots,x_k^*\}}(x), \\
&\tilde P_0^* = \frac{ \alpha(1+U)^\sigma}{C}  P_0 + \frac{1}{C}\sum_{i=1}^k (n_i - \sigma) \delta_{x_i^*},
\end{align*}
and $U = U_n$ has probability density proportional to $u^{n-1}(1+u)^{k \sigma -n} e^{-\frac{\alpha}{\sigma}(1+u)^\sigma} \mathbbm{1}_{(0,+\infty)}(u)$ while $C = C_n$ is defined in~\eqref{eq:def_C}.
\end{proposition}

We then use this expression to compute the OT distance a posteriori between a generalized gamma and a gamma CRM. 

\begin{proposition}
\label{th:wass_gengamma_post}
Let $\crm^1$ be a generalized gamma CRM of parameters $(\alpha, \sigma, P_0)$, and let $\crm^2$ be a gamma CRM with same total base measure $\alpha$ and base probability $P_0$. Then,
\[
\wad(\mathcal{L}(\tilde \nu^{1*} \s), \mathcal{L}(\tilde  \nu^{2*} \s))  = \E_U(\mathcal{J}_1^U + \mathcal{J}_2^U + \mathcal{A}^U),
\]
where
\begin{align*}
&\mathcal{J}_1^U = \frac{\alpha (1+U)^ \sigma }{C} \W_{*}\bigg( \frac{C^{1-\sigma}}{\Gamma(1-\sigma)} \frac{e^{-Cs}}{s^{1+\sigma}} \ddr s,(\alpha+ n)\frac{e^{-(\alpha+ n)s}}{s} \ddr s\bigg), \\
& \mathcal{J}_2^U = \frac{n-k \sigma}{C} \W_{*}\bigg( C \frac{e^{-Cs}}{s} \ddr s, (\alpha+ n)\frac{e^{-(\alpha+ n)s}}{s} \ddr s\bigg), \\
& \mathcal{A}^U = \Wt \bigg(\frac{\alpha ( U +1)^\sigma}{C} P_0 + \frac{1}{C}\sum_{i=1}^k (n_i - \sigma) \delta_{X_i^*}, \frac{\alpha}{\alpha + n} P_0 + \frac{1}{\alpha+n} \sum_{i=1}^k n_i \delta_{x_i^*} \bigg),
\end{align*}
with $U = U_n$ and $C = C_n$ are defined as in Proposition~\ref{th:gengamma_posteriori}.
\end{proposition}

\begin{remark}The expression in Proposition~\ref{th:wass_gengamma_post} can be approximated through Monte Carlo integration, that is, by taking $u_1, \dots, u_N \simiid \mathcal{L}(U)$ and then evaluating the unbiased estimator
\[
\hat \wad = \frac{1}{N} \sum_{i=1}^N (\mathcal{J}_1^{u_i} + \mathcal{J}_2^{u_i} + \mathcal{A}^{u_i}),
\]
where $\mathcal{J}_1^{u_i}$ and $\mathcal{J}_2^{u_i}$ can be evaluated by numerical integration and $\mathcal{A}^{u_i}$ only involves the classical 1-Wasserstein distance on $\mathbb{X}$ with a truncated metric, which can be approximated by many existing efficient algorithms \citep{Peyre2019computational}.
\end{remark}

To investigate general features of $\wad$ as the sample size $n$ grows, one needs to study the asymptotic behavior of the latent variables $U=U_n$. The next result finds the asymptotic distribution of $(1+U)^\sigma$ as the number of observations diverge. We focus on $(1+U)^\sigma$ instead of $U$ because the former is featured in the expressions of Proposition~\ref{th:wass_gengamma_post}. The core of the proof shows that $(1+U)^\sigma$ has a log-concave density, and concentrates around a point of maximum of its log-density. Indeed, with a change of variable one finds that the density of $(1+U)^\sigma$ is proportional to $\exp(-f_n(x))$, where
\[
f_n(x) =  - (k_n-1) \log(x) + \frac{\alpha}{\sigma}  x - (n-1) \log \bigg( 1- \frac{1}{x ^{1/\sigma}} \bigg)
\]
is a convex function. If $f_n$ were proportional to $n$, say $f_n(x) = n g(x)$ for some smooth function $g$, then the result would come directly from Laplace's method. Here the situation is more delicate as there are different scales (namely, the constant $\alpha/\sigma$, $k_n$ and $n$) in the function $f_n$. We extend Laplace's method to this setting in the Supplementary Material \citep{CatalanoLavenant2025proofs}, which yields a general result of independent interest. Note on the other hand that the density of $U$ itself is \emph{not} log-concave.

Before stating the result let us first introduce the following standard notations for positive deterministic sequences $(a_n)_{n \in \N}$ and $(b_n)_{n \in \N}$:
\begin{itemize}
\item $a_n \ll b_n$ if $(a_n/b_n)_{n \in \N}$ converges to $0$. Equivalently, $a_n = o(b_n)$.
\item $a_n \lesssim b_n$ if $(a_n/b_n)_{n \in \N}$ stays bounded from above.
\item $a_n \sim b_n$ if $(a_n/b_n)_{n \in \N}$ converges to $1$.
\item $a_n \asymp b_n$ if $(a_n/b_n)_{n \in \N}$ stays bounded from above and from below by a strictly positive constant. Equivalently, $a_n \lesssim b_n$ and $b_n \lesssim a_n$.
\end{itemize}

\begin{theorem} 
\label{th:latent}
As $n \rightarrow +\infty$, for the $L^1$ convergence,
\begin{equation*}
\lim_{n \to + \infty} \frac{(1+U)^\sigma}{r_n} = 1,
\end{equation*}
where the sequence $(r_n)_{n \in \N}$ is defined by
\[
r_n =
\begin{cases}
\alpha^{-\frac{\sigma}{1+\sigma}} \ n^\frac{\sigma}{1+\sigma} &\text{\qquad if } k  \ll n^{\frac{\sigma}{1+\sigma}},\\
\gamma \ n^\frac{\sigma}{1+\sigma} &\text{\qquad if } k \sim \lambda n^{\frac{\sigma}{1+\sigma}},\\
\sigma \alpha^{-1} \ k &\text{\qquad if } k \gg  n^{\frac{\sigma}{1+\sigma}},
\end{cases}
\]
being here $\gamma$ is the unique solution of $\sigma \lambda x^{-1} +  x^{-(1+\sigma)/\sigma} = \alpha$.
\end{theorem}

Theorem~\ref{th:latent} highlights a transition in the asymptotic behavior of $(1+U)^\sigma$  as the number of unique values $k$ increases. Indeed, one can easily derive that
\[ \E((1+U)^\sigma) \asymp \max \left( n^{\frac{\sigma}{1+\sigma}}, k \right). \]
This implies that the growth rate of $(1+U)^\sigma$ is not affected by $k$ if $k$ does not diverge or diverge sufficiently slowly; otherwise it coincides with $k$ itself, as shown in Figure~\ref{fig:phase_transition}. 

We now use Theorem~\ref{th:latent} to study the asymptotic behavior of the distance in Proposition~\ref{th:wass_gengamma_post}.

\begin{figure}
\begin{center}
\includegraphics[width = 0.437\textwidth, trim=0cm 0cm 0cm 0cm, clip]{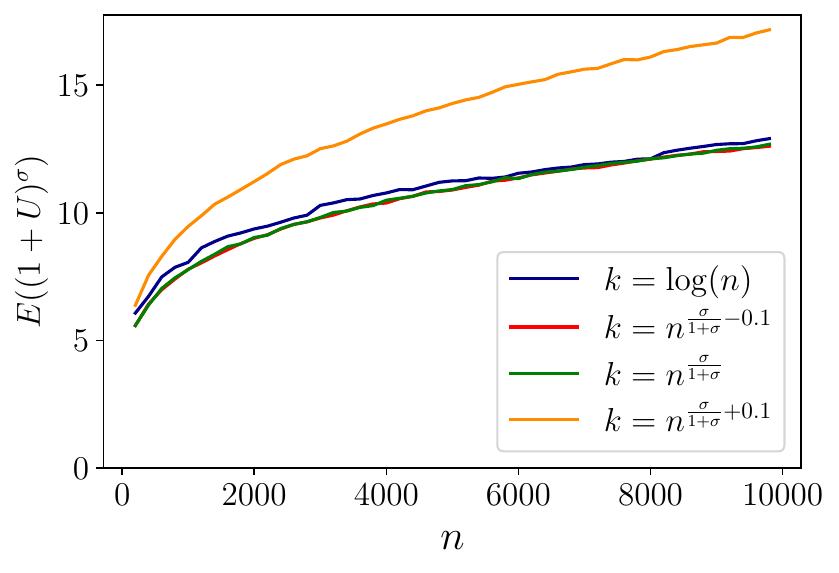}
\caption{Expected value of $(1+U)^\sigma$ for different numbers of unique values $k = k_n$, as the number of observations $n$ increases. The discount parameter is fixed to $\sigma = 0.3$. 
}
\label{fig:phase_transition}
\end{center}
\end{figure}

\begin{figure}
\begin{center}
\includegraphics[width = 0.437\textwidth, trim=0cm 0cm 0cm 0cm, clip]{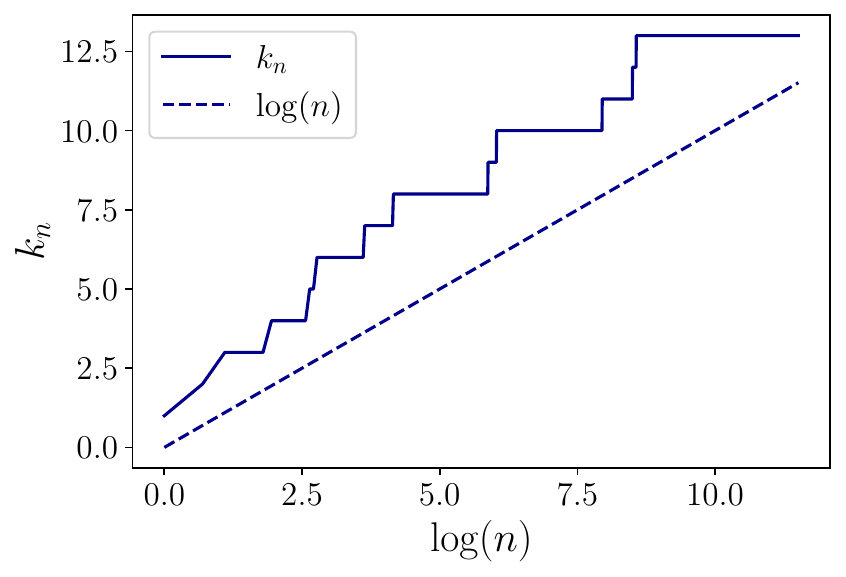}
\includegraphics[width = 0.553\textwidth, trim=0cm 0cm 0cm 0cm, clip]{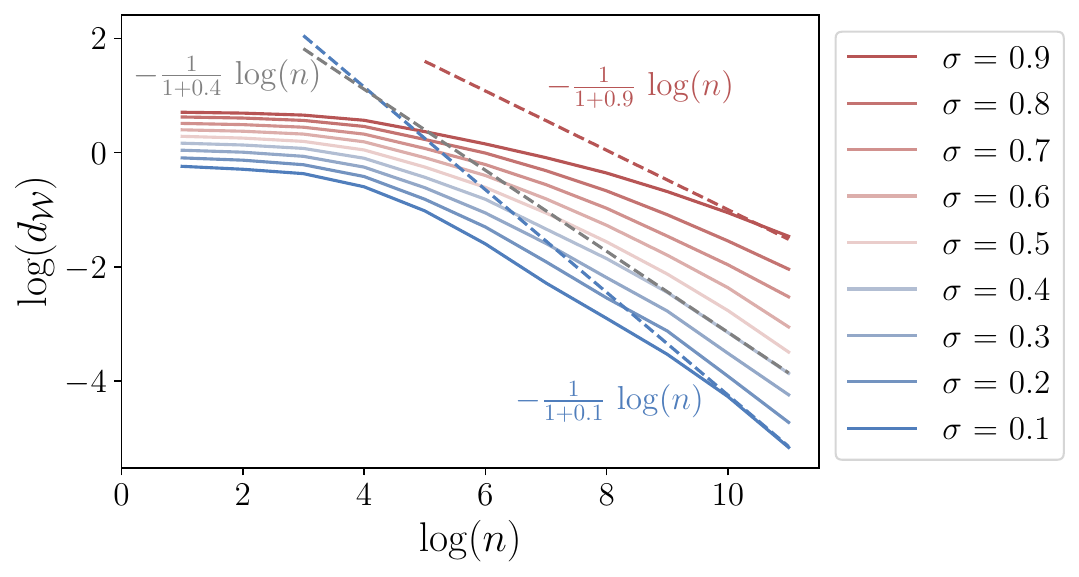}
\caption{Left: Unique values $k = k_n$ for a conditionally i.i.d. sample from a Dirichlet process of parameters $(\bar \alpha = 1, \bar P_0 = \mathcal{N}(0,1))$, as the number of observations $n$ increases. Right: $\wad(\mathcal{L}(\tilde \nu^{1*} \s), \mathcal{L}(\tilde  \nu^{2*} \s))$ when $\crm^1$ is a generalized gamma$(\alpha_1=100,P_0^1 = \mathcal{N}(0,1), \sigma)$ CRM and $\crm^2$ is a gamma$(\alpha_2 = 100,P_0^2 = \mathcal{N}(0,1))$ CRM, as the number of observations $n$ increases exponentially. The observations are conditionally i.i.d. sample from a Dirichlet process of parameters $(\bar \alpha = 1, \bar P_0 = \mathcal{N}(0,1))$.
}
\label{fig:DP}
\end{center}
\end{figure}

\begin{figure}
\begin{center}
\includegraphics[width = 0.437\textwidth, trim=0cm 0cm 0cm 0cm, clip]{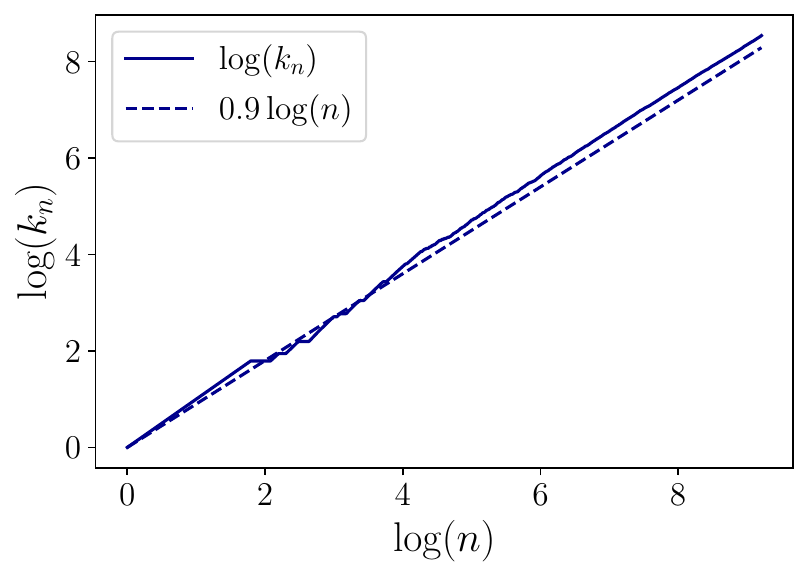}
\includegraphics[width = 0.553\textwidth, trim=0cm 0cm 0cm 0cm, clip]{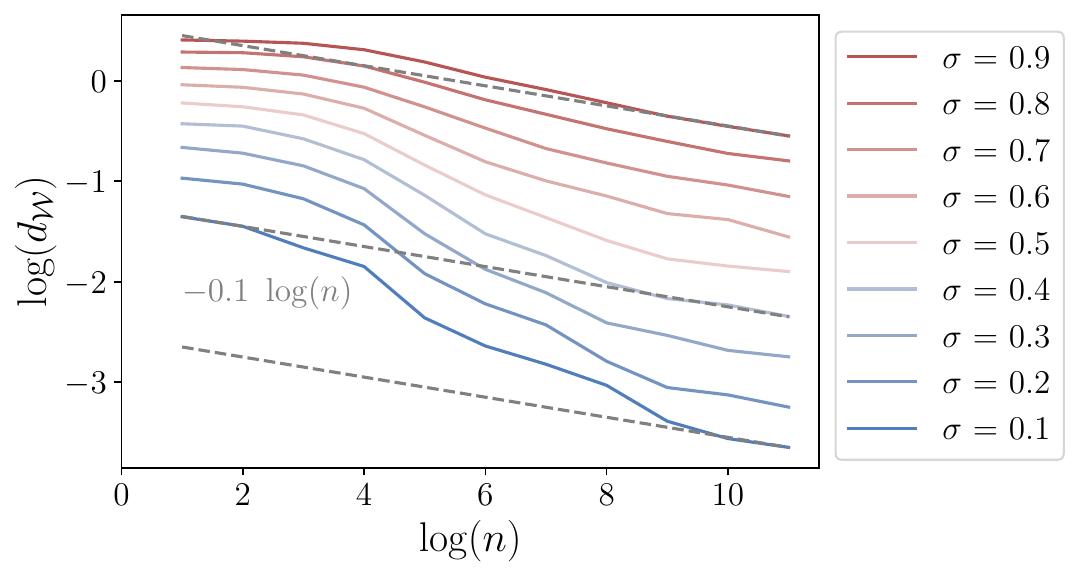}
\caption{Left: Unique values $k = k_n$ for a conditionally i.i.d. sample from a Pitman-Yor process of parameters $(\bar \alpha = 1, \bar \sigma = 0.9, \bar P_0 = \mathcal{N}(0,1))$, as the number of observations $n$ increases. Right: $\wad(\mathcal{L}(\tilde \nu^{1*} \s), \mathcal{L}(\tilde  \nu^{2*} \s))$ when $\crm^1$ is a generalized gamma$(\alpha_1=100,P_0^1 = \mathcal{N}(0,1), \sigma)$ CRM and $\crm^2$ is a gamma$(\alpha_2 = 100,P_0^2 = \mathcal{N}(0,1))$ CRM, as the number of observations $n$ increases exponentially. The observations are conditionally i.i.d. sample from a Pitman-Yor process of parameters $(\bar \alpha = 1, \bar \sigma = 0.9, \bar P_0 = \mathcal{N}(0,1))$.
}
\label{fig:PY}
\end{center}
\end{figure}

\begin{theorem} 
\label{th:wass_rates}
Let $\tilde \mu_1,\tilde \mu_2$ as in Proposition~\ref{th:wass_gengamma_post} and let $(x_n)_{n \ge 1}$ be such that $k_n \ll n$. Then as $n \rightarrow +\infty$, 
\[
\wad(\mathcal{L}(\tilde \nu^{1*} \s), \mathcal{L}(\tilde  \nu^{2*} \s))  \asymp \max (n^{-1/(1+\sigma)} , k/n ).
\]
In particular, $\wbl(\mathcal{L}(\crm^{1*} \s), \mathcal{L}(\crm^{2*} \s)) \lesssim \max (n^{-1/(1+\sigma)} , k/n)$.
\end{theorem}

Theorem~\ref{th:wass_rates} shows many interesting properties of the merging between a gamma and a generalized gamma CRM. First, recall that when comparing two gamma CRMs as in Section~\ref{sec:dp}, the change in the parameters $(\alpha, P_0)$ did not change the merging rate of $1/n$. In contrast, Theorem~\ref{th:wass_rates} shows that when the number of distinct values $k$ diverges sufficiently slowly, a change in $\sigma$ has a heavy impact on the merging, whose rate of convergence slows down to $n^{- \frac{1}{1+\sigma}} \in (1/n, 1/\sqrt{n})$. To illustrate this situation, in Figure~\ref{fig:DP} we sampled conditionally i.i.d. observations from a Dirichlet process of parameters $(\bar \alpha, \bar P_0)$, since it is well known that $k_n \sim \bar \alpha \log(n)$ a.s. \citep{KorwarHollander1973} and thus $k_n \lesssim n^{\frac{\sigma}{1+\sigma}}$ a.s. for every $\sigma$. 

Moreover, Theorem~\ref{th:wass_rates} shows that as the number of distinct values $k$ increases, the rate drastically deteriorates and does not depend anymore on $\sigma$. This is pictured in Figure~\ref{fig:PY}, where we sampled conditionally i.i.d. observations from a Pitman-Yor process \citep{PitmanYor1997} of parameters $(\bar \alpha, \bar \sigma, \bar P_0)$, since in such a case it is well known that $\mathbb{E}(k_n) \asymp n^{\bar \sigma}$ \citep{Pitman2006}, and thus for $\bar \sigma \in [0.5, 1)$ we expect $k_n \gtrsim n^{\frac{\sigma}{1+\sigma}}$ for every $\sigma$. 

Eventually, if $k \asymp n$ there will be no merging of opinions. As a notable example, this happens a.s. when the observations are sampled i.i.d. from a continuous distribution, as described in the next result. It is worth underlining that the lack of merging should not worry the practitioner too much: since model \eqref{eq:model} with normalization gives a positive probability to ties among the observations, i.e., $\mathbb{P}(X_i = X_j) >0$, when dealing with independent and continuous data other models for random densities are to be preferred, such as kernel mixtures over a random probability measure.

\begin{proposition} 
\label{th:wass_rates_continuous}
Let $\tilde \mu_1,\tilde \mu_2$ as in Proposition~\ref{th:wass_gengamma_post} and let $x_n \simiid P$, where $P$ is a continuous distribution.
Then as $n \rightarrow +\infty$, for a.e. realization of $(x_n)_{n \ge 1}$,
\[
\wad(\mathcal{L}(\tilde \nu^{1*} \s), \mathcal{L}(\tilde  \nu^{2*} \s)) \rightarrow \sigma \bigg( \int_0^{+\infty} \bigg| \frac{\Gamma(-\sigma,t)}{\Gamma(1-\sigma)} - \Gamma(0,t) \bigg| \ddr t + \Wt(P_0, P ) \bigg).
\]
Moreover, the limit is increasing in $\sigma$.
\end{proposition}

We observe that the limit in Proposition~\ref{th:wass_rates_continuous} does not depend on $\alpha$ and it is composed by two terms. The first term coincides with the distance \emph{a priori} between a scaled gamma and scaled generalized gamma with same total base measure, and can thus be seen as a measure of discrepancy between the priors for the jumps. The second term coincides with the distance between the base measure and the law of the observations, and can thus be seen as a measure of discrepancy between the (common) prior for the atoms and the data. As expected, the limit vanishes as $\sigma \to 0$ and it increases with $\sigma$. \\
From a statistical perspective, Theorem~\ref{th:wass_rates} detects the regimes where using a normalized generalized gamma CRM will actually guarantee a different learning outcome, with the following take-home message: either the number of observations $n$ is sufficiently small or the number of distinct values $k$ should be sufficiently large, and how large it must be depends on the choice of discount $\sigma$, namely, $k \gtrsim n^{\frac{\sigma}{1+\sigma}}$.
As notable examples, this happens when the observations are i.i.d. from a continuous distribution or when they are conditionally i.i.d. given a normalized generalized gamma CRM or a Pitman-Yor process, but it does not happen when the observations are i.i.d. from a discrete distribution with a finite number of atoms or when they are conditionally i.i.d. given a Dirichlet process or a hierarchical Dirichlet process.

\begin{remark}
Gamma and generalized gamma random measures are arguably the most used classes of completely random measures but specific needs could lead to the use of other CRMs not belonging to this class. In such a case, the merging could be explored with techniques similar to those exposed in this section. Whenever the L\'evy measure is available in analytical form, we expect an analogous of Proposition~16 to be readily available. The difficult part is to find the asymptotic behavior of the latent variable $U$. Our analysis is based on an ad hoc strategy that recognizes a log-concave distribution for a transformation of $U$ that appears in our expressions. We currently do not know to what extent this feature can be extended to the latent variables arising from other CRMs. 
\end{remark}

\section*{Acknowledgments}

The authors would like to thank the Editor, an Associate Editor, and three Referees for the constructive feedback, which has led to significant improvement to the present work. The authors would also like to thank A. Lijoi and I. Pr\"unster for their valuable insights. Part of this research was conducted at the Department of Statistics at the University of Warwick, during a visit by H. Lavenant and while M. Catalano was a Harrison Early Career Assistant Professor. The Department is warmly acknowledged for its hospitality. H. Lavenant is also affiliated with the Bocconi Institute for Data Science and Analytics (BIDSA). M. Catalano was partially supported by a Research Development Fund from the University of Warwick.

\bibliography{bibmerging}

\begin{thebibliography}{6}
\providecommand{\natexlab}[1]{#1}
\providecommand{\url}[1]{\texttt{#1}}
\expandafter\ifx\csname urlstyle\endcsname\relax
  \providecommand{\doi}[1]{doi: #1}\else
  \providecommand{\doi}{doi: \begingroup \urlstyle{rm}\Url}\fi

\bibitem[Catalano et~al.(2020)Catalano, Lijoi, and Prünster]{Catalano2020}
M.~Catalano, A.~Lijoi, and I.~Prünster.
\newblock Approximation of {B}ayesian models for time-to-event data.
\newblock \emph{Electron. J. Statist.}, 14\penalty0 (2):\penalty0 3366--3395,
  2020.

\bibitem[Kallenberg(2017)]{Kallenberg2017}
O.~Kallenberg.
\newblock \emph{Random Measures, Theory and Applications}.
\newblock Probability Theory and Stochastic Modelling. Springer International
  Publishing, Cham, 2017.
\newblock ISBN 9783319415987.

\bibitem[Pitman and Yor(1992)]{PitmanYor1992}
J.~Pitman and M.~Yor.
\newblock {Arcsine Laws and Interval Partitions Derived from a Stable
  Subordinator}.
\newblock \emph{Proceedings of the London Mathematical Society}, s3-65\penalty0
  (2):\penalty0 326--356, 09 1992.
\newblock ISSN 0024-6115.

\bibitem[Santambrogio(2015)]{Santambrogio2015}
F.~Santambrogio.
\newblock Optimal transport for applied mathematicians.
\newblock \emph{Birkh{\"a}user, NY}, 55\penalty0 (58-63):\penalty0 94, 2015.

\bibitem[Sato(1999)]{Sato1999}
K.~Sato.
\newblock \emph{L{\'e}vy Processes and Infinitely Divisible Distributions}.
\newblock Cambridge Studies in Advanced Mathematics. Cambridge University
  Press, Cambridge, 1999.
\newblock ISBN 9780521553025.

\bibitem[Villani(2009)]{villani2009optimal}
C.~Villani.
\newblock \emph{Optimal transport: old and new}, volume 338.
\newblock Springer, Heidelberg, 2009.

\end{thebibliography}


\begin{thebibliography}{60}
\providecommand{\natexlab}[1]{#1}
\providecommand{\url}[1]{\texttt{#1}}
\expandafter\ifx\csname urlstyle\endcsname\relax
  \providecommand{\doi}[1]{doi: #1}\else
  \providecommand{\doi}{doi: \begingroup \urlstyle{rm}\Url}\fi

\bibitem[Backhoff et~al.(2022)Backhoff, Bartl, Beiglb{\"o}ck, and
  Wiesel]{backhoff2022estimating}
J.~Backhoff, D.~Bartl, M.~Beiglb{\"o}ck, and J.~Wiesel.
\newblock Estimating processes in adapted wasserstein distance.
\newblock \emph{The Annals of Applied Probability}, 32\penalty0 (1):\penalty0
  529--550, 2022.

\bibitem[Barbour and Brown(1992)]{barbourbrown1992}
A.~Barbour and T.~Brown.
\newblock Stein's method and point process approximation.
\newblock \emph{Stochastic Processes and their Applications}, 43\penalty0
  (1):\penalty0 9--31, 1992.
\newblock ISSN 0304-4149.
\newblock URL
  \url{https://www.sciencedirect.com/science/article/pii/030441499290073Y}.

\bibitem[Barrios et~al.(2013)Barrios, Lijoi, Nieto-Barajas, and
  Prünster]{Barrios2013}
E.~Barrios, A.~Lijoi, L.~E. Nieto-Barajas, and I.~Prünster.
\newblock {Modeling with Normalized Random Measure Mixture Models}.
\newblock \emph{Statistical Science}, 28\penalty0 (3):\penalty0 313 -- 334,
  2013.

\bibitem[Bartl et~al.(2024)Bartl, Beiglb{\"o}ck, and
  Pammer]{bartl2024wasserstein}
D.~Bartl, M.~Beiglb{\"o}ck, and G.~Pammer.
\newblock {The Wasserstein space of stochastic processes}.
\newblock \emph{Journal of the European Mathematical Society}, 2024.

\bibitem[Blackwell and Dubins(1962)]{BlackwellDubins1962}
D.~Blackwell and L.~Dubins.
\newblock {Merging of Opinions with Increasing Information}.
\newblock \emph{The Annals of Mathematical Statistics}, 33\penalty0
  (3):\penalty0 882 -- 886, 1962.

\bibitem[Bogachev(2007)]{bogachev2007measure}
V.~I. Bogachev.
\newblock \emph{Measure theory. {V}ol. {I}, {II}}.
\newblock Springer-Verlag, Berlin, 2007.
\newblock ISBN 978-3-540-34513-8; 3-540-34513-2.

\bibitem[Brix(1999)]{Brix1999}
A.~Brix.
\newblock Generalized gamma measures and shot-noise cox processes.
\newblock \emph{Advances in Applied Probability}, 31\penalty0 (4):\penalty0
  929--953, 1999.
\newblock ISSN 00018678.

\bibitem[Caron and Fox(2017)]{CaronFox2017}
F.~Caron and E.~Fox.
\newblock Sparse graphs using exchangeable random measures.
\newblock \emph{J R Stat Soc Series B Stat Methodol.}, 5\penalty0
  (79):\penalty0 1295--1366, 2017.

\bibitem[Catalano and Lavenant(2024)]{CatalanoLavenant2024}
M.~Catalano and H.~Lavenant.
\newblock Hierarchical integral probability metrics: A distance on random
  probability measures with low sample complexity.
\newblock In R.~Salakhutdinov, Z.~Kolter, K.~Heller, A.~Weller, N.~Oliver,
  J.~Scarlett, and F.~Berkenkamp, editors, \emph{Proceedings of the 41st
  International Conference on Machine Learning}, volume 235 of
  \emph{Proceedings of Machine Learning Research}, pages 5841--5861. PMLR,
  21--27 Jul 2024.
\newblock URL \url{https://proceedings.mlr.press/v235/catalano24a.html}.

\bibitem[Catalano and Lavenant(2025)]{CatalanoLavenant2025proofs}
M.~Catalano and H.~Lavenant.
\newblock Supplement to ``merging rate of opinions via optimal transport on
  random measures".
\newblock 2025.

\bibitem[Catalano et~al.(2020)Catalano, Lijoi, and Prünster]{Catalano2020}
M.~Catalano, A.~Lijoi, and I.~Prünster.
\newblock Approximation of {B}ayesian models for time-to-event data.
\newblock \emph{Electron. J. Statist.}, 14\penalty0 (2):\penalty0 3366--3395,
  2020.

\bibitem[Catalano et~al.(2021)Catalano, Lijoi, and Pr\"unster]{Catalano2021}
M.~Catalano, A.~Lijoi, and I.~Pr\"unster.
\newblock Measuring dependence in the {W}asserstein distance for {B}ayesian
  nonparametric models.
\newblock \emph{Ann. Statist.}, 49\penalty0 (5):\penalty0 2916--2947, 2021.

\bibitem[Catalano et~al.(2024)Catalano, Lavenant, Lijoi, and
  Prünster]{Catalano2024}
M.~Catalano, H.~Lavenant, A.~Lijoi, and I.~Prünster.
\newblock A wasserstein index of dependence for random measures.
\newblock \emph{Journal of the American Statistical Association}, 119\penalty0
  (547):\penalty0 2396--2406, 2024.

\bibitem[Cox(1955)]{cox1955some}
D.~R. Cox.
\newblock Some statistical methods connected with series of events.
\newblock \emph{Journal of the Royal Statistical Society: Series B
  (Methodological)}, 17\penalty0 (2):\penalty0 129--157, 1955.

\bibitem[De~Blasi et~al.(2009)De~Blasi, Peccati, and Pr{\"u}nster]{DeBlasi2009}
P.~De~Blasi, G.~Peccati, and I.~Pr{\"u}nster.
\newblock {Asymptotics for posterior hazards}.
\newblock \emph{The Annals of Statistics}, 37\penalty0 (4):\penalty0 1906 --
  1945, 2009.

\bibitem[Decreusefond et~al.(2016)Decreusefond, Schulte, and
  Thäle]{decreusefond2016}
L.~Decreusefond, M.~Schulte, and C.~Thäle.
\newblock Functional poisson approximation in kantorovich-rubinstein distance
  with applications to u-statistics and stochastic geometry.
\newblock \emph{The Annals of Probability}, 44\penalty0 (3):\penalty0
  2147--2197, 2016.
\newblock ISSN 00911798.
\newblock URL \url{http://www.jstor.org/stable/24735850}.

\bibitem[Diaconis and Freedman(1986)]{DiaconisFreedman1986}
P.~Diaconis and D.~Freedman.
\newblock {On the Consistency of Bayes Estimates}.
\newblock \emph{The Annals of Statistics}, 14\penalty0 (1):\penalty0 1 -- 26,
  1986.

\bibitem[Doksum(1974)]{Doksum1974}
K.~Doksum.
\newblock Tailfree and neutral random probabilities and their posterior
  distributions.
\newblock \emph{The Annals of Probability}, 2:\penalty0 183--201, 1974.

\bibitem[Dukler et~al.(2019)Dukler, Li, Lin, and
  Mont{\'u}far]{dukler2019wasserstein}
Y.~Dukler, W.~Li, A.~Lin, and G.~Mont{\'u}far.
\newblock {Wasserstein of Wasserstein loss for learning generative models}.
\newblock In \emph{International conference on machine learning}, pages
  1716--1725. PMLR, 2019.

\bibitem[Dykstra and Laud(1981)]{DykstraLaud1981}
R.~L. Dykstra and P.~Laud.
\newblock A {B}ayesian nonparametric approach to reliability.
\newblock \emph{The Annals of Statistics}, 9\penalty0 (2):\penalty0 356--367,
  1981.

\bibitem[Erbar et~al.(2025)Erbar, Huesmann, Jalowy, and
  M{\"u}ller]{erbar2023optimal}
M.~Erbar, M.~Huesmann, J.~Jalowy, and B.~M{\"u}ller.
\newblock {Optimal transport of stationary point processes: Metric structure,
  gradient flow and convexity of the specific entropy}.
\newblock \emph{Journal of Functional Analysis}, 289\penalty0 (4):\penalty0
  110974, 2025.

\bibitem[Ferguson(1973)]{Ferguson1973}
T.~S. Ferguson.
\newblock {A Bayesian Analysis of Some Nonparametric Problems}.
\newblock \emph{The Annals of Statistics}, 1\penalty0 (2):\penalty0 209 -- 230,
  1973.

\bibitem[Ferguson(1974)]{Ferguson1974}
T.~S. Ferguson.
\newblock {Prior Distributions on Spaces of Probability Measures}.
\newblock \emph{The Annals of Statistics}, 2\penalty0 (4):\penalty0 615 -- 629,
  1974.

\bibitem[Figalli and Gigli(2010)]{FigalliGigli2010}
A.~Figalli and N.~Gigli.
\newblock A new transportation distance between non-negative measures, with
  applications to gradients flows with {D}irichlet boundary conditions.
\newblock \emph{Journal de Math{\'e}matiques Pures et Appliqu{\'e}es},
  94\penalty0 (2):\penalty0 107--130, 2010.

\bibitem[Fortet and Mourier(1953)]{fortet1953convergence}
R.~Fortet and E.~Mourier.
\newblock {Convergence de la répartition empirique vers la répartition
  théorique}.
\newblock \emph{{Annales scientifiques de l'École Normale Supérieure}},
  70:\penalty0 267--285, 1953.

\bibitem[Fournier and Guillin(2015)]{Fournier2015rate}
N.~Fournier and A.~Guillin.
\newblock {On the rate of convergence in Wasserstein distance of the empirical
  measure}.
\newblock \emph{Probability Theory and Related Fields}, 162\penalty0
  (3):\penalty0 707--738, 2015.

\bibitem[Giordano et~al.(2022)Giordano, Liu, Jordan, and
  Broderick]{Giordano2022}
R.~Giordano, R.~Liu, M.~I. Jordan, and T.~Broderick.
\newblock {Evaluating Sensitivity to the Stick-Breaking Prior in Bayesian
  Nonparametrics (with discussion)}.
\newblock \emph{Bayesian Analysis}, pages 1 -- 53, 2022.

\bibitem[Griffin and Leisen(2017)]{GriffinLeisen2017}
J.~E. Griffin and F.~Leisen.
\newblock Compound random measures and their use in bayesian non-parametrics.
\newblock \emph{Journal of the Royal Statistical Society. Series B (Statistical
  Methodology)}, 79\penalty0 (2):\penalty0 525--545, 2017.
\newblock ISSN 13697412, 14679868.

\bibitem[Guillen et~al.(2019)Guillen, Mou, and {\'S}wi{\c e}ch]{Guillen2019}
N.~Guillen, C.~Mou, and A.~{\'S}wi{\c e}ch.
\newblock Coupling {L}{\'e}vy measures and comparison principles for viscosity
  solutions.
\newblock \emph{Transactions of the American Mathematical Society},
  372\penalty0 (10):\penalty0 7327--70, 2019.

\bibitem[Gustafson(1996)]{Gustafson1996}
P.~Gustafson.
\newblock {Local sensitivity of posterior expectations}.
\newblock \emph{The Annals of Statistics}, 24\penalty0 (1):\penalty0 174 --
  195, 1996.

\bibitem[Hjort(1990)]{Hjort1990}
N.~L. Hjort.
\newblock Nonparametric {B}ayes estimators based on beta processes in models
  for life history data.
\newblock \emph{The Annals of Statistics}, 18:\penalty0 1259--1294, 1990.

\bibitem[Ho et~al.(2017)Ho, Nguyen, Yurochkin, Bui, Huynh, and Phung]{Ho2017}
N.~Ho, X.~Nguyen, M.~Yurochkin, H.~H. Bui, V.~Huynh, and D.~Phung.
\newblock Multilevel clustering via {W}asserstein means.
\newblock 70:\penalty0 1501--1509, 06--11 Aug 2017.

\bibitem[Huesmann(2016)]{huesmann2016optimal}
M.~Huesmann.
\newblock Optimal transport between random measures.
\newblock In \emph{Annales de l’Institut Henri Poincar{\'e}-Probabilit{\'e}s
  et Statistiques}, volume~52, pages 196--232, 2016.

\bibitem[James(2005)]{James2005}
L.~F. James.
\newblock Bayesian {P}oisson process partition calculus with an application to
  {B}ayesian {L\'e}vy moving averages.
\newblock \emph{The Annals of Statistics}, 33:\penalty0 1771--1799, 2005.

\bibitem[James et~al.(2006)James, Lijoi, and Pr{\"u}nster]{James2006}
L.~F. James, A.~Lijoi, and I.~Pr{\"u}nster.
\newblock Conjugacy as a distinctive feature of the {D}irichlet process.
\newblock \emph{Scandinavian Journal of Statistics}, 33\penalty0 (1):\penalty0
  105--120, 2006.

\bibitem[James et~al.(2009)James, Lijoi, and Pr\"unster]{James2009}
L.~F. James, A.~Lijoi, and I.~Pr\"unster.
\newblock Posterior analysis for normalized random measures with independent
  increments.
\newblock \emph{Scandinavian Journal of Statistics}, 36\penalty0 (1):\penalty0
  76--97, 2009.

\bibitem[Kallenberg(2017)]{Kallenberg2017}
O.~Kallenberg.
\newblock \emph{Random Measures, Theory and Applications}.
\newblock Probability Theory and Stochastic Modelling. Springer International
  Publishing, Cham, 2017.
\newblock ISBN 9783319415987.

\bibitem[Kingman(1967)]{Kingman1967}
J.~F.~C. Kingman.
\newblock Completely random measures.
\newblock \emph{Pacific J. Math.}, \penalty0 (21):\penalty0 59--78, 1967.

\bibitem[Korwar and Hollander(1973)]{KorwarHollander1973}
R.~M. Korwar and M.~Hollander.
\newblock {Contributions to the Theory of Dirichlet Processes}.
\newblock \emph{The Annals of Probability}, 1\penalty0 (4):\penalty0 705 --
  711, 1973.

\bibitem[Lau and Cripps(2022)]{LauCripps2022}
J.~W. Lau and E.~Cripps.
\newblock {Thinned completely random measures with applications in competing
  risks models}.
\newblock \emph{Bernoulli}, 28\penalty0 (1):\penalty0 638 -- 662, 2022.

\bibitem[Lee et~al.(2023)Lee, Miscouridou, and Caron]{Lee2023}
J.~Lee, X.~Miscouridou, and F.~Caron.
\newblock {A unified construction for series representations and finite
  approximations of completely random measures}.
\newblock \emph{Bernoulli}, 29\penalty0 (3):\penalty0 2142 -- 2166, 2023.

\bibitem[Ley et~al.(2017)Ley, Reinert, and Swan]{Ley2017}
C.~Ley, G.~Reinert, and Y.~Swan.
\newblock Distances between nested densities and a measure of the impact of the
  prior in bayesian statistics.
\newblock \emph{The Annals of Applied Probability}, 27\penalty0 (1):\penalty0
  216--241, 2017.
\newblock ISSN 10505164.

\bibitem[Lijoi and Pr\"unster(2010)]{LijoiPruenster2010}
A.~Lijoi and I.~Pr\"unster.
\newblock \emph{Models beyond the Dirichlet process}, pages 80--136.
\newblock Cambridge University Press, Cambridge, 2010.

\bibitem[Lijoi et~al.(2005)Lijoi, Mena, and Prünster]{Lijoi2005}
A.~Lijoi, R.~H. Mena, and I.~Prünster.
\newblock Hierarchical mixture modeling with normalized inverse-gaussian
  priors.
\newblock \emph{Journal of the American Statistical Association}, 100\penalty0
  (472):\penalty0 1278--1291, dec 2005.

\bibitem[Lijoi et~al.(2007)Lijoi, Mena, and Prünster]{Lijoi2007}
A.~Lijoi, R.~H. Mena, and I.~Prünster.
\newblock Controlling the reinforcement in bayesian non-parametric mixture
  models.
\newblock \emph{Journal of the Royal Statistical Society: Series B (Statistical
  Methodology)}, 69\penalty0 (4):\penalty0 715--740, 2007.

\bibitem[Lo(1984)]{Lo1984}
A.~Y. Lo.
\newblock On a class of {B}ayesian nonparametric estimates: I. {D}ensity
  estimates.
\newblock \emph{The Annals of Statistics}, 12:\penalty0 351--357, 1984.

\bibitem[Nguyen(2013)]{Nguyen2013}
X.~Nguyen.
\newblock {Convergence of latent mixing measures in finite and infinite mixture
  models}.
\newblock \emph{The Annals of Statistics}, 41\penalty0 (1):\penalty0 370 --
  400, 2013.

\bibitem[Nguyen(2016)]{Nguyen2016}
X.~Nguyen.
\newblock Borrowing strengh in hierarchical bayes: Posterior concentration of
  the dirichlet base measure.
\newblock \emph{Bernoulli}, 22\penalty0 (3):\penalty0 1535--1571, 2016.
\newblock ISSN 13507265.

\bibitem[Nieto-Barajas and Prünster(2009)]{Nieto-BarajasPruenster2009}
L.~E. Nieto-Barajas and I.~Prünster.
\newblock A sensitivity analysis for bayesian nonparametric density estimators.
\newblock \emph{Statistica Sinica}, 19\penalty0 (2):\penalty0 685--705, 2009.
\newblock ISSN 10170405, 19968507.

\bibitem[Peyr{\'e} and Cuturi(2019)]{Peyre2019computational}
G.~Peyr{\'e} and M.~Cuturi.
\newblock Computational optimal transport -- with applications to data science.
\newblock \emph{Foundations and Trends in Machine Learning}, 11\penalty0
  (5-6):\penalty0 355--607, 2019.

\bibitem[Pflug and Pichler(2012)]{pflug2012distance}
G.~C. Pflug and A.~Pichler.
\newblock A distance for multistage stochastic optimization models.
\newblock \emph{SIAM Journal on Optimization}, 22\penalty0 (1):\penalty0 1--23,
  2012.

\bibitem[Pitman(2006)]{Pitman2006}
J.~Pitman.
\newblock \emph{Combinatorial stochastic processes}, volume 1875 of
  \emph{Lecture Notes in Mathematics}.
\newblock Springer-Verlag, Berlin, 2006.
\newblock ISBN 978-3-540-30990-1; 3-540-30990-X.
\newblock Lectures from the 32nd Summer School on Probability Theory held in
  Saint-Flour, July 2002.

\bibitem[Pitman and Yor(1992)]{PitmanYor1992}
J.~Pitman and M.~Yor.
\newblock {Arcsine Laws and Interval Partitions Derived from a Stable
  Subordinator}.
\newblock \emph{Proceedings of the London Mathematical Society}, s3-65\penalty0
  (2):\penalty0 326--356, 09 1992.
\newblock ISSN 0024-6115.

\bibitem[Pitman and Yor(1997)]{PitmanYor1997}
J.~Pitman and M.~Yor.
\newblock The two-parameter {P}oisson-{D}irichlet distribution derived from a
  stable subordinator.
\newblock \emph{The Annals of Probability}, 25\penalty0 (2):\penalty0 855--900,
  1997.

\bibitem[Regazzini et~al.(2003)Regazzini, Lijoi, and
  Pr\"{u}nster]{Regazzini2003}
E.~Regazzini, A.~Lijoi, and I.~Pr\"{u}nster.
\newblock Distributional results for means of normalized random measures with
  independent increments.
\newblock \emph{The Annals of Statistics}, 31:\penalty0 560--585, 2003.

\bibitem[Saha and Kurtek(2019)]{SahaKurtek2019}
A.~Saha and S.~Kurtek.
\newblock Geometric sensitivity measures for bayesian nonparametric density
  estimation models.
\newblock \emph{Sankhya Ser A.}, 81(1)\penalty0 (104-143), 2019.

\bibitem[Santambrogio(2015)]{Santambrogio2015}
F.~Santambrogio.
\newblock Optimal transport for applied mathematicians.
\newblock \emph{Birkh{\"a}user, NY}, 55\penalty0 (58-63):\penalty0 94, 2015.

\bibitem[Schuhmacher and Xia(2008)]{schuhmacher2008new}
D.~Schuhmacher and A.~Xia.
\newblock {A new metric between distributions of point processes}.
\newblock \emph{Advances in applied probability}, 40\penalty0 (3):\penalty0
  651--672, 2008.

\bibitem[Villani(2009)]{villani2009optimal}
C.~Villani.
\newblock \emph{Optimal transport: old and new}, volume 338.
\newblock Springer, Heidelberg, 2009.

\bibitem[Yurochkin et~al.(2019)Yurochkin, Claici, Chien, Mirzazadeh, and
  Solomon]{yurochkin2019hierarchical}
M.~Yurochkin, S.~Claici, E.~Chien, F.~Mirzazadeh, and J.~M. Solomon.
\newblock Hierarchical optimal transport for document representation.
\newblock \emph{Advances in neural information processing systems}, 32, 2019.

\end{thebibliography}

\end{document}


\maketitle

{\em Organization of the supplementary material.}
We present here all the proofs of our results, divided by section. To ease cross-reading between the main manuscript and the supplement, we use a prefix SM for the numbering of results and definitions of the supplementary material (e.g., Proposition SM1, Equation (SM1)) whereas results and definitions of the main manuscript are denoted without prefix (e.g. Proposition 1, Equation (1)).

\section*{Proofs of Section 2}

\begin{proof}[Proof of Proposition 1]
These are classical results of optimal transport, see e.g. Theorems 6.9 and 6.18 in \citet{villani2009optimal}. The only new point is to characterize the space $\mathcal{P}_1((\mathcal{M}_B(\mathbb{X}), \bl))$ of random measures with finite first moment in the sense of random objects on a metric space. For the first moment, we can use $\E(\bl(\crm, \bar{\mu}))$ for some deterministic given measure $\bar{\mu}$, see Definition 6.4 in \citet{villani2009optimal}. We choose $\bar{\mu} = 0$, as we have $\bl(\mu,0) = \mu(\mathbb{X})$ for any $\mu$, so that $\E(\bl(\crm, \bar{\mu})) = \E(\crm(\mathbb{X}))$. The conclusion follows. 
\end{proof}

\begin{proof}[Proof of Proposition 2]
Fix $\varepsilon > 0$.
Let $f$ be $1$-bounded and $1$-Lipschitz such that $\bl(\E(\tilde{\mu}^1), \E(\tilde{\mu}^2)) - \varepsilon \leq \int f \, \ddr \E(\tilde{\mu}^2) - \int f \, \ddr \E(\tilde{\mu}^1)$. Then as $\bl(\tilde{\mu}^1,\tilde{\mu}^2) \geq \int f \, \ddr \tilde{\mu}^2 - \int f \, \ddr \tilde{\mu}^1$ a.s.,
for any coupling between $\tilde{\mu}^1$ and $\tilde{\mu}^2$, 
\begin{equation*}
\E (\bl(\tilde{\mu}^1,\tilde{\mu}^2)) \geq \E \left( \int_\mathbb{X} f \, \ddr \tilde{\mu}^2 - \int_\mathbb{X} f \, \ddr \tilde{\mu}^1 \right) =\int_\mathbb{X} f \, \ddr \E(\tilde{\mu}^2) - \int_\mathbb{X} f \, \ddr \E(\tilde{\mu}^1) \geq \bl(\E(\tilde{\mu}^1), \E(\tilde{\mu}^2)) - \varepsilon. 
\end{equation*}
Taking the infimum over couplings of $\tilde{\mu}^1$ and $\tilde{\mu}^2$ and then sending $\varepsilon \to 0$ we get the conclusion.
\end{proof}

\section*{Proofs of Section 3}

\begin{proof}[Proof of Proposition 3]
Campbell's formula yields, for any Borel set $A \subseteq \mathbb{X}$, 
\begin{equation*}
\iint_{(0,+\infty) \times A} s \, \ddr \nu(s,x) = \E(\crm(A)) < + \infty. 
\end{equation*}
By applying it for $A = \mathbb{X}$, we see that the measure $\ddr \bar \nu(s,y) = s \, \ddr \nu(s,x)$
is finite on the product space $(0,+\infty) \times \mathbb{X}$. Thus its disintegration $\ddr \bar \nu(s,x) = \ddr \varrho_x(s) \ddr Q(x)$ is well-defined, with $Q$ the second marginal measure of $\bar \nu$ on $\mathbb{X}$ and $\varrho$ a probability kernel. Again by Campbell's formula, it is easy to see that $Q(A) = \E( \crm(A) )$ for any Borel set $A$. We define
\begin{equation*}
P_0(\cdot) = \frac{Q(\cdot)}{Q(\mathbb{X})}, \qquad \ddr \rho_x(s) = Q(\mathbb{X}) \frac{\ddr \varrho_x(s)}{s},
\end{equation*}
which entails the sought decomposition: $P_0$ is a probability measure and since $\varrho$ is a probability kernel, it is clear that for $P_0$-a.e. $x$ there holds
\begin{equation*}
\int_0^{+ \infty} s \, \ddr \rho_x(s) = Q(\mathbb{X}) \int_0^{+ \infty}  \ddr \varrho_x(s) = Q(\mathbb{X}) = \E( \crm(\mathbb{X}) ). \qedhere
\end{equation*}
\end{proof}

We introduce some notation and auxiliary results for the remaining proofs of this section.
We denote by $\Leb$ the Lebesgue measure on $(0, +\infty)$ and $\Leb^2$ the Lebesgue measure on $(0, +\infty) \times  (0, +\infty)$. We start with a preliminary lemma which mimics \citet[Proposition 2.2]{Santambrogio2015}.

\begin{lemma}
\label{th:inverse_tail}
Let $\rho$ a positive measure on $(0, + \infty)$ with tail integral $U_\rho (\cdot) = \rho(\cdot,+\infty)$ and let $U_\rho^{-1}$ denote its generalized inverse, defined as $U_\rho^{-1}(s) = \inf \{ u \geq 0 \ : \ U_\rho(u) \leq s \}$. 
Then the restriction to $(0, + \infty)$ of $U_\rho^{-1} \# \Leb$ is equal to $\rho$.
\end{lemma}

\begin{proof}
We only need to evaluate $U_\rho^{-1} \# \Leb$ on sets $(a,+ \infty)$ for $a > 0$, and show that it coincides with $U_\rho(a)$. Indeed, $(U_\rho^{-1} \# \Leb) (a, + \infty)$ is equal to
\[
\Leb( \{ s \in (0, + \infty ) \ : \ U_\rho^{-1}(s) > a \} )  = \Leb ( \{ s \in (0, + \infty ) \ : \ s < U_\rho(a) \} )  = U_\rho(a), 
\]
since $U_\rho^{-1}(s) > a$ if and only if $s < U_\rho(a)$, which follows from the definition of $U_\rho^{-1}$.
\end{proof}

\begin{lemma}
\label{th:build_coupling}
Let $\nu^i= \ddr \rho^i_{x}(s) \ddr P_0^i(x)$ for $i=1,2$ two Lévy intensities and $\pi \in \Pi(P_0^1, P_0^2)$. Then there exists a measure $\overline{\pi} = \overline{\pi}(s_1,x_1,s_2,x_2)$ on $([0,+\infty) \times \mathbb{X})^2 \setminus (\{ 0 \} \times \mathbb{X})^2$ with three properties: i) the projection $\overline{\pi}_{12}$ on the first two marginals, when restricted to $(0,+\infty) \times \mathbb{X}$, coincide with $\nu^1$; ii) the projection $\overline{\pi}_{34}$ on the last two marginals, when restricted to $(0,+\infty) \times \mathbb{X}$, coincide with $\nu^2$; iii) the coupling $\overline{\pi}$ satisfies 
\begin{multline*}
\E_{(X,Y) \sim \pi} \left( \frac{M_1(\nu^1) + M_1(\nu^2)}{2}  \dXt(X,Y) + \W_{*}(\rho_X^1,\rho_Y^2) \right)  \\ =
\int_{([0,+\infty) \times \mathbb{X})^2 \setminus (\{ 0 \} \times \mathbb{X})^2} \left( \frac{s_1+s_2}{2} \dXt (x_1,x_2) + |s_1-s_2| \right) \; \ddr \overline{\pi} (s_1,x_1,s_2,x_2). 
\end{multline*}

\end{lemma}

\begin{proof}
We define $\overline{\pi}$ as follows. Consider $(X,Y) \sim \pi$. Then for any positive measurable function $f : ([0,+\infty) \times \mathbb{X})^2 \to \R$ with $f(0,x_1,0,x_2) = 0$ for all $x_1,x_2 \in \mathbb{X}$ we define 
\begin{equation}
\label{eq:wass_coupling}
\int f \, \ddr \overline{\pi} = \E_{X,Y} \left( \int_0^{+ \infty} f \left( U_{\rho^1_X}^{-1}(s), X,  U_{\rho^2_Y}^{-1}(s),Y \right) \,  \ddr s \right).
\end{equation}
This defines a positive measure on $([0,+\infty) \times \mathbb{X})^2 \setminus (\{ 0 \} \times \mathbb{X})^2$. We use \eqref{eq:wass_coupling} to prove the three properties. First, for every $f(s_1,x_1, s_2,x_2) = g(s_1,x_1)$ such that $g(0,x_1) = 0$, by Lemma~\ref{th:inverse_tail} and since the first marginal of $\pi$ is $P_0^1$,
\[
\int f \ddr \overline{\pi} = \mathbb{E}\bigg(\int g(U_{\rho^1_X}^{-1}(s), X) \, \ddr s\bigg) = \int g(s,x) \, \ddr \rho_{x}(s) \ddr P_0^1(x) = \int g \, \ddr \nu^1.
\]
This proves i) and similar techniques may be used for ii). As for iii), we consider $f(s_1,x_1,s_2,x_2) =  (s_1+s_2)  \dXt (x_1,x_2)/2 + |s_1-s_2|$: with the help of Lemma~\ref{th:inverse_tail},  
\begin{align*}
\int f \, \ddr \overline{\pi} & =  \E \left( \int_0^{+ \infty} \left( \frac{U_{\rho^1_X}^{-1}(s) + U_{\rho^2_Y}^{-1}(s)}{2} \dXt (X,Y) + | U_{\rho^1_X}^{-1}(s) - U_{\rho^2_Y}^{-1}(s) | \right) \ddr s \right)  \\
& = \E \left( \frac{\dXt (X,Y)}{2} \left( \int_0^{+ \infty} s \, \ddr \rho_X^1(s) + \int_0^{+ \infty} s \, \ddr \rho_Y^2(s) \right)  \right) + \E \left( \int_0^{+ \infty} | U_{\rho^1_X}^{-1}(s) - U_{\rho^2_Y}^{-1}(s) | \, \ddr s \right)  \\
& = \frac{M_1(\nu^1) + M_2(\nu^2)}{2} \E(\dXt (X,Y)) + \E \left( \int_0^{+ \infty} | U_{\rho^1_X}^{-1}(s) - U_{\rho^2_Y}^{-1}(s) | \, \ddr s \right). 
\end{align*}
To conclude we use the following identity, which comes from Fubini's theorem and whose proof is analogue to the one of \citet[Proposition 2.17]{Santambrogio2015}: 
\begin{align*}
 \int_0^{+ \infty} | U_{\rho^1_X}^{-1}(s) - U_{\rho^2_Y}^{-1}(s) | \, \ddr s & = \Leb^2 \left\{ (s,u) \ : \ \min(U_{\rho_X^1}(s),U_{\rho_Y^2}(s)) \leq u \leq \max(U_{\rho_X^1}(s),U_{\rho_Y^2}(s)) \right\}  \\
& =  \int_0^{+ \infty} | U_{\rho^1_X}(s) - U_{\rho^2_Y}(s) | \, \ddr s = \W_*(\rho^1_X,\rho^2_Y). 
\end{align*}
Taking the expectation with respect to $X,Y$ and plugging in the previous formula yields the result.
\end{proof}

We are now ready to move on to the proofs of the main results of this section. 

\begin{proof}[Proof of Theorem 4]
The proof uses Lemma~\ref{th:inverse_tail} and Lemma~\ref{th:build_coupling} stated and proved above. The main idea is to use the coupling between L\'evy intensities given by Lemma~\ref{th:build_coupling} to build a coupling between the corresponding CRMs. 

We fix $\varepsilon > 0$. Consider $\pi \in \Pi(P^1_0, P^2_0)$ an $\varepsilon$-optimal coupling with respect to the infimum in the definition of $\ad$. We look at $\overline{\pi}$ the measure on $([0,+\infty) \times \mathbb{X})^2 \setminus (\{ 0 \} \times \mathbb{X})^2$ given by Lemma~\ref{th:build_coupling}. In particular, from (iii) of Lemma~\ref{th:build_coupling}, 
\begin{equation*}
\int_{([0,+\infty) \times \mathbb{X})^2\setminus (\{ 0 \} \times \mathbb{X})^2} \left( \frac{s_1+s_2}{2} \dXt (x_1,x_2) + |s_1-s_2| \right) \; \ddr \overline{\pi} (s_1,x_1,s_2,x_2) \leq \ad(\nu^1 , \nu^2 ) + \varepsilon.
\end{equation*} 

We consider $\tilde{\mathcal{N}}$ a Poisson random measure on $([0,+\infty) \times \mathbb{X})^2 \setminus (\{ 0 \} \times \mathbb{X})^2$  with intensity $\overline{\pi}$. Then, because of the marginal properties of $\overline{\pi}$,
\begin{equation*}
\tilde{\mu}^1 \eqd  \int  s_1 \delta_{x_1} \, \ddr \tilde{\mathcal{N}}(s_1,x_1,s_2,x_2), \qquad \tilde{\mu}^2 \eqd \int  s_2 \delta_{x_2} \, \ddr \tilde{\mathcal{N}}(s_1,x_1,s_2,x_2).
\end{equation*} 
We have built a coupling between the two CRMs $\tilde{\mu}^1$ and $\tilde{\mu}^2$, we now have to estimate $\E(\bl(\crm^1,\crm^2))$. 

For any $(s_1,x_1,s_2,x_2)$ and any $f : \mathbb{X} \to \R$, easy algebra yields
\begin{equation*}
s_2 f(x_2) - s_1 f(x_1)
 = \frac{s_2 - s_1}{2} (f(x_2) + f(x_1)) + \frac{s_2+s_1}{2} (f(x_2) - f(x_1)).
\end{equation*}
Thus if $f$ is $1$-bounded and $1$-Lipschitz with respect to $\dXt$, we easily find
\begin{equation*}
s_2 f(x_2) - s_1 f(x_1) \leq |s_2 - s_1| + \frac{s_1+s_2}{2} \dXt(x_1,x_2). 
\end{equation*}
Integrating with respect to $\tilde{\mathcal{N}}$, a.s. we have 
\begin{equation*}
\int f \, \ddr \crm^2 - \int f \, \ddr \crm^1 \leq  \int \left( |s_2 - s_1| + \frac{s_1+s_2}{2} \dXt(x_1,x_2) \right) \, \ddr \tilde{\mathcal{N}}(s_1,x_1,s_2,x_2).   
\end{equation*}
As this is valid for any $f$ which is $1$-bounded and $1$-Lipschitz, the right hand side is an upper bound on $\bl(\crm^1,\crm^2)$. Thus, taking expectation, from Campbell's formula we find, 
\begin{align*}
\W_{\bl}(\law(\crm^1), \law(\crm^2)) & \leq \E(\bl(\crm^1,\crm^2)) \\
& \leq \E \left( \int \left( |s_2 - s_1| + \frac{s_1+s_2}{2} \dXt(x_1,x_2) \right) \, \ddr \tilde{\mathcal{N}}(s_1,x_1,s_2,x_2) \right) \\
& = \int_{([0,+\infty) \times \mathbb{X})^2 \setminus (\{ 0 \} \times \mathbb{X})^2} \left( \frac{s_1+s_2}{2} \dXt (x_1,x_2) + |s_1-s_2| \right) \; \ddr \overline{\pi} (s_1,x_1,s_2,x_2) \\
& \leq \ad(\nu^1 , \nu^2 ) + \varepsilon,
\end{align*}
where the last equality comes from the choice of $\overline{\pi}$. We conclude as $\varepsilon$ is arbitrary.
\end{proof}

\begin{proof}[Proof of Proposition 5]
First, note that $\W_*$ is clearly a distance: symmetry is immediate, the triangle inequality is a consequence of the triangle inequality in $L^1((0, + \infty), \Leb)$, and identifiability comes from the fact that $U_{\rho^1} = U_{\rho^2}$ a.e. implies $\rho^1 = \rho^2$.

We now fix $m > 0$ and turn to the adapted discrepancy over $\mathcal{M}_m((0, + \infty) \times \mathbb{X})$. Symmetry is also immediate. 

For the triangle inequality, consider $\nu^1, \nu^2$ and $\nu^3$. If $\pi^{12}$ and $\pi^{23}$ are couplings respectively in $\Pi(P_0^1, P_0^2)$ and $\Pi(P_0^2, P_0^3)$, by the gluing Lemma~\cite[Chapter 1]{villani2009optimal}, we can find $(X,Y,Z)$ such that $(X,Y) \sim \pi^{12}$ and $(Y,Z) \sim \pi^{23}$. Then we see that a.s.
\begin{equation*}
m \dXt(X,Z) + \W_{*}(\rho^{1}_X, \rho^3_Z) \leq m \dXt(X,Y) + m \dXt(Y,Z)  +  \W_{*}(\rho^{1}_X, \rho^2_Y) + \W_{*}(\rho^{2}_Y, \rho^3_Z)  
\end{equation*}
thanks to the triangle inequality for $\dXt$ and $\W_{*}$. Note here the importance that $m = m_i = M(\nu^i)$ for $i=1,2,3$. Indeed, in general, we cannot conclude that $\frac{m_1+m_3}{2} \dXt(X,Z) \leq \frac{m_1+m_2}{2} \dXt(X,Y) + \frac{m_2+m_3}{2} \dXt(Y,Z)$ only with the triangle inequality for $\dXt$, e.g., if $m_2 > \max(m_1,m_3)$. Then, taking the expectation,
\begin{align*}
\ad(\nu^1, \nu^3) & \leq \E \left( m \dXt(X,Z) + \W_{*}(\rho^{1}_X, \rho^3_Z) \right) \\
& \leq \E \left( m \dXt(X,Y) + \W_{*}(\rho^{1}_X, \rho^2_Y) \right) + \E \left( m \dXt(Y,Z) + \W_{*}(\rho^{2}_Y, \rho^3_Z) \right). 
\end{align*}
Taking the infimum in $\pi^{12}$ and $\pi^{23}$ in the right hand side yields the triangle inequality.

Eventually, we have to check identifiability, that is, what happens if $\ad(\nu^1, \nu^2) = 0$. For that we can use Theorem 4: indeed $\ad(\nu^1, \nu^2) = 0$ implies $\W_\bl(\law(\crm^1),\law(\crm^2)) = 0$, that is, $\law(\crm^1) = \law(\crm^2)$. Thus the Laplace functional of $\crm^1$ and $\crm^2$ are the same, which means that their Lévy intensities $\nu^1$ and $\nu^2$ must coincide, see~\eqref{eq:characterization_CRM} below.
\end{proof}

\begin{proof}[Proof of Theorem 6]
We use the convexity property of the value of the optimal transport cost: indeed for every coupling $(\tilde{\nu}^1,\tilde{\nu}^2)$,
\begin{equation*}
\W_\bl(\law(\crm^1),\law(\crm^2)) \leq \E_{(\tilde{\nu}^1,\tilde{\nu}^2)} \left( \W_\bl(\law(\crm^1|\tilde{\nu}^1),\law(\crm^2|\tilde{\nu}^2) \right) \leq \E_{(\tilde{\nu}^1,\tilde{\nu}^2)}(\ad(\tilde{\nu}^1, \tilde{\nu}^2)),
\end{equation*}
where the first inequality is \citet[Theorem 4.8]{villani2009optimal} and the second inequality is Theorem 4. Taking then the infimum over all coupling between $\tilde{\nu}^1$ and $\tilde{\nu}^2$ yields
\begin{equation*}
\W_\bl(\law(\crm^1),\law(\crm^2)) \leq \W_\ad(\law(\tilde{\nu}^1),\law(\tilde{\nu}^2)). \qedhere 
\end{equation*}
\end{proof}

\section*{Proofs of Section 4}

We recall now that the Lévy intensity $\nu$ characterizes the Laplace functional of a CRM $\crm$, in the sense that $\crm \sim \mathrm{CRM}(\nu)$ if and only if, for every $f : \mathbb{X} \to \mathbb{R}$ measurable and bounded,
\begin{equation}
\label{eq:characterization_CRM}
\log \E \left( - \int_{\mathbb{X}} f(x) \, \ddr \crm(x) \right) = - \int_{(0,+\infty) \times \mathbb{X}} (1- e^{-sf(x)}) \, \ddr \nu(s,x).
\end{equation}
We will denote by $T \#\rho$ the pushforward of a measure $\rho$ by a measurable map $T$, defined by $(T \# \rho)(A) = \rho(T^{-1}(A))$ for any measurable set $A$.

Before the main theorems about identifiability of normalization, we first state a preliminary result that can be seen as a special case of \citet[Lemma 7]{Catalano2020}. We report its proof for completeness.

\begin{lemma}
\label{th:crm_divided_constant}
Let $\crm$ be a CRM with L\'evy intensity $\rho_x(\ddr s) \ddr P_0(x)$ and let $k>0$ be a constant. Then $\crm_k (A):= \crm (A)/k$ is a CRM with L\'evy intensity $\ddr (S_k \# \rho_x)(  s) \, \ddr P_0(x)$, being $S_k : s \in (0,+\infty) \to s/k \in (0,+\infty)$ the scaling by $1/k$. 

In particular, if $\ddr \rho_x( s) = \rho_x(s) \ddr s$ has a density with respect to the Lebesgue measure, $\crm_k$ has L\'evy intensity $k \rho_x(k s) \, \ddr s \, \ddr P_0(x)$.
\end{lemma}

\begin{proof}
Clearly, $\crm_k$ has independent evaluations on disjoint sets. Its Laplace functional satisfies
\begin{equation*}
\mathbb{E} \left( \exp \left( - \int_{\mathbb{X}} f(x) \, \ddr \tilde{\mu}_k(x) \right) \right)
= \exp\left\{- \int_{(0,+\infty) \times \mathbb{X}} ( 1- e^{-s \, \frac{f(x)}{k}}) \ddr \rho_x( s)   \ddr P_0(x))\right\},
\end{equation*}
which shows the conclusion in the general case by definition of the pushforward measure and characterization~\eqref{eq:characterization_CRM} of the Lévy intensity. For the case where $\rho_x(\ddr s) = \rho_x(s) \ddr s$  we do the change of variables $t = s/k$. 
\end{proof}

We now move to the proofs of the main results. 

\begin{proof}[Proof of Theorem 7]
The proof uses Lemma~\ref{th:crm_divided_constant} stated and proved above. As written in the main text, it is inspired by a result on the marginal laws of subordinators by David Aldous and Stevan Evans \cite[Lemma 7.5]{PitmanYor1992}. The first and second steps below can be seen as adaptation of their ideas to our context as they characterize the (marginal) law of $\tilde \mu(A)$ up to a scaling constant, possibly depending on $A$. Further work is needed to show that the characterization holds for the entire law of $\crm$, which, by independence of the evaluations on disjoint sets, amounts to showing that the constant does not depend on $A$, i.e., that there exists $\alpha>0$ such that $\crm^1(A) \eqd \alpha \crm^2(A)$ for every Borel set $A$.  We write $\ddr \nu^i(s,x) = \ddr \rho_x^i( s) \ddr P_{0}^i(x)$ for the L\'evy intensity of $\crm^i$, as in Proposition 3. 

\medskip

\emph{Preliminary step: a note on null-sets.} By the infinite activity assumption it follows that $\crm^i(A) = 0$ a.s. if and only if $P_0^i(A) = 0$. In particular $\crm^i(A)/\crm^i(\mathbb{X})=0$ a.s. if and only if $P_0^i(A) = 0$. Since $\crm^1/ \crm^1(\mathbb{X}) \eqd \crm^2/ \crm^2(\mathbb{X})$, we conclude that $P^1_0, P_0^2$ are equivalent, that is, they have the same null set. Moreover, still by the infinite activity assumption, if $P^1_0(A) > 0$, then $0< \crm^i(A) < +\infty$ a.s., for $i=1,2$.

\medskip

\emph{First step: defining auxiliary Cox CRMs.}
Let $A \subseteq \mathbb{X}$ a subset with $P_0^1(A) > 0$ and $A^c$ its complement. For every $i=1,2$ we can define a random measure on $A^c$ as
\[
\tilde \xi^i(\cdot) = \frac{\crm^i(\cdot)}{\crm^i(A)}.
\]
Because of the independence of the evaluations on disjoint sets, conditionally on $\crm^i(A)$, $\tilde \xi^i$ is a CRM on $A^c$. Thus, $\tilde \xi^i$ is a Cox CRM as in Definition 5. Moreover, we claim that $\tilde \xi^1 \eqd \tilde \xi^2$. Indeed, by writing $\tilde p^i =\crm^i/ \crm^i(\mathbb{X})$, there holds a.s., for any sets $B_1, B_2, \ldots, B_n$ in $A^c$,
\begin{equation*}
(\tilde{\xi}^i(B_1), \ldots, \tilde{\xi}^i(B_n)) = \left( \frac{\crm^i(B_1)}{\crm^i(A)}, \ldots, \frac{\crm^i(B_n)}{\crm^i(A)} \right) = \left( \frac{\tilde p^i(B_1)}{\tilde p^i(A)}, \ldots, \frac{\tilde p^i(B_n)}{\tilde p^i(A)} \right),
\end{equation*}  
and the right hand side has clearly the same law for $i=1,2$ as $\tilde p^1 \eqd \tilde p^2$.

\medskip

\emph{Second step: deducing information on Lévy intensities.}
By Lemma~\ref{th:crm_divided_constant}, conditionally on $\crm^i(A)$, $\tilde{\xi}^i$ is a CRM with L\'evy intensity $\ddr (S_{\crm^i(A)} \# \rho_x^i)(s) \, \ddr P_{0}^i(x)$. Here $S_k : (0,+\infty) \to (0,+ \infty)$ is defined by $S_k(s) = s/k$, it is the scaling by $1/k$. Thus $\ddr \tilde \nu^i(s,x) = \ddr (S_{\crm^i(A)} \# \rho_x^i)(s) \, \ddr P_{0}^i(x)$ is the random L\'evy intensity of the Cox CRM $\xi^i$. By standard results on Cox processes \cite[Theorem 3.3]{Kallenberg2017}, the law of the random L\'evy intensity characterizes the process, which implies
\begin{equation*}
\ddr (S_{\crm^1(A)} \# \rho_x^1)(s) \, \ddr P_{0}^1(x) \eqd \ddr (S_{\crm^2(A)} \# \rho_x^2)(s) \, \ddr P_{0}^2(x).
\end{equation*}
We then multiply both sides by $s$ and integrate over $(0, + \infty) \times A^c$. By definition of the pushforward, and since $\rho_x^i$ have first moment equal to $\E(\crm^i(\mathbb{X}))$ (see Proposition 3), we deduce that 
\begin{equation*}
\frac{P_0^1(A^c) \E(\tilde{\mu}^1(\mathbb{X}))}{\crm^1(A)} \eqd \frac{P_0^2(A^c) \E(\tilde{\mu}^2(\mathbb{X}))}{\crm^2(A)}.  
\end{equation*}
Both numerators are deterministic, so we conclude $\tilde \mu^1(A) \eqd \alpha_A \tilde{\mu}^2(A)$, at least if $P_0^1(A)$ and $P_0^1(A^c)$ do not vanish (cf. preliminary step). To conclude the proof, we need to show that $\alpha_A$ does not depend on $A$ and treat the case where $P_0^1(A)=0$ or $P_0^1(A^c)=0$. 

\medskip

\emph{Third step: showing that $\alpha_A$ does not depend on $A$.}
First we take $A$ such that $\alpha_A$ and $\alpha_{A^c}$ are well defined, that is, both $P^1_0(A)$ and $P^1_0(A^c)$ are strictly positive, and we show that $\alpha_A = \alpha_{A^c}$. Note that such a set $A$ exists as $\crm^1$ is non-degenerate. By the discussion of the first step:
\begin{equation*}
\frac{\crm^1(A^c)}{\crm^1(A)} = \tilde \xi^1(A^c)  \eqd \tilde{\xi}^2(A^c) = \frac{\crm^2(A^c)}{\crm^2(A)}.
\end{equation*}
Moreover since $\crm^i$ is a CRM, $(\crm^1(A), \crm^1(A^c)) \eqd (\alpha_A \crm^2(A), \alpha_{A^c} \crm^2(A^c))$, and both correspond to a random vector with independent components. We deduce that 
\begin{equation*}
\frac{\crm^1(A^c)}{\crm^1(A)} \eqd \frac{\alpha_{A^c}}{\alpha_{A}} \frac{\crm^2(A^c)}{\crm^2(A)},
\end{equation*}
and this necessarily yields $\alpha_{A^c}= \alpha_A$. It follows that a.s.
\begin{equation*}
\crm^1(\mathbb{X}) = \crm^1(A) + \crm^1(A^c)  \eqd \alpha_A \left( \crm^2(A) + \crm^2(A^c) \right) = \alpha_A \crm^2(\mathbb{X}).
\end{equation*}
This implies $\alpha_A=\alpha_{\mathbb{X}}$, which thus does not depend on $A$. We denote it by $\alpha$. 

Eventually we have to consider the case $P^1_0(A) = 0$ or $P^1_0(A^c) = 0$. If $P^1_0(A) = 0$, by the preliminary step also $P^2_0(A) = 0$, so that both $\crm^1(A)$ and $\crm^2(A)$ are a.s. equal to $0$. In particular, $\crm^1(A) = \alpha  \crm^2(A) = 0$ a.s. On the other hand, if $P^1_0(A^c) = 0$, then $P^2_0(A^c) = 0$, and thus $\crm^1(A) = \crm^1(\mathbb{X}) \eqd \alpha \crm^2(\mathbb{X})$ while $\crm^2(A) = \crm^2(\mathbb{X})$ a.s., which completes the proof.
\end{proof}

\begin{proof}[Proof of Lemma 9]
The proof follows by Lemma~\ref{th:crm_divided_constant} by taking $k = \mathbb{E}(\crm(\mathbb{X}))$. 
\end{proof}

\begin{proof}[Proof of Theorem 10]
We follow the same line of proof as Theorem 7, but with additional steps to handle the new layers of randomness. We write $\ddr \tilde \nu^i(s,x) = \ddr \tilde \rho_x^i( s) \ddr \tilde P_{0}^i(x)$ for the random L\'evy intensity of $\crm^i$, as in Proposition 3. We define the real-valued random variable $\tilde q^i = \E( \crm^i(\mathbb{X}) | \tilde \nu^i )$, in such a way that a.s.
\begin{equation*}
\int_0^{+ \infty} s \, \ddr \tilde \rho_x^i( s) = \tilde{q}^i,
\end{equation*} 
as written in Proposition 3. We recall that $S_k : (0,+\infty) \to (0,+ \infty)$ is defined by $S_k(s) = s/k$, i.e., it is the scaling by $1/k$. We only need to prove that the Lévy intensities of the scaled Cox CRM are the same, that is, given Lemma~\ref{th:crm_divided_constant}, we want to prove the following equality in distribution of random measures:
\begin{equation*}
\ddr (S_{ \tilde{q}^1 } \# \tilde \rho_x^1)( s) \ddr \tilde P_{0}^1(x) \eqd \ddr (S_{ \tilde{q}^2 } \#  \tilde \rho_x^2)( s) \ddr \tilde P_{0}^2(x).
\end{equation*}
We call a measurable set $A$ \emph{active} if $\tilde{P}^i_0(A) > 0$ a.s. for $i=1,2$. As the CRMs are infinitely active but with finite mean, it implies $0 < \crm^i(A) < + \infty$ a.s. for $i = 1,2$. Moreover, by assumption of the theorem, if a set $A$ is not active then $\tilde{P}^1_0(A) = \tilde{P}^2_0(A) = 0$ a.s. and thus $\crm^1(A) = \crm^2(A) = 0$ a.s.

\medskip

\emph{First step: starting as in Theorem 7}. Let $A_1, A_2, \ldots, A_m$ a collection of active subset. We write $A = \bigcup_{k=1}^m A_k$ and we define on $A^c$ the random vector of measures $\tilde{\bm{\xi}}^i$ as
\begin{equation*}
\tilde{\bm \xi}^i(\cdot) = \left( \frac{\crm^i(\cdot)}{\crm^i(A_k)} \right)_{k=1, \ldots, m}.
\end{equation*}
As in the first step of the proof of Theorem 7, we claim that $\tilde{\bm{\xi}}^1 \eqd \tilde{\bm{\xi}}^2$. Indeed, for $i=1,2$ the law of $(\tilde{\bm{\xi}}^i(B_1), \ldots, \tilde{\bm{\xi}}^i(B_m))$ depends only on the joint law of $ \crm^i(B_l) / \crm^i(A_k)$ for $k =1, \ldots, m$ and $l = 1, \ldots, n$, and the latter is determined by the law of $ \crm^i(\cdot) / \crm^i(\mathbb{X})$. 

Moreover, conditionally to $(\tilde{\nu}^i, \crm^i(A_1), \ldots, \crm^i(A_k))$, the random vector of measure $\tilde{\bm \xi}^i$ is a completely random vector (see~\cite{Catalano2020} for the precise definition), and its (random and vector-valued) Lévy intensity is $(\ddr ( S_{\crm^i(A_k)} \# \tilde \rho_x^i )(s) \, \ddr \tilde P_0^i(x))_{k=1, \ldots, m}$. As in the proof of Theorem 7, invoking \citet[Theorem 3.3]{Kallenberg2017}, we have equality in law of the random Lévy intensities, that is,
\begin{equation*}
\left( \ddr ( S_{\crm^1(A_k)} \# \tilde \rho_x^1 )(s) \, \ddr \tilde P_0^1(x) \right)_{k=1,\ldots,m} \eqd \left( \ddr ( S_{\crm^2(A_k)} \# \tilde \rho_x^2 )(s) \, \ddr \tilde P_0^2(x) \right)_{k=1, \ldots, m}.
\end{equation*} 
We take this equality and integrate it over $(0,+\infty) \times B$ where $B \subseteq A^c$. We use the normalization of $\tilde{\rho}^i_x$ recalled above and the definition of the pushforward measure to conclude to the following equality in distribution over $\R^m$:
\begin{equation*}
\left( \frac{\tilde{q}^1 \tilde P_0^1(B) }{\crm^1(A_k)} \right)_{k=1, \ldots, m} \eqd \left( \frac{\tilde{q}^2 \tilde P_0^2(B) }{\crm^2(A_k)} \right)_{k=1, \ldots, m},
\end{equation*}
for all active sets $A_1, \ldots, A_m$ and $B$ and $A_k \cap B = \emptyset$ for any $k=1,\ldots,m$.

\medskip

\emph{Second step: doing the same reasoning on a new Cox CRM}. Fix an active set $B$, it guarantees that $\tilde P_0^i(B) > 0$ a.s. for $i=1,2$. We define the random measures $\tilde \zeta^i$ on $B^c$ as
\begin{equation*}
\tilde{\zeta}^i(\cdot) = \frac{\crm^i(\cdot)}{\tilde{q}^i \tilde{P}^i_0(B)}.
\end{equation*}
Conditionally to $\tilde{\nu}^i$ (recall that $\tilde{q}^i$ is $\sigma(\tilde{\nu}^i)$-measurable) these are CRM, thus they are Cox CRM and their random Lévy intensity is $\ddr (S_{\tilde{q}^i \tilde{P}^i_0(B)} \# \tilde \rho_x^i)(s) \, \ddr \tilde{P}^i_0(x)$. Moreover, the previous step showed that $\tilde{\zeta}^1(\cdot) \eqd \tilde{\zeta}^2(\cdot)$ as random measures. Indeed, we showed $(\tilde{\zeta}^1(A_1), \ldots, \tilde{\zeta}^1(A_m)) \eqd (\tilde{\zeta}^2(A_1), \ldots, \tilde{\zeta}^2(A_m))$ if $A_1, \ldots, A_m$ are active. This equality in distribution easily extends if $A_k$ is not active for some $k$, as it implies $ \tilde{\zeta}^i(A_k) = \crm^i(A_k) = 0$ a.s. for $i=1,2$. Thus, invoking again \citet[Theorem 3.3]{Kallenberg2017},
\begin{equation}
\label{eq:zz_aux_eqd}
\ddr (S_{\tilde{q}^1 \tilde{P}^1_0(B)} \# \tilde \rho_x^1)(s) \, \ddr \tilde{P}^1_0(x) \eqd \ddr (S_{\tilde{q}^2 \tilde{P}^2_0(B)} \# \tilde \rho_x^2)(s) \, \ddr \tilde{P}^2_0(x).
\end{equation}

\smallskip

\emph{Third step: deducing equality in law of the scaled random Lévy intensities}. We take the previous equality in distribution, multiply by $s$, and integrate it over $(0,+\infty) \times B^c$. Using again the normalization of $\tilde{\rho}_x^i$, we know conclude
\begin{equation*}
\frac{\tilde{P}^1_0(B^c)}{\tilde{P}^1_0(B)} \eqd \frac{\tilde{P}^2_0(B^c)}{\tilde{P}^2_0(B)}. 
\end{equation*}
As $\tilde{P}^i_0(B) + \tilde{P}^i_0(B^c) = 1$, we deduce that $\tilde{P}^1_0(B) \eqd \tilde{P}^2_0(B)$ if $B$ is active. Actually we deduce more: for $i=1,2$, $\tilde{P}^i_0(B)$ is a deterministic function of $\ddr (S_{\tilde{q}^i \tilde{P}^i_0(B)} \# \tilde \rho_x^i)(s) \, \ddr \tilde{P}^i_0(x)$. Note that $\ddr (S_{ \tilde{q}^i } \# \tilde \rho_x^i)( s) \ddr \tilde P_{0}^i(x)$ is obtained from $\ddr (S_{ \tilde{q}^i \tilde{P}^i_0(B) } \# \tilde \rho_x^i)( s) \ddr \tilde P_{0}^i(x)$ from a scaling of $s$ by $\tilde{P}^i_0(B)$. Thus $\ddr (S_{ \tilde{q}^i } \# \tilde \rho_x^i)( s) \ddr \tilde P_{0}^i(x)$ is a deterministic function of $\ddr (S_{\tilde{q}^i \tilde{P}^i_0(B)} \# \tilde \rho_x^i)(s) \, \ddr \tilde{P}^i_0(x)$. Combined with~\eqref{eq:zz_aux_eqd} we deduce 
\begin{equation*}
\ddr (S_{ \tilde{q}^1 } \# \tilde \rho_x^1)( s) \ddr \tilde P_{0}^1(x) \eqd \ddr (S_{ \tilde{q}^2 } \#  \tilde \rho_x^2)( s) \ddr \tilde P_{0}^2(x)
\end{equation*}
at least on $(0,+\infty) \times B^c$, where $B$ is an active set. To conclude the proof we only need to prove that equality in distribution is valid on $(0,+\infty) \times \mathbb{X}$, that is, take $B = \emptyset$.

\medskip

\emph{Fourth step: extending to the whole space}. Recall that we assumed that $\E(\tilde{P}^i_0)$ is not purely atomic. By a theorem of Sierpinski we can find a non-increasing sequence of sets $(B_n)$ with $\E(\tilde{P}^1_0)(B_n) > 0$ for all $n$ but $\E(\tilde{P}^1_0)(\cap_n B_n) = 0$. In particular, we deduce that all sets $B_n$ are active. Thus the measures $\ddr (S_{ \tilde{q}^i } \# \tilde \rho_x^i)( s) \ddr \tilde P_{0}^i(x)$ have the same law for $i=1,2$ on $(0,+ \infty) \times B_n^c$ for any $n$. By a monotone convergence argument, it implies that it is also the case on $(0,+ \infty) \times (\cup_n B_n^c)$. Eventually, as $\ddr (S_{ \tilde{q}^i } \# \tilde \rho_x^i)( s) \ddr \tilde P_{0}^i(x)$ gives a.s. measure zero to $(0,+\infty) \times \cap_n B_n$ for $i=1,2$ we conclude that the law is the same on the whole space $(0,+\infty) \times \mathbb{X}$ and we are done.
\end{proof}

\section*{Proofs of Section 5.1 (Dirichlet process)}

Before starting the proofs, we recall some convexity properties of the Wasserstein distance. A useful property that we often use is the convexity inequality
\begin{equation}
\label{eq:convex}
\mathcal{W}_{\dXt}( \lambda P^1 + (1-\lambda) Q^1,  \lambda P^2 + (1-\lambda) Q^2) \le \lambda \mathcal{W}_{\dXt}( P^1 ,   P^2) + (1-\lambda) \mathcal{W}_{\dXt}( Q^1 ,   Q^2),
\end{equation}
which holds for all $\lambda \in (0,1)$ and $P^1,P^2,Q^1,Q^2$ probability distributions. Moreover, if $P^1 = P^2 = P$ then there is equality, that is,  
\begin{equation}
\label{eq:convex_specific}
\mathcal{W}_{\dXt}( \lambda P + (1-\lambda) Q^1,  \lambda P + (1-\lambda) Q^2) =  (1-\lambda) \mathcal{W}_{\dXt}( Q^1 ,   Q^2).
\end{equation}
This last equality case is specific to the Wasserstein distance of order 1 and it can be proved using the dual formulation of optimal transport, that is, equation (4) of the main text. 

In order to prove Lemma 12, we start with an auxiliary result that we will also use for the normalized Gamma generalized process. 

Informally, we recall that a random variable is infinitely divisible if its probability distribution can be expressed of the sum of an arbitrary number of i.i.d. random variables. One can prove that for every measure $\rho$ on $(0,+\infty)$ such that $\int (s^2 \wedge 1) \ddr \rho(s) <+\infty$, a random variable $J$ such that
\[
\log\big(\mathbb{E}\big(e^{-\lambda J}\big)\big) = - \int (1-e^{-\lambda s}) \, \ddr \rho(s)
\]
is infinitely divisible. We call $J$ an infinitely divisible with L\'evy measure $\rho$ and refer to \cite{Sato1999} for a complete account on infinite divisibility.

\begin{lemma}
\label{th:id_jumps}
Let $\crm \sim \textup{CRM}(\nu)$ on a Polish space $\mathbb{X}$ and let $J$ be an infinitely divisible random variable on $(0,+\infty)$ independent of $\crm$ with L\'evy measure $\rho$. Then for every  $y \in \mathbb{X}$, $\crm + J \delta_y \sim \textup{CRM}(\nu + \rho \otimes \delta_y)$.
\end{lemma}

\begin{proof}
For any measurable function $f$ on $\mathbb{X}$, since $J$ and $\crm$ are independent,
\begin{align*}
\mathbb{E}  & \left( \exp \left( - \int f(x) \ddr (\crm +J \delta_y ) (x) \right) \right)   =\mathbb{E} \left(  \exp \left( - \int f(x) \ddr \crm (x) \right) \right) \mathbb{E} \left( \exp \left( - f(y) J \right) \right) \\
& \qquad = \exp \left( - \int_0^{\infty} \int_{\mathbb{X}} (1-e^{-s f(x)}) \ddr \nu(s,x)\right) \exp \left( - \int_0^{\infty} (1-e^{-s f(y)}) \ddr \rho(s) \right)  \\
& \qquad = \exp \left( - \int_0^{\infty} \int_{\mathbb{X}} (1-e^{-s f(x)}) \ddr (\nu(s,x) + \rho(s) \delta_y(x)) \right).
\end{align*}
Since CRMs are characterized by their Laplace functional, see~\eqref{eq:characterization_CRM}, this proves the result.
\end{proof}

\begin{proof}[Proof of Lemma 12]
The proof uses Lemma~\ref{th:id_jumps} stated and proved above. 

By specializing Theorem 11 to a gamma CRM, we know that there exists a latent random variable $U$ such that
\[
\crm^{*}|U \eqd \crm_{U} + \sum_{i=1}^k J^{U}_{i} \delta_{x_i^*},
\]
where $\crm_{U}$ is a gamma CRM with L\'evy intensity $s^{-1}e^{-(U+1)s} \alpha P_0$ and the jumps $J_{i} ^{U} \sim \text{gamma}(n_i, U + 1)$, where $n_i = \#\{j : x_j = x_i^*\}$ is the number of observations equal to $x_i^*$ and everything is independent. Since the gamma distribution is infinitely divisible, thanks to Lemma~\ref{th:id_jumps} this implies that $\crm^{*}|U$ is a gamma CRM with canonical L\'evy intensity 
\[
\ddr \nu_{U} (s, x) =  (\alpha + n) \frac{e^{-(U+1)s}}{s} \ddr s \bigg( \frac{\alpha}{\alpha +n} \ddr P_0(x) + \frac{1}{\alpha +n} \sum_{i=1}^n \delta_{x_i}(x) \bigg).
\]
Thus $\crm^{*}$ is a Cox CRM that is conditionally a gamma CRM. By Example 3 its corresponding Cox scaled CRM $\crm^{*} \s$ has L\'evy intensity
\[
\ddr \tilde \nu^{i} \s (s, x)|{U} = (\alpha + n) \frac{e^{(\alpha + n)s}}{s} \bigg( \frac{\alpha}{\alpha +n} \ddr P_0(x) + \frac{1}{\alpha +n} \sum_{i=1}^n \delta_{X_i}(x)\bigg). 
\]
In particular, we observe that by scaling the conditional L\'evy intensity all the randomness disappears and we are left with a deterministic L\'evy intensity. 
\end{proof}

\begin{proof}[Proof of Proposition 13]
Given Lemma 12, as both Lévy intensities are homogeneous, we can use Remark 2 and thus $\wad(\mathcal{L}(\tilde \nu^1 \s),\mathcal{L}(\tilde \nu^2 \s)) = \ad (\nu^{1*} \s, \nu^{2*} \s) = \mathcal{J} + \mathcal{A}$ with 
\begin{align*}
\mathcal{J} & = \W_* \left( (\alpha_1 + n) \frac{e^{-(\alpha_1+n)s}}{s} \ddr s, (\alpha_2 + n) \frac{e^{-(\alpha_2+n)s}}{s} \ddr s   \right), \\
\mathcal{A} &= \W_{\dXt} \bigg(\frac{\alpha_1}{\alpha_1 + n} P_{0}^1 +\frac{1}{\alpha_1 + n} \sum_{i=1}^n \delta_{x_i} , \frac{\alpha_2}{\alpha_2 + n} P_{0}^2 +\frac{1}{\alpha_2 + n} \sum_{i=1}^n \delta_{x_i}\bigg).
\end{align*}
The expression of $\mathcal{J}$ is then further expanded using Definition 3 and a linear change of variables, which gives the expression of the statement.
\end{proof}

To move forward and prove Proposition 14, the main difficulty is the analysis of the jump part. For this we prove an auxiliary result which will also be useful for the proof of Proposition 20. Indeed we develop a general framework for the study of the jump component of the distance $\ad$ on CRMs. 

\begin{proposition}
\label{prop:tail_integral_comp}
Let $U(\theta,t)$ a continuous function $[0,+\infty) \times (0,+\infty) \to \R$ such that the following assumptions hold:
\begin{enumerate}
\item $\mathcal{I}(t) = \int_0^{+ \infty} U(\theta,t) \, \ddr t$ does not depend on $\theta$. \label{lab:not_depend}
\item For any $\theta > 0$, there exists one single point $t_\theta$ such that $U(0,t_\theta) = U(\theta,t_\theta)$. \label{lab:single_inter}
\item The function $D(\theta,t) = \frac{\dr U}{\dr \theta}(\theta,t)$ is defined everywhere, continuous and satisfies the integrability condition:
\begin{equation*}
\forall \Theta > 0, \quad  \int_0^{+ \infty} \sup_{0 \leq \theta \leq \Theta} |D(\theta,t)| \ddr t \; < + \infty.
\end{equation*}
Moreover, for any $\theta \geq 0$ the function $D(\theta,\cdot)$ vanishes at a single point $s_\theta$. \label{lab:integrability}
\item For $t \geq \max(t_\theta, s_\theta)$ (or equivalently for $t \leq \min(t_\theta, s_\theta)$) the functions $U(\theta,t)- U(0,t)$ and $D(\theta,t)$ have the same sign. \label{lab:same_sign}
\end{enumerate}
Then the function
\begin{equation*}
F : \theta \mapsto \int_0^{+ \infty} |U(\theta,t) - U(0,t)| \ddr t
\end{equation*}
is strictly increasing, differentiable everywhere and
\begin{equation}
\label{eq:aux_parametric_bound_derivative}
F'(\theta) \leq \int_0^{+ \infty} |D(\theta,t)| \ddr t,
\end{equation}
with equality when $\theta = 0$.
\end{proposition}

\begin{figure}
\begin{center}
\includegraphics[scale=1.]{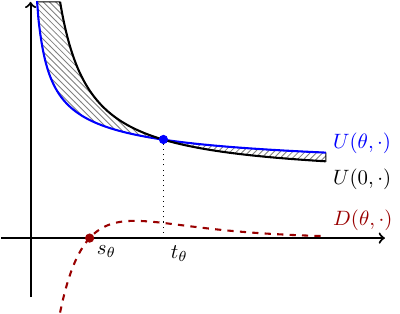}
\end{center}
\caption{Illustration of the geometric situation depicted in Proposition~\ref{prop:tail_integral_comp} with the functions $U(\theta,\cdot)$, $U(\theta,0)$ and the derivative $D(\theta, \cdot) = \frac{\dr U}{\dr \theta}(\theta,\cdot)$. The function $F(\theta)$ corresponds to the (unsigned) area between the two curves. By Assumption~\ref{lab:not_depend} the areas before and after $t_\theta$ are the same, so $F(\theta)$ is twice the area before $t_\theta$, and also twice the area after $t_\theta$.}
\label{fig:help_proof_intersection}
\end{figure}

\begin{remark}
In our examples $U(\theta, \cdot) = \rho_{\theta}(\cdot,+\infty)$ is the tail integral of a measure $\rho_{\theta}$. Assumption~\ref{lab:not_depend} can then be verified through the identity $\int_0^{+ \infty} U(\theta,t) \ddr t = \int_0^{+ \infty} s \, \ddr \rho_\theta(s)$, as we consider scaled CRMs.
\end{remark}

\begin{proof}
Up to changing $U$ into $-U$, we can assume that $U(\theta,t) - U(0,t)$ is negative on $(0,t_\theta)$ and positive on $(t_\theta,+ \infty)$. We refer to Figure~\ref{fig:help_proof_intersection} for an illustration. We can write 
\begin{equation*}
F(\theta) = \int_0^{t_\theta} ( U(0,t) - U(\theta,t) ) \ddr t + \int_{t_\theta}^{+ \infty} ( U(\theta,t) - U(0,t) ) \ddr t.
\end{equation*}
As $\int_0^{+ \infty} U(0,t) \ddr t = \int_0^{+ \infty} U(\theta,t) \ddr t$, thanks to Assumption~\ref{lab:not_depend} we can rewrite
\[
F(\theta) = 2 \int_0^{t_\theta} ( U(0,t) - U(\theta,t) ) \ddr t  =  2 \int_{t_\theta}^{+ \infty} ( U(\theta,t) - U(0,t) ) \ddr t.
\]
When we take the derivative (thanks to the integrability Assumption~\ref{lab:integrability} we can differentiate under the $\int$ sign), using that the integrand vanishes exactly at $t_\theta$, we find three different expressions of the derivative:
\[
F'(\theta )  = - \int_0^{t_\theta} D(\theta,t)  \ddr t + \int_{t_\theta}^{+ \infty} D(\theta,t) \ddr t = -  2 \int_0^{t_\theta} D(\theta,t)  \ddr t = 2 \int_{t_\theta}^{+ \infty} D(\theta,t) \ddr t.
\]
Recall that $s_\theta$ is the point where $D(\theta,t)$ vanishes. By Assumption~\ref{lab:same_sign}, we know that $D(\theta,t)$ is negative on $(0,s_\theta)$ and positive on $(s_\theta,+ \infty)$. If $ s_\theta \leq t_\theta$, using the last expression for $F'$ yields $F'(\theta) > 0$. On the other hand, if $t_\theta \leq s_\theta$, then it is by using the second expression that we find $F'(\theta) > 0$. This concludes the proof that $F$ is strictly increasing.

Eventually the first expression of the derivative yields our bound~\eqref{eq:aux_parametric_bound_derivative} when we take the absolute value with the help of the triangle equality. 

For the equality in~\eqref{eq:aux_parametric_bound_derivative} when $\theta = 0$, we first justify that $t_\theta$ converges to $s_0$ as $\theta \to 0$. Fix $t_1 < s_0 < t_2$. As $D(0,\cdot)$ changes sign around $s_0$, then $U(\theta,t_1) - U(0,t_1) \sim \theta D(0,t_1)$ and $U(\theta,t_2) - U(0,t_2) \sim \theta D(0,t_2)$ take different signs for $\theta$ small enough. By continuity of $U(\theta, \cdot)$ we deduce that $t_1 < t_\theta < t_2$ if $\theta$ is small enough. This is enough to show that $t_\theta$ converges to $s_0$ as $\theta \to 0$. Thus, taking the limit in the first expression of $F'$ we find
\begin{equation*}
\lim_{\theta \to 0} F'(\theta)  = - \int_0^{s_0} D(0,t)  \ddr t + \int_{s_0}^{+ \infty} D(0,t) \ddr t = \int_0^{+ \infty} |D(0,t)| \ddr t, 
\end{equation*}
as $D(0,\cdot)$ is negative on $(0,s_0)$ and positive on $(s_0,+\infty)$. This is enough to justify that $F'(0)$ exists and is equal to the value on the right hand side.
\end{proof}

We then specialize this result to the tail functions of the Lévy intensities of the Dirichlet process. 

\begin{proposition}
\label{prop:J_dp}
Let $F : [1,+\infty) \to \R$ defined by
\begin{equation*}
F(\alpha) = \int_0^{+\infty} |  \Gamma(0, t) - \alpha \Gamma(0,\alpha t) | \ddr t.
\end{equation*}
The function $F$ is increasing and $F(\alpha) \leq c \log(\alpha)$, for $c = \int_0^{+ \infty} |\Gamma(0,t) - e^{-t}| \ddr t$. Moreover, $F$ is differentiable with $F'(1) = c$.
\end{proposition}

\begin{proof}
We use Proposition~\ref{prop:tail_integral_comp} with $\theta = \alpha - 1$ and $U(\theta,t) = \alpha \Gamma(0,\alpha t)$.

We claim that there is a unique $t_\alpha \in (0, + \infty)$ with $\Gamma(0, t) = \alpha \Gamma(0,\alpha t)$. Moreover,  $\Gamma(0, t) - \alpha \Gamma(0,\alpha t)$ is negative on $(0,t_\alpha)$ and positive on $(t_\alpha, + \infty)$. Indeed let us call $f : t \mapsto t \Gamma(0,t)$ and $\bar t$ its point of maximal, so that the equation can be rewritten as $f(t) = f(\alpha t)$. Standard analysis yields that $f$ is strictly increasing between $0$ and $\bar t$, and strictly decreasing after. Quite easily, this yields that $f(t) < f(\alpha t)$ when $0 < t \leq \bar t/\alpha$, while $f(t) > f(\alpha t)$ when $t \geq \bar t$. Eventually, on $(\bar t/\alpha, \bar t)$ the function $f(\alpha t) - f(t)$ is strictly decreasing and continuous, and thus it vanishes at exactly one point.

Moreover, 
\begin{equation*}
\frac{\dr}{\dr \alpha} (\alpha \Gamma(0,\alpha t)) = \Gamma(0,\alpha t) - \exp(- \alpha t),
\end{equation*}
and one can check that this function also vanishes at a single point, is negative after that point and positive before. For the integrability Assumption~\ref{lab:integrability} note that for any $\alpha \geq 1$,
\begin{equation*}
\left| \Gamma(0,\alpha t) - \exp(- \alpha t) \right| \leq \Gamma(0,t) + \exp(-t), 
\end{equation*}
and the right hand side integrates to $2 < + \infty$ by Fubini's theorem. Thus by Proposition~\ref{prop:tail_integral_comp} we directly see that $F$ is increasing with
\begin{equation*}
F'(\alpha) \leq \int_{0}^{+ \infty} |\Gamma(0,\alpha t) - \exp(- \alpha t)| \ddr t = \frac{1}{\alpha} \int_0^{+ \infty} |\Gamma(0,t) - e^{-t}| \ddr t,
\end{equation*}  
and equality when $\alpha = 1$. Integrating this inequality between $1$ and $\alpha$ yields the upper bound. 
\end{proof}

We finally have all the tools to prove Proposition 14. 

\begin{proof}[Proof of Proposition 14]
The proof uses Proposition~\ref{prop:J_dp} stated and proved above. Indeed to prove 1., 2. and 3. we notice, using the notations of Proposition~\ref{prop:J_dp} and a simple change of variable that
\begin{equation*}
\mathcal{J} = \int_0^{+\infty} | (\alpha_1 +n) \Gamma(0, (\alpha_1 +n) s) - (\alpha_2 +n) \Gamma(0,(\alpha_2 + n) s) | \ddr s = F \left( \frac{\alpha_2 + n}{\alpha_1 + n} \right).
\end{equation*}
Thus 1., 2. and 3. are a simple rewriting of Proposition~\ref{prop:J_dp}, using in particular the Taylor expansion $F(\alpha) = c (\alpha - 1) +  o(\alpha - 1)$ as $F'(1) = c$. 

The proof of 4. relies on the behavior of the Wasserstein distance with respect to convex combination.
Let $\lambda_i = \alpha_i (\alpha_i + n)^{-1}$, so that $0< \lambda_1 < \lambda_2< 1$ for every $n$. Then using the convexity inequality~\eqref{eq:convex} followed by~\eqref{eq:convex_specific}, 
\begin{align*}
\mathcal{A}  & \le \lambda_1 \W_{\dXt}(P_0^1, P_0^2) + (1-\lambda_1) \W_{\dXt} \bigg( \frac{1}{n} \sum_{i=1}^n \delta_{x_i}, \frac{\lambda_2 - \lambda_1}{1 - \lambda_1} P_0^2 + \frac{1 - \lambda_2 }{1 - \lambda_1} \frac{1}{n} \sum_{i=1}^n \delta_{x_i} \bigg) \\
& = \lambda_1 \W_{\dXt}(P_0^1, P_0^2) + (\lambda_2 - \lambda_1) \W_{\dXt} \bigg(\frac{1}{n} \sum_{i=1}^n \delta_{x_i}, P_0^2 \bigg).
\end{align*}
The proof of the upper bound follows by substituting the values of $\lambda_1$ and $\lambda_2$. Moreover, if $\alpha_1 = \alpha_2$ or $P_0^1 = P_0^2$, we can reduce it to an equality thanks to~\eqref{eq:convex_specific}.
\end{proof}

\begin{proof}[Proof of Theorem 15]
With the notations of Proposition 13, and with Proposition 14, it is clear that $\mathcal{J} \asymp 1/n$ and $\mathcal{J} \leq c(\alpha_2-\alpha_1)/n$. On the other hand, 
using $\alpha_1 / (\alpha_1 + n) \leq \alpha_1/n$, as well as $n(\alpha_2 - \alpha_1)/(\alpha_1 + n)(\alpha_2 + n) \leq (\alpha_2 - \alpha_1)/n$ and that $\dXt$ (and thus $\W_{\dXt}$) is bounded by $2$, we have 
\begin{equation*}
\frac{1}{n} \lesssim \mathcal{J} + \mathcal{A} \leq \frac{\alpha_1}{n} \W_{\dXt}(P^1_0,P^2_0) + \left( c+2 \right) \frac{\alpha_2 - \alpha_1}{n}, 
\end{equation*}
at least if $\alpha_1 \leq \alpha_2$.
The conclusion $\ad(\nu^{1*} \s, \nu^{2*} \s)  \asymp 1/n$ follows. 
Then $\wbl(\mathcal{L}(\crm^{1*} \s), \mathcal{L}(\crm^{2*} \s))  \lesssim 1/n$ follows from Theorem~4.
\end{proof}

\section*{Proofs of Section 5.2 (Normalized generalized gamma process)}

\begin{proof}[Proof of Proposition 16]
The proof follows the same lines as Lemma 12. By specializing the formula for the jumps $J_i^U$ in Theorem 11 we see that they follow a gamma distribution. Since the gamma distribution is infinitely divisible, thanks to Lemma~\ref{th:id_jumps} $\crm^{*}|U$ is characterized by the L\'evy intensity
\[
\nu_U^* (s,x) = \frac{\alpha}{\Gamma(1-\sigma)} \frac{e^{-(U+1)s}}{s^{1+\sigma}} \ddr s \, \ddr P_0 (x) + \frac{e^{-(U +1)s}}{s}    \sum_{i=1}^k (n_i - \sigma) \, \ddr s \delta_{x_i^*}(x).
\]
With simple calculations we derive
\[
\E(\crm^*(\mathbb{X})|U) =  \frac{ \alpha (U+1)^{\sigma}+ n-k \sigma}{U+1}.
\]
Denote $C=C_n =  \alpha (U+1)^{\sigma}+ n-k \sigma$. Then by Lemma 9 the corresponding Cox scaled CRM $\crm^* \s$ has random L\'evy intensity
\[
\tilde \nu^* \s(\ddr s, \ddr x) = \frac{\alpha}{\Gamma(1-\sigma)} \bigg(\frac{U+1}{C} \bigg)^{\sigma} \frac{e^{-C s}}{s^{1 + \sigma}} \ddr s \, \ddr P_0 (x) + \frac{e^{-Cs}}{s}   \sum_{i=1}^k (n_i - \sigma) \, \ddr s \delta_{x_i^*}(x).
\]
We now find its canonical expression. Starting from the previous formula, by integrating out $f(s) = s$ we find the mean measure of $\crm^* \s$ to be
\[
P_0^* = \frac{\alpha ( U +1)^\sigma}{C} P_0 + \frac{1}{C}\sum_{i=1}^k (n_i - \sigma) \delta_{x_i^*}.
\]
In general it seems not trivial to find the canonical decomposition of a sum of independent homogeneous CRMs, which overall is a non-homogeneous CRM. In this case, however, we can use the fact that $P_0$ is nonatomic, so that every component of the sum has different support. Thus, the disintegration of $\tilde \nu^* \s$ with respect to $P_0^*$ is
\[
\frac{C^{1-\sigma}}{\Gamma(1-\sigma)} \frac{e^{-Cs}}{s^{1+\sigma}} \mathbbm{1}_{\mathbb{X}\setminus\{x_1^*,\dots,x_k^*\}}(x) \, \ddr s +C \frac{e^{-Cs}}{s}  \mathbbm{1}_{\{x_1^*,\dots,x_k^*\}}(x) \, \ddr s.
\]
Moreover, the expression of the distribution of $U$ can be easily derived with Equation~(18) of the main text.  
\end{proof}

\begin{proof}[Proof of Proposition 17]
As $\crm^2$ is a Gamma CRM, the expression of $\tilde{\nu}^{2*}_\mathcal{S}$ can be found in Lemma 12. In particular it is deterministic and homogeneous, thus with Remark 2, 
\begin{align*}
\W_{\ad}(\law(\tilde{\nu}^{1*}_\mathcal{S}),\law(\tilde{\nu}^{2*}_\mathcal{S})) & = \E \left( \ad(\tilde{\rho}^{*1}_x(s) \ddr s \ddr \tilde{P}^{1*}_0(x), \rho^{*2}(s) \ddr s \ddr P^{2*}_0(x)) \right) \\
& = \E \left( \W_{\dXt}(\tilde{P}^{1*}_0, P^{2*}_0) + \E_{X \sim \tilde{P}^{1*}_0}( \W_*(\tilde{\rho}^{*1}_X, \rho^{2*})) \right).
\end{align*}
The first term is easily computed thanks to the explicit expressions of Lemma 12 and Proposition 16:
\begin{align*}
\E \left( \W_{\dXt}(\tilde{P}^{1*}_0, P^{2*}_0) \right) & = \E_U \left( \W_{\dXt} \bigg(\frac{\alpha ( U +1)^\sigma}{C} P_0 + \frac{1}{C}\sum_{i=1}^k (n_i - \sigma) \delta_{x_i^*}, \frac{\alpha}{\alpha + n} P_0 + \frac{1}{\alpha+n} \sum_{i=1}^k n_i \delta_{x_i^*} \bigg) \right) \\
&= \E_U(\mathcal{A}^U).
\end{align*}
For the second term, we use that $P_0$ is non-atomic and thus
\begin{equation*}
\E \left(\E_{X \sim \tilde{P}^{1*}_0}( \W_*(\tilde{\rho}^{*1}_X, \rho^{2*})) \right) = \E_U \left( \frac{\alpha (1+U)^\sigma}{C} \E_{X \sim P_0} \W_*(\tilde{\rho}^{*1}_X, \rho^{2*}) + \frac{1}{C} \sum_{i=1}^k (n_i - \sigma) \W_*(\tilde{\rho}^{*1}_{x^*_i}, \rho^{2*}) \right).
\end{equation*}
Given the explicit expressions for $\tilde{\rho}^{*1}_x$ and $\rho^{2*}$, we see that this coincides with $\E_U(\mathcal{J}^U_1 + \mathcal{J}^U_2)$. 
\end{proof}

We move to the proof of Theorem 18, on the asymptotic behavior of the variables $U_n$. We rely on the following lemma to that end.

\begin{lemma}
\label{lemma:aux_proof_latent}
Let $(Y_n)_{n \in \N}$ a sequence of random variables on $I = (a, + \infty)$ with densities proportional to $\exp(-f_n)$, for some sequence $(f_n)_{n \in \N}$ of functions from $I$ to $\R$. 
We assume:
\begin{enumerate}
\item Each $f_n$ is smooth, convex, and minimized at a single point $r_n \in I$, with the sequence $(r_n)_{n \in \N}$ diverging to $+ \infty$;
\item 
There exist $0 < \alpha < \beta < + \infty$ and $\gamma > 1/2$ such that, for every $n$ large enough, for $x \in [r_n - r_n^\gamma, r_n + r_n^\gamma]$, $\alpha \leq r_n f_n''(x) \leq \beta$.
\end{enumerate}
Then, for the topology of $L^1$ convergence there holds
\begin{equation*}
\lim_{n \to + \infty} \frac{Y_n}{r_n} = 1. 
\end{equation*}
\end{lemma}

\begin{remark}
The second condition quantifies $f_n'' \asymp 1/r_n$ on a large enough neighborhood of $r_n$. This inequality, together with the convexity of $f_n$, enables to show that $Y_n$ indeed concentrates around $r_n$. Such condition will be satisfied in our case, and the scaling $f_n''(r_n) \asymp 1/r_n$ comes from the logarithms in the expression of $f_n$. 
\end{remark}

\begin{proof}
Without loss of generality, we can assume $\gamma < 1$ so that $\gamma \in (1/2,1)$. 

\medskip

\emph{First step: simplifying the setting.}
We need to prove that 
\begin{equation*}
a_n = \frac{1}{r_n} \left( \int_I \left| x-r_n \right| \exp(-f_n(x)) \, \ddr x \right) \left( \int_I  \exp(-f_n(x)) \, \ddr x \right)^{-1}
\end{equation*}
converges to $0$ as $n \to + \infty$. We first shift the function $f_n$ by defining $g_n(x) = f_n(x-r_n) - f_n(x_n)$, which is still convex but now minimized at $g_n(0) = 0$. The bound in 2. now reads, for $x \in [- r_n^\gamma, r_n^\gamma]$, $\alpha \leq r_n g_n''(x) \leq \beta$. Moreover $g_n$ is defined on $I_n = (a-r_n,+ \infty)$.
Thus we can rewrite
\begin{equation*}
a_n = \frac{1}{r_n} \left( \int_{I_n} |x| \exp(-g_n(x)) \, \ddr x \right) \left( \int_{I_n}  \exp(-g_n(x)) \, \ddr x \right)^{-1}.
\end{equation*}
We set $J_n = [-r_n^\gamma, r_n^\gamma]$ and introduce 
\begin{equation*}
b_n = \int_{J_n}  \exp(-g_n(x)) \, \ddr x, \qquad c_n = \int_{I_n \setminus J_n}  |x| \exp(-g_n(x)) \, \ddr x. 
\end{equation*}
We see clearly that $\int_{J_n} |x| \exp(-g_n(x)) \, \ddr x \leq r_n^\gamma b_n$. Thus, bounding from below the denominator of $a_n$ by $b_n$,
\begin{equation*}
a_n\leq \frac{\int_{I_n} \left| x \right| \exp(-g_n(x)) \, \ddr x}{r_n b_n} = \frac{\int_{J_n} \left| x \right| \exp(-g_n(x)) \, \ddr x + \int_{I_n \setminus J_n} \left| x \right| \exp(-g_n(x)) \, \ddr x}{r_n b_n} \leq \frac{r_n^\gamma}{r_n} + \frac{c_n}{r_n b_n}.
\end{equation*}
The first term converges to $0$ because $\gamma < 1$, while we have more work to do to prove the convergence to zero of the second term. 

\medskip

\emph{Second step. Lower bound on $b_n$.} On $J_n = [- r_n^\gamma, r_n^\gamma]$, using $g_n(0) = 0$ together with the bound $g_n'' \leq \beta/r_n$, we find
\begin{equation*}
g_n(x) \leq \frac{\beta}{2 r_n} x^2.
\end{equation*}
Plugging this back in the expression for $b_n$,
\begin{equation*}
b_n \geq \int_{- r_n^\gamma}^{r_n^\gamma} \exp \left( - \frac{\beta }{2 r_n} x^2 \right) \,  \ddr x = \sqrt{\frac{2 \pi r_n}{\beta}} \mathrm{erf} \left( \sqrt{\frac{\beta}{2}} r_n^{\gamma - 1/2} \right),
\end{equation*}
being $\mathrm{erf}$ the error function. Sending $n$ to $+ \infty$, as $r_n^{\gamma - 1/2} \to + \infty$, we see that $b_n \gtrsim \sqrt{r_n}$ and 
\begin{equation*}
\liminf_{n \to + \infty}  \frac{b_n}{\sqrt{r_n}} \geq \sqrt{\frac{2 \pi}{\beta}}.
\end{equation*}

\medskip

\emph{Third step. Upper bound on $c_n$}. We need to analyze $c_n$. By symmetry we can focus on the integral over $[r_n^\gamma, + \infty)$. We first notice that, again thanks to the bounds on $g_n''$,
\begin{equation*}
g_n'(r_n^\gamma) \geq \alpha r_n^{\gamma - 1}, \qquad g_n(r_n^\gamma) \geq \frac{\alpha}{2} r_n^{2 \gamma - 1}. 
\end{equation*} 
As the function $g_n$ is convex, thus above its tangents, we deduce that for $x \geq r_n^\gamma$,
\begin{equation*}
g_n(x) \geq \frac{\alpha}{2} r_n^{2 \gamma - 1} + \alpha r_n^{\gamma - 1} (x - r_n^\gamma). 
\end{equation*} 
We plug this bound in the integral we want to compute to obtain, using the change of variables $y = \alpha r_n^{\gamma-1} (x-r_n^\gamma)$, 
\begin{align*}
\int_{r_n^\gamma}^{+ \infty}  x  \exp(-g_n(x))  \, \ddr x & \leq  \int_{r_n^\gamma}^{+ \infty}  x \exp\left( - \frac{\alpha}{2} r_n^{2 \gamma - 1} - \alpha r_n^{\gamma - 1} (x - r_n^\gamma)\right) \, \ddr x \\
& = \exp\left( - \frac{\alpha}{2} r_n^{2 \gamma - 1} \right) \int_{r_n^\gamma}^{+ \infty} x \exp\left( - \alpha r_n^{\gamma - 1} (x - r_n^\gamma) \right) \, \ddr x \\
& = \frac{\exp\left( - \frac{\alpha}{2} r_n^{2 \gamma - 1} \right)}{\alpha r_n^{\gamma - 1}}  \int_{0}^{+ \infty}  \left( \frac{y}{\alpha r_n^{\gamma - 1}} + r_n^{\gamma} \right) \exp(-y) \, \ddr y \\
& = \frac{\exp\left( - \frac{\alpha}{2} r_n^{2 \gamma - 1} \right)}{\alpha r_n^{\gamma - 1}} \left( \frac{1}{\alpha r_n^{\gamma - 1}} + r_n^\gamma \right). 
\end{align*}
By symmetry, $c_n$ is bounded by twice the quantity above.

\medskip

\emph{Fourth step. Putting all estimates together.} Plugging this back together we obtain
\begin{equation*}
a_n \lesssim  r_n^{\gamma - 1} + \frac{1}{r_n^{3/2}} \frac{\exp\left( - \frac{\alpha}{2} r_n^{2 \gamma - 1} \right)}{\alpha r_n^{\gamma - 1}} \left( \frac{1}{\alpha r_n^{\gamma - 1}} + r_n^\gamma \right)  . 
\end{equation*}
Sending $n \to + \infty$, the first term in the sum converges to $0$, as well as the second one as the exponential factor $\exp\left( - \frac{\alpha}{2} r_n^{2 \gamma - 1} \right)$ dominates everything as $2 \gamma - 1 > 0$.  
\end{proof}

\begin{proof}[Proof of Theorem 18]
The proof uses Lemma~\ref{lemma:aux_proof_latent} stated and proved above.

Let $h(U) = (1+U)^\sigma$. With a change of variable we obtain that the density of $h(U)$ is proportional to $\exp(-f_n(x))$, with
\[
f_n(x) =  - (k_n-1) \log(x) + \frac{\alpha}{\sigma}  x - (n-1) \log \bigg( 1- \frac{1}{x ^{1/\sigma}} \bigg) ,
\]
for $x>1$, where the subscript $k_n = k$ is introduced to underline the dependence on $n$. We will show that $h(U)$ concentrates around the point of minimum of $f_n$ through Lemma~\ref{lemma:aux_proof_latent}. A direct computation yields
\begin{align*}
f_n'(x) &= - \frac{k_n-1}{x} + \frac{\alpha}{\sigma} -  \frac{n-1 }{\sigma (x^{1+1/\sigma} - x)};\\
f_n''(x) &= \frac{k_n-1}{x^2} + (n-1)  \frac{(1+1/\sigma) x^{1/\sigma} - 1}{\sigma (x^{1+1/\sigma} - x)^2}.
\end{align*}
The last expression tells us that the function $f_n$ is indeed convex as $f_n'' \geq 0$. Moreover, $f_n$ is minimized at a single point which is a zero of $f_n'$. Thus, it is minimized at a solution $r_n$ of 
\begin{equation}
\label{eq:equation_defining_minimum}
\frac{k_n-1}{x} +  \frac{n-1}{\sigma(x^{1+1/\sigma} - x)} = \frac{\alpha}{\sigma}.
\end{equation} 

As the left hand side is the sum of non-negative terms increasing in $n$, $r_n \to + \infty$ and 
\begin{equation}
\label{eq:equation_information_a_priori}
r_n \gtrsim \max \left( k_n, n^{\sigma/(1+\sigma)} \right). 
\end{equation}
We find the asymptotic behavior of $r_n$ depending on $k_n$. We distinguish three cases.

\emph{First case}. If $k_n \ll n^{\sigma/(1+\sigma)}$, from~\eqref{eq:equation_information_a_priori} then we deduce
\begin{equation*}
\frac{k_n-1}{r_n} \ll 1.
\end{equation*}
Thus we can neglect the first term in~\eqref{eq:equation_defining_minimum} and conclude that
\begin{equation*}
\frac{n-1 }{\sigma(r_n^{1+1/\sigma} - r_n)} \sim \frac{\alpha}{\sigma} \qquad \Leftrightarrow \qquad r_n \sim \left( \frac{n}{\alpha} \right)^{\sigma/(1+\sigma)}.
\end{equation*} 

\emph{Second case}. On the other hand if $k_n \gg n^{\sigma/(1+\sigma)}$, then from~\eqref{eq:equation_information_a_priori} we deduce
\begin{equation*}
\frac{k_n-1}{r_n} \gg \frac{n-1 }{\sigma(r_n^{1+1/\sigma} - r_n)},
\end{equation*}
and thus we can neglect the second term in~\eqref{eq:equation_defining_minimum} and conclude that
\begin{equation*}
\frac{k_n-1}{r_n} \sim \frac{\alpha}{\sigma} \qquad \Leftrightarrow \qquad r_n \sim \frac{\sigma k_n}{\alpha}.
\end{equation*} 

\emph{Third case}. If $k_n \sim \lambda n^{\sigma/(1+\sigma)}$, then both terms in the left hand side of~\eqref{eq:equation_defining_minimum} are of the same order. Specifically, the reduced variable $\bar r_n = x_n/n^{\sigma/(1+\sigma)}$ will converge to a solution $\bar r$ of 
\begin{equation*}
\frac{\lambda}{\bar{r}} + \frac{1}{\sigma \bar{r}^{1+ 1/\sigma}} = \frac{\alpha}{\sigma}. 
\end{equation*}

To apply Lemma~\ref{lemma:aux_proof_latent} we need to control of the derivative $f_n''$ in each of the three cases. As $r_n \to + \infty$, $f_n''$ is asymptotically equivalent, say for $x \in [1/2r_n, 2 r_n]$, to
\begin{equation}
\label{eq:equation_second_derivative_simplified}
\frac{k_n}{x^2} + \frac{(1+1/\sigma) n}{\sigma x^{2+1/\sigma}}.
\end{equation}
We show that the last expression is asymptotically equivalent to $1/r_n$, say for $x \in [1/2r_n, 2 r_n]$. This will allow us to use Lemma~\ref{lemma:aux_proof_latent} for any value of $\gamma \in (1/2,1)$. 

\emph{First case}. If $k_n \ll n^{\sigma/(1+\sigma)}$ the first term in~\eqref{eq:equation_second_derivative_simplified} is negligible compared to the second one which behaves like 
\begin{equation*}
\frac{n}{r_n^{2+1/\sigma}} \asymp \frac{1}{n^{\sigma/(1+\sigma)}} \asymp \frac{1}{r_n}.
\end{equation*}

\emph{Second case}. If $k_n \gg n^{\sigma/(1+\sigma)}$ the second term in~\eqref{eq:equation_second_derivative_simplified} is negligible compared to the first one which behaves like $1/r_n$.

\emph{Third case}. If $k_n \sim \lambda n^{\sigma/(1+\sigma)}$ the both terms are of the same order of magnitude, and once again they both behave like $1/r_n$. 
\end{proof}

\begin{proof}[Proof of Theorem 19]
We study each term in the sum $\E(\mathcal{J}_1^U) + \E(\mathcal{J}_2^U) + \E(\mathcal{A}^U)$ sequentially, where all expectations are taken with respect to the random variable $U$. Specifically, we will show that
\[
\E(\mathcal{J}_1^U) \asymp \max(n^{-1/(1+\sigma)}, k/n ), \qquad \E(\mathcal{J}_2^U), \E(\mathcal{A}^U) \lesssim \max(n^{-1/(1+\sigma)}, k/n ).
\]

\emph{First term}. We recall that
\begin{equation*}
\mathcal{J}_1^U = \frac{\alpha (1+U)^ \sigma }{C} \W_{*}\bigg( \frac{C^{1-\sigma}}{\Gamma(1-\sigma)} \frac{e^{-Cs}}{s^{1+\sigma}} \ddr s,(\alpha+ n)\frac{e^{-(\alpha+ n)s}}{s} \ddr s\bigg).
\end{equation*}
For the extended Wasserstein distance, with a change of variable $t = C s $,
\begin{align*}
\W_{*}\bigg( \frac{C^{1-\sigma}}{\Gamma(1-\sigma)} \frac{e^{-Cs}}{s^{1+\sigma}} \ddr s,(\alpha+ n)\frac{e^{-(\alpha+ n)s}}{s} \ddr s\bigg) 
& = \int_0^{+\infty} \bigg| \frac{ C \Gamma(-\sigma, C s)}{\Gamma(1-\sigma)} - (\alpha+n) \Gamma(0, (\alpha+n) s) \bigg| \ddr s \\
& = \int_0^{+\infty} \bigg| \frac{ \Gamma(-\sigma, t)}{\Gamma(1-\sigma)} - \frac{\alpha+n}{C} \Gamma\bigg(0, \frac{\alpha+n}{C}t \bigg) \bigg| \ddr t.
\end{align*}
From this we claim that, in probability, 
\begin{equation}
\label{eq:limit_W*_aux}
\lim_{n \to + \infty} \W_{*}\bigg( \frac{C^{1-\sigma}}{\Gamma(1-\sigma)} \frac{e^{-Cs}}{s^{1+\sigma}} \ddr s,(\alpha+ n)\frac{e^{-(\alpha+ n)s}}{s} \ddr s\bigg)  = \int_0^{+\infty} \bigg| \frac{ \Gamma(-\sigma, t)}{\Gamma(1-\sigma)} - \Gamma(0, t) \bigg| \ddr t, 
\end{equation}
which can be proved by using the continuous mapping theorem as $(\alpha+n)/C$ converges in probability to $1$. Moreover, each $\W_{*}$ is bounded by $2$ as we have scaled Lévy intensities On the other hand from the notation $r_n$ coming from Theorem 18, the $L^1$ convergence implies that in probability 
\begin{equation*}
\lim_{n \to + \infty} \frac{n}{r_n}  \frac{\alpha (1+U)^ \sigma }{C} = \alpha, 
\end{equation*}
and the term is uniformly integrable as $n r_n^{-1}  \alpha (1+U)^ \sigma C^{-1} \leq (r_n^{-1} \alpha (1+U)^ \sigma )  n/(n-k \sigma)$. Putting the two pieces together, we have convergence in probability
\begin{equation*}
\lim_{n \to + \infty}  \frac{n}{r_n}  \mathcal{J}_1^U = \alpha \int_0^{+\infty} \bigg| \frac{ \Gamma(-\sigma, t)}{\Gamma(1-\sigma)} - \Gamma(0, t) \bigg| \ddr t > 0,
\end{equation*}
together with uniform integrability. Thus by taking expectation and given the expression of $r_n$ in Theorem 18 we have 
\begin{equation*}
\E \left( \mathcal{J}_1^U \right)\asymp \max \left( \frac{1}{n^{1/(1+\sigma)}}, \frac{k}{n} \right).
\end{equation*}

\emph{Second term}. We recall that
\begin{equation*}
\mathcal{J}_2^U = \frac{n-k \sigma}{C} \W_{*} \bigg( C \frac{e^{-Cs}}{s} \ddr s, (\alpha+ n)\frac{e^{-(\alpha+ n)s}}{s} \ddr s\bigg),
\end{equation*}
which is bounded by $2$ by Equation (11).
The prefactor satisfies $(n-k\sigma)/C \leq 1$.
From Proposition 14 the extended Wasserstein distance satisfies
\begin{equation*}
\W_{*} \bigg( C \frac{e^{-Cs}}{s} \ddr s, (\alpha+ n)\frac{e^{-(\alpha+ n)s}}{s} \ddr s\bigg) \sim c \left|  \frac{C}{\alpha +n} - 1 \right|, 
\end{equation*}
which again is bounded by $2$. The asymptotic analysis of Theorem 18 guarantees that, at least in probability,
\begin{equation*}
\frac{C}{\alpha +n} - 1 = \frac{\alpha (1+U)^\sigma/n - k/n \sigma - \alpha/n }{1 + \alpha/n} \lesssim \max \left( \frac{1}{n^{1/(1+\sigma)}}, \frac{k}{n} \right).
\end{equation*}
Thus the same asymptotic holds in probability for $\mathcal{J}_2^U$ thus we can deduce that it also hold in expectation because of the boundedness condition.  

\emph{Last term}. We recall that
\begin{equation*}
\mathcal{A}^U = \W_{\dXt} \bigg(\frac{\alpha ( U +1)^\sigma}{C} P_0 + \frac{1}{C}\sum_{i=1}^k (n_i - \sigma) \delta_{x_i^*}, \frac{\alpha}{\alpha + n} P_0 + \frac{1}{\alpha+n} \sum_{i=1}^k n_i \delta_{x_i^*} \bigg).
\end{equation*}
We apply the convexity property~\eqref{eq:convex} of the Wasserstein distance to obtain
\[
\mathcal{A}^U \le \lambda_1 \W_{\dXt} \bigg(P_0 , \frac{1}{n} \sum_{i=1}^k n_i \delta_{x_i^*} \bigg) + \lambda_2 \W_{\dXt} \bigg(\frac{1}{n-k \sigma} \sum_{i=1}^k (n_i- \sigma) \delta_{x_i^*}, \frac{1}{n}\sum_{i=1}^k n_i \delta_{x_i^*} \bigg),
\]
where
\begin{equation*}
\lambda_1 = \frac{\alpha(1+U)^\sigma}{C} - \frac{\alpha}{\alpha+n}, \qquad \lambda_2 = 1 - \frac{\alpha (1+U)^\sigma}{C}.
\end{equation*}
From the previous discussion we know that $\E (\lambda_1) \lesssim \max(n^{-1/(1+\sigma)}, k/n )$ while $\E (\lambda_2) \sim 1$. As $\W_{\dXt}$ is bounded by $2$, the first term of the sum is bounded by $\max(n^{-1/(1+\sigma)}, k/n )$ in expectation.  We then use one last time the convexity inequality~\eqref{eq:convex} by noting that 
\begin{equation*}
\frac{1}{n}\sum_{i=1}^k n_i \delta_{X_i^*} = \frac{n - k \sigma}{n} \frac{1}{n - k \sigma}\sum_{i=1}^k (n_i - \sigma) \delta_{x_i^*} + \frac{k \sigma}{n} \frac{1}{k} \sum_{i=1}^k  \delta_{x_i^*}
\end{equation*}
which yields
\begin{equation*}
\W_{\dXt} \bigg(\frac{1}{n-k \sigma} \sum_{i=1}^k (n_i- \sigma) \delta_{X_i^*}, \frac{1}{n}\sum_{i=1}^k n_i \delta_{x_i^*} \bigg) 
 \leq \frac{k \sigma}{n} \W_{\dXt} \left( \frac{1}{n - k \sigma}\sum_{i=1}^k (n_i - \sigma) \delta_{x_i^*}, \frac{1}{k} \sum_{i=1}^k  \delta_{x_i^*} \right). 
\end{equation*}
As $\W_{\dXt}$ is bounded by $2$, we conclude that this term is bounded by $2k \sigma/n$. Summing up,
\begin{equation*}
\E (\mathcal{A}^U) \lesssim \max \left( \frac{1}{n^{1/(1+\sigma)}}, \frac{k}{n} \right) + \frac{k}{n} \lesssim \max \left( \frac{1}{n^{1/(1+\sigma)}}, \frac{k}{n} \right). \qedhere
\end{equation*}
\end{proof}

\begin{proof}[Proof of Proposition 20]
Since $P$ is continuous $k = n$ almost surely. We consider a realization of $X_n$ with $k=n$, and thus the only source of randomness now comes from $U$ and $C$. With a similar reasoning to Proposition 19, we will show
\begin{equation*}
\E(\mathcal{J}_1^U) \sim \sigma \int_0^{+\infty} \bigg| \frac{ \Gamma(-\sigma, t)}{\Gamma(1-\sigma)} - \Gamma(0, t) \bigg| \ddr t, \qquad
\E(\mathcal{J}_2^U) \ll 1, \qquad
 \E(\mathcal{A}^U) \sim  \sigma \W_{\dXt}(P_0,P).
\end{equation*}
First of all we observe that from Theorem 18 and the definition of $C$,
\begin{equation*}
\lim_{n \to + \infty} \frac{\alpha (1+U)^\sigma}{C} = \sigma, \quad
\lim_{n \to + \infty} \frac{n - k \sigma}{C} = 1 - \sigma, \quad
\lim_{n \to + \infty} \frac{C}{n} = 1, 
\end{equation*}
and all the limits hold in $L^1$. 

\emph{First term}. We recall that
\begin{equation*}
\mathcal{J}_1^U = \frac{\alpha (1+U)^ \sigma }{C} \W_*\bigg( \frac{C^{1-\sigma}}{\Gamma(1-\sigma)} \frac{e^{-Cs}}{s^{1+\sigma}} \ddr s,(\alpha+ n)\frac{e^{-(\alpha+ n)s}}{s} \ddr s\bigg).
\end{equation*}
The prefactor converges in $L^1$ to $\sigma$ while staying bounded by $1$. 
From~\eqref{eq:limit_W*_aux} and using that $\W_*$ is always bounded by $2$, by dominated convergence we have
\begin{equation*}
\lim_{n \to + \infty} \E \left( \W_{*}\bigg( \frac{C^{1-\sigma}}{\Gamma(1-\sigma)} \frac{e^{-Cs}}{s^{1+\sigma}} \ddr s,(\alpha+ n)\frac{e^{-(\alpha+ n)s}}{s} \ddr s\bigg) \right) = \int_0^{+\infty} \bigg| \frac{ \Gamma(-\sigma, t)}{\Gamma(1-\sigma)} - \Gamma(0, t) \bigg| \ddr t,
\end{equation*}
The asymptotics for the first term follows.

\emph{Second term}. We recall that
\begin{equation*}
\mathcal{J}_2^U = \frac{n-k \sigma}{C} \W_{*}\bigg( C \frac{e^{-Cs}}{s} \ddr s, (\alpha+ n)\frac{e^{-(\alpha+ n)s}}{s} \ddr s\bigg).
\end{equation*}
The prefactor stays bounded by $1$. For the extended Wasserstein distance we use the bound of Proposition 14 to achieve
\begin{equation*}
\W_{*}\bigg( C \frac{e^{-Cs}}{s} \ddr s, (\alpha+ n)\frac{e^{-(\alpha+ n)s}}{s} \ddr s\bigg) \lesssim  \log \bigg(\left| \frac{C}{\alpha+n} \right| \bigg). 
\end{equation*}
As $C/(\alpha + n)$ converges in $L^1$ to $1$, while the logarithm is sublinear at $+ \infty$, we easily conclude that $\E(\mathcal{J}_2^U)$ converges to $\log(1) = 0$.

\emph{Last term}. We recall that
\begin{equation*}
\mathcal{A}^U = \W_{\dXt} \bigg(\frac{\alpha ( U +1)^\sigma}{C} P_0 + \frac{(1 - \sigma)}{C}\sum_{i=1}^n  \delta_{x_i}, \frac{\alpha}{\alpha + n} P_0 + \frac{1}{\alpha+n} \sum_{i=1}^n \delta_{x_i} \bigg),
\end{equation*}
where we have used that here the sequences $(x_n)_{n \geq 1}$ and $(x_k^*)_{k \geq 1}$ coincide. The empirical distribution of $(x_n)_{n \geq 1}$ converges to $P$ for the topology of weak convergence. Thus, 
\begin{equation*}
\frac{\alpha ( U +1)^\sigma}{C} P_0 + \frac{(1 - \sigma)}{C}\sum_{i=1}^n  \delta_{x_i}  \rightarrow \sigma P_0 + (1-\sigma ) P; \qquad \frac{\alpha}{\alpha + n} P_0 + \frac{1}{\alpha+n} \sum_{i=1}^n \delta_{x_i} \rightarrow P.
\end{equation*}
Since the Wasserstein distance $\W_{\dXt}$ metrizes this convergence, we conclude that 
\begin{equation*}
\E(\mathcal{A}^U) \rightarrow \W_{\dXt}(\sigma P_0 + (1- \sigma) P, P) = \sigma \W_{\dXt}(P_0, P),
\end{equation*}
where the last equality follows by~\eqref{eq:convex_specific}. 

\emph{Proving that the limit is increasing}. Eventually we justify that the limit is increasing in $\sigma$. To that end, we only need to prove that the function
\begin{equation*}
G : \sigma \mapsto \int_0^{+\infty} \bigg| \frac{ \Gamma(-\sigma, t)}{\Gamma(1-\sigma)} - \Gamma(0, t) \bigg| \ddr t
\end{equation*}
is increasing in $\sigma$. We check that the assumptions of Proposition~\ref{prop:tail_integral_comp} apply with $\theta = \sigma$ and $U(\theta,t) = \frac{ \Gamma(-\sigma, t)}{\Gamma(1-\sigma)}$.
We first claim that there exists a unique point $t_\sigma$ where the integrand in $G$ vanishes, and that the integrand is positive before $t_\sigma$ and negative after. Indeed, let us introduce
\begin{equation*}
f_\sigma(t) = \Gamma(0, t) - \frac{\Gamma(-\sigma, t)}{\Gamma(1 - \sigma)}.
\end{equation*}
It is easy to check that $\lim_{t \to 0} f_\sigma(t) = - \infty$ while $\lim_{t \to + \infty} f_\sigma(t) = 0$. Moreover, since
\begin{equation*}
f'_\sigma(t) = - \frac{e^{-t}}{t} + \frac{e^{-t}}{t^{1+\sigma} \Gamma(1-\sigma)},
\end{equation*}
we see that $f_\sigma$ is first increasing and then decreasing. Thus it admits a unique zero, and it is positive before its zero and negative after. Then we study the derivative by introducing
\begin{equation*}
D(\sigma,t) = \frac{\dr }{\dr \sigma}\left[ \frac{\Gamma(-\sigma, t)}{\Gamma(1 - \sigma)} \right].
\end{equation*}
We recall that the derivative of the Gamma incomplete function satisfies
\begin{equation*}
\frac{\dr \Gamma(z,t) }{\dr z} = - \int_t^{+ \infty} \log (x) \, x^{1-z} e^{-x} \, \ddr x.
\end{equation*} 
A straightforward computation then leads to
\begin{equation*}
D(\sigma,t) = \frac{1}{\Gamma(1-\sigma)} \left\{ \frac{\Gamma'(1-\sigma)}{\Gamma(1-\sigma)} \Gamma(- \sigma, t) -  \int_{t}^{+ \infty} \log (s) \frac{e^{-s}}{s^{1+\sigma}} \, \ddr s  \right\}.
\end{equation*}
From this expression we see that
\begin{equation*}
\lim_{t \to 0} D(\sigma,t) = + \infty, \qquad \lim_{t \to + \infty} D(\sigma,t) = 0. 
\end{equation*}
Moreover, $D(\sigma, \cdot)$ is decreasing until $a_\sigma = \exp ( \Gamma'(1-\sigma) / \Gamma(1-\sigma))  < 1$ and increasing after. Thus it vanishes at a single point $s_\sigma \in (0,a_\sigma)$, it is positive on $(0,s_\sigma)$ and negative after. Next, we need to check the integrability Assumption~\ref{lab:integrability}. Fix $\Sigma \in (0,1)$. Notice first that $\Gamma(1-\sigma)$ and $\Gamma'(1-\sigma)$ are uniformly bounded for $\sigma \in [0,\Sigma]$. Up to a constant $c_\Sigma$ which depends only on $\Sigma$ we deduce that for $\sigma \in [0,\Sigma]$, $|D(\sigma,t)|$ is bounded from above by 
\begin{equation*}
c_\Sigma \left( |\Gamma(-\sigma,t)| + \int_{t}^{+ \infty} |\log (s)| \frac{e^{-s}}{s^{1+\sigma}} \, \ddr s \right) \leq c_\Sigma  \int_{t}^{+ \infty}  (1+ |\log (s)|) \frac{e^{-s}}{\min(s, s^{1+\Sigma})} \, \ddr s .
\end{equation*}
We integrate the expression above in $t$ and, by Fubini's theorem, we find that
\begin{align*}
\int_0^{+ \infty} \sup_{\sigma \in [0,\Sigma]} \left| D(t,\sigma) \right| \ddr t  &\leq c_\Sigma \int_0^{+\infty} \left( \int_{t}^{+ \infty}  (1+ |\log (s)|) \frac{e^{-s}}{\min \left(s, s^{1+\Sigma} \right)} \, \ddr s \right) \ddr t \\
 &= c_\Sigma \int_0^{+\infty} (1+|\log (s)|) \max \left(1,s^{-\Sigma} \right) e^{-s} \, \ddr s < + \infty. 
\end{align*}
We checked all assumption of Proposition~\ref{prop:tail_integral_comp}, thus we conclude that $G$ is increasing in $\sigma$.
\end{proof}

\bibliography{bibmerging}